\documentclass[12pt]{amsart}
\usepackage[colorlinks,
            linkcolor=black,
            anchorcolor=blue,
            citecolor=black]{hyperref}
\usepackage{graphicx}
\usepackage{amsmath}
\usepackage{amssymb}
\usepackage{setspace}
\usepackage{amsthm}
\newtheorem{thm}{Theorem}[section]

\newtheorem{lem}[thm]{Lemma}

\theoremstyle{definition}

\theoremstyle{remark}
\newtheorem{rem}[thm]{Remark}
\theoremstyle{conclusion}

\theoremstyle{conjecture}

\numberwithin{equation}{section}

\newcommand{\lr}{\left(}
\newcommand{\rr}{\right)}

\newcommand{\be}{\begin{equation}}
\newcommand{\ee}{\end{equation}}

\newcommand{\De}{\Delta}

\newcommand{\la}{\lambda}

\newcommand{\Ns}{\frac{N}{2}+s}

\newcommand{\ra}{{\mbox{$\rightarrow$}}}

\begin{document}
\title[Direct methods for pseudo-relativistic Schr\"{o}dinger operators]{Direct methods for pseudo-relativistic Schr\"{o}dinger operators}

\author{Wei Dai$^\dag$, Guolin Qin$^\ddag$, Dan Wu$^*$}

\address{$^\dag$ School of Mathematical Sciences, Beihang University (BUAA), Beijing 100083, P. R. China}
\email{weidai@buaa.edu.cn}

\address{$^\ddag$ Institute of Applied Mathematics, Chinese Academy of Sciences, Beijing 100190, and University of Chinese Academy of Sciences, Beijing 100049, P. R. China}
\email{qinguolin18@mails.ucas.ac.cn}

\address{$^*$ College of Mathematics, Hunan University, Changsha 410082, P. R. China}
\email{danwu@hnu.edu.cn}

\thanks{$^*$ Corresponding author: Dan Wu at danwu@hnu.edu.cn. \\ Wei Dai is supported by the NNSF of China (No. 11971049) and the Fundamental Research Funds for the Central Universities. Dan Wu was supported by the NSFC (Grant No. 11901183), NSF of Hunan Province (No. 2017JJ3028) and Young Teachers Program of Hunan University (531118040104).}

\begin{abstract}
In this paper, we establish various maximal principles and develop the direct moving planes and sliding methods for equations involving the physically interesting (nonlocal) pseudo-relativistic Schr\"{o}dinger operators $(-\Delta+m^{2})^{s}$ with $s\in(0,1)$ and mass $m>0$. As a consequence, we also derive multiple applications of these direct methods. For instance, we prove monotonicity, symmetry and uniqueness results for solutions to various equations involving the operators $(-\Delta+m^{2})^{s}$ in bounded domains, epigraph or $\mathbb{R}^{N}$, including pseudo-relativistic Schr\"odinger equations, 3D boson star equations and the equations with De Giorgi type nonlinearities.
\end{abstract}
\maketitle {\small {\bf Keywords:} Pseudo-relativistic Schr\"{o}dinger operators; Inhomogeneous fractional Laplacians; 3D boson star equations; Direct methods of moving planes; Direct sliding methods.\\

{\bf 2010 MSC} Primary: 35R11; Secondary: 35B06, 35B53.}

\section{Introduction}
\subsection{Background and setting of the problem}
The pseudo-relativistic Schr\"odinger equations
\begin{equation}\label{NLS_pseudo}
i \frac{\partial \Psi}{\partial t}=\sqrt{-c^{2} \Delta+m^{2} c^{4}} \Psi-m c^{2} \Psi
-f\lr|\Psi(t, x)|^{2}\rr\Psi, \quad \forall \,\, (t,x)\in\mathbb{R}_{+} \times \mathbb{R}^N,
\end{equation}
in which $c$ denotes the speed of the light, $m>0$ represents the particle mass, $\Psi(t,x)$ is a complex-valued wave field and $f:[0, \infty) \rightarrow \mathbb{R}$ is a nonlinear function, describes, from the physical viewpoint, the dynamics of systems consisting of identical spin-$0$ bosons whose motions are relativistic (see \cite{LT,LY}), such as boson stars. We refer to \cite{A,CS1,DSS,ES,FJL,FJL1,FL,Le1,Le,LL,LT} for the rigorous derivation of the equations and the study of their dynamical properties.

We say \eqref{NLS_pseudo} is pseudo-relativistic Schr\"odinger equations in the sense that, when passing to the limit as $c\rightarrow+\infty$ in \eqref{NLS_pseudo}, we derive the more familiar nonlinear Schr\"odinger equations
\begin{equation}\label{NLS}
i \frac{\partial \Psi}{\partial t}=-\frac{1}{2m}\De\Psi
-f\lr|\Psi(t,x)|^{2}\rr\Psi \quad \text { in } \mathbb{R}_{+} \times \mathbb{R}^N
\end{equation}
as the non-relativistic limit of equations \eqref{NLS_pseudo} (see e.g. \cite{CS1}).

In particular, if we consider the standing waves solution $\Psi(t, x)=e^{i \omega t}u(x)$ to \eqref{NLS_pseudo}, where $\omega \in \mathbb{R}$ is the frequency and $u(x)$ is real-valued, then equation \eqref{NLS_pseudo} becomes a static semi-linear elliptic problems:
\begin{equation}\label{NLS_sta}
\sqrt{-c^{2}\Delta+m^{2} c^{4}} u(x)+(\omega-m c^{2}) u(x)=f(|u(x)|^{2})u(x), \quad \forall \,\, x\in\mathbb{R}^{N},
\end{equation}
which is the relativistic version of the limit equations:
\begin{equation}\label{NLS_staL}
-\frac{1}{2m}\Delta u(x)+\omega u(x)=f(|u(x)|^{2})u(x), \quad \forall \,\, x\in\mathbb{R}^{N}.
\end{equation}

We will also consider the 3D pseudo-relativistic Hartree equation:
\begin{equation}\label{H-pseudo}
  i \frac{\partial \Psi}{\partial t}=\sqrt{-c^{2} \Delta+m^{2} c^{4}} \Psi-m c^{2} \Psi-\left(\frac{1}{|x|}\ast|\Psi|^{2}\right)\Psi \quad \text{in} \,\, \mathbb{R}_{+}\times \mathbb{R}^{3},
\end{equation}
where $\Psi(t,x)$ is a complex-valued wave field, and the symbol $\ast$ stands for convolution on $\mathbb{R}^{3}$. The operator $\sqrt{-c^{2} \Delta+m^{2} c^{4}}-mc^{2}$ is the kinetic energy operator of a relativistic particle with mass $m>0$, and the convolution kernel $\frac{1}{|x|}$ represents the Newtonian gravitational potential in appropriate physical units. From the physical viewpoint, pseudo-relativistic Hartree equation \eqref{H-pseudo} arises as an effective dynamical description for an $N$-body quantum system of relativistic bosons with two-body interaction given by Newtonian gravity (see \cite{ES}, see also \cite{FJL,FL}). Such a system is a model system for a pseudo-relativistic boson star.

Pseudo-relativistic Hartree equation \eqref{H-pseudo} is the relativistic version of the limit equation (as $c\rightarrow+\infty$):
\begin{equation}\label{H-pseudoL}
  i \frac{\partial \Psi}{\partial t}=-\frac{1}{2m}\Delta\Psi-\left(\frac{1}{|x|}\ast|\Psi|^{2}\right)\Psi \quad \text{in} \,\, \mathbb{R}_{+}\times \mathbb{R}^{3}.
\end{equation}
By inserting the solitary wave solution $\Psi(t,x)=e^{i \omega t}u(x)$ into \eqref{H-pseudo}, we derive the following static semi-linear scalar field equation:
\begin{equation}\label{H-pseudo-sta}
  \sqrt{-c^{2}\Delta+m^{2} c^{4}} u(x)+(\omega-m c^{2}) u(x)=\left(\frac{1}{|x|}\ast|u|^{2}\right)u(x), \quad \forall \,\, x\in\mathbb{R}^{3},
\end{equation}
whose non-relativistic limit equation is:
\begin{equation}\label{H-pseudo-staL}
 -\frac{1}{2m}\Delta u(x)+\omega u(x)=\left(\frac{1}{|x|}\ast|u|^{2}\right)u(x), \quad \forall \,\, x\in\mathbb{R}^{3}.
\end{equation}

One should note that, solutions to both classical and fractional nonlinear Schr\"odinger equations \eqref{NLS} or Hartree equations \eqref{H-pseudoL} enjoy some kind of scaling invariance if we consider the power type nonlinearities $f(t)=t^{\frac{p-1}{2}}$ or general Hartree type nonlinearities. However, since the differential operators $\sqrt{-c^{2}\Delta+m^{2} c^{4}}$ are inhomogeneous, no such scaling invariance works for pseudo-relativistic Schr\"odinger equations \eqref{NLS_pseudo} and Hartree equations \eqref{H-pseudo} with any $c,m\in(0,+\infty)$.

\medskip

In this paper, we are mainly concerned with the normalized (i.e., $c=1$) general static pseudo-relativistic Schr\"odinger equations \eqref{NLS_pseudo} and Hartree equations \eqref{H-pseudo} ($s=\frac{1}{2}\rightarrow s\in(0,1)$):
\begin{equation}\label{NLS_staN}
\left(-\Delta+m^{2}\right)^{s}u(x)+(\omega-m)u(x)=f(|u(x)|^{2})u(x), \quad \forall \,\, x\in\mathbb{R}^{N},
\end{equation}
and
\begin{equation}\label{H-pseudo-staN}
  \left(-\Delta+m^{2}\right)^{s}u(x)+(\omega-m)u(x)=\left(\frac{1}{|x|}\ast|u|^{2}\right)u(x), \quad \forall \,\, x\in\mathbb{R}^{3},
\end{equation}
please see Theorem \ref{T:sub_symm} and Theorem \ref{RSO-HC-Symm}. Furthermore, by developing direct moving planes and sliding methods, we will investigate the following generalized equations involving general pseudo-relativistic operators $(-\Delta+m^{2})^{s}$ ($0<s<1$):
\begin{equation}\label{gPDE}
  (-\Delta+m^{2})^{s}u(x)=f(x,u(x))
\end{equation}
with various kinds of nonlinearities (including De Giorgi type) and in various types of regions, such as bounded domains $D$, epigraph $\Omega$ and whole-space $\mathbb{R}^{N}$ (see Section 2 and 3).

\medskip

The general pseudo-relativistic operators $(-\Delta+m^{2})^{s}$ with $0<s<1$ is defined in nonlocal way by (refer to e.g. \cite{A,CMS,FF1,FF2,H})
\begin{equation}\label{Definition}
\left(-\Delta+m^{2}\right)^{s}u(x):=c_{N,s} m^{\Ns} P.V. \int_{\mathbb{R}^{N}} \frac{u(x)-u(y)}{|x-y|^{\Ns}}K_{\Ns}(m|x-y|)d y+m^{2 s}u(x)
\end{equation}
for every $x \in \mathbb{R}^{N}$, where $P.V.$ stands for Cauchy principal value, and
\begin{equation}
c_{N, s}=2^{1-\frac{N}{2}+s}\pi^{-\frac{N}{2}}\frac{s(1-s)}{\Gamma(2-s)}.
\end{equation}
Here $K_{\nu}$ denotes the modified Bessel function of the second kind with order $\nu$ (see \cite{EMOT,FF1,FF2}), which solves the equation
\begin{equation}
r^{2} K_{\nu}^{\prime \prime}+r K_{\nu}^{\prime}-\left(r^{2}+\nu^{2}\right) K_{\nu}=0.
\end{equation}
Note that $K_{\nu}(r)$ has the integral representation
$$
K_{\nu}(r)=\int_{0}^{\infty} e^{-r \cosh t} \cosh (\nu t) dt.
$$
It is easy to verify that $K_{\nu}(r)$ is a real and positive function satisfying $K'_{\nu}(r)<0$ for all $r>0$, and $K_{\nu}=K_{-\nu}$ for $\nu<0$. We have (see \cite{EMOT,FF1,FF2}), for $\nu>0$,
\begin{equation}\label{B-asymp0}
  K_{\nu}(r) \sim \frac{\Gamma(\nu)}{2}\left(\frac{r}{2}\right)^{-\nu},
\end{equation}
as $r \rightarrow 0$. Hence there exists a small $r_0>0$ and two constants $C_0>c_0>0$ such that
\begin{equation}\label{Bessel_asy0}
\frac{c_0}{r^\nu}\leq K_\nu(r)\leq\frac{C_0}{r^\nu}
\end{equation}
for all $r\leq r_0$. We also have, for $\nu>0$,
\begin{equation}\label{B-asymp1}
  K_{\nu}(r) \sim \frac{\sqrt{\pi}}{\sqrt{2}} r^{-\frac12} e^{-r},
\end{equation}
as $r \rightarrow\infty$. Hence there exists a large $R_\infty>0$ and two constants $C_\infty>c_\infty>0$ such that
\begin{equation}\label{Bessel_asy1}
\frac{c_\infty}{r^{\frac12}e^r}\leq K_\nu(r)\leq\frac{C_\infty}{r^{\frac12}e^r}
\end{equation}
for all $r\geq R_\infty$. Let
\begin{equation}\label{space}
\mathcal{L}_s(\mathbb{R}^{N}):=\left\{u : \mathbb{R}^{N} \rightarrow \mathbb{R} \,\Bigg|\, \int_{\mathbb{R}^{N}} \frac{e^{-|x|}|u(x)|}{1+|x|^{\frac{N+1}{2}+s}}dx<+\infty\right\}.
\end{equation}
Then, by asymptotic properties of $K_{\nu}(r)$ as $r\rightarrow0$ and $+\infty$, one can easily verify that for any $u \in \mathcal{L}_s(\mathbb{R}^{N})\cap C_{\text {loc }}^{1,1}(\mathbb{R}^{N})$, the integral on the right hand side of the definition \eqref{Definition} is well-defined. Hence $(-\Delta+m^{2})^{s}u$ makes sense for all functions $u \in \mathcal{L}_s(\mathbb{R}^{N})\cap C_{\text {loc }}^{1,1}(\mathbb{R}^{N})$. The general pseudo-relativistic operators $(-\Delta+m^{2})^{s}$ ($0<s<1$) can also be expressed equivalently via extension method, please see e.g. \cite{FF1,FF2} and the references therein.

When $m\rightarrow0+$, the general pseudo-relativistic operators $(-\Delta+m^{2})^{s}$ degenerates into the familiar fractional Laplacian $(-\Delta)^{s}$, which is also a nonlocal pseudo-differential operator given by (see e.g. \cite{CG,CLL,CLM,CLZ,CQ,DQ})
\begin{eqnarray}\label{def-fracL}
(-\Delta)^s u(x) = C_{N, s} \, P.V. \int_{\mathbb{R}^N} \frac{u(x)-u(y)}{|x-y|^{N+2s}} dy.
\end{eqnarray}
Fractional Laplacian $(-\Delta)^{s}$ ($0<s<1$) is well-defined for any $u\in C^{1,1}_{loc}(\mathbb{R}^{N})\cap\dot{L}_{s}(\mathbb{R}^{N})$ with the function spaces
\[\dot{L}_{s}(\mathbb{R}^{N}):=\left\{ u: \mathbb{R}^{N} \rightarrow \mathbb{R} \,\Big|\, \int_{\mathbb{R}^{N}} \frac{|u(x)|}{1 + |x|^{N+2s}}dx <+\infty \right\}.\]
It can also be defined equivalently through Caffarelli and Silvestre's extension method (refer to \cite{CS}, see also \cite{BCPS,CT,S}).

\medskip

In recent years, fractional order operators have attracted more and more attentions. Besides various applications in fluid mechanics, molecular dynamics, relativistic quantum mechanics of stars (see e.g. \cite{CV,Co}) and conformal geometry (see e.g. \cite{CG}), it also has many applications in probability and finance (see \cite{Be,CT}). The fractional Laplacians can be understood as the infinitesimal generator of a stable L\'{e}vy diffusion process (see \cite{Be}). The general pseudo-relativistic operators with singular potentials describes a spin zero relativistic particle of charge $e$ and mass $m$ in the Coulomb field of an infinitely heavy nucleus of charge $Z$ (see e.g. \cite{FF1,FF2,H,Lieb}).

\smallskip

However, the non-locality virtue of $(-\Delta+m^{2})^{s}$ ($m\geq0$) makes it difficult to investigate. To overcome this difficulty, we basically have two approaches. One way is to define $(-\Delta+m^{2})^{s}$ ($m\geq0$) via extension method, so as to reduce the nonlocal problem into a local one in higher dimensions. Another approach is to derive the integral representation formulae of solutions (see \cite{CLO,CLM}). After establishing the equivalence between the fractional order equation and its corresponding integral equation (with Bessel kernel if $m>0$, Riesz kernel if $m=0$), one can apply the method of moving planes (spheres) in integral forms, the method of scaling spheres or regularity lifting methods to study various properties of solutions to the fractional order equations involving operators $(-\Delta+m^{2})^{s}$. These two methods have been applied successfully to study equations involving nonlocal fractional operators, and a series of fruitful results have been derived (see \cite{BCPS,CLO,CS,CT,DFHQW,DL,DLQ,DQ0,DQ3,FLS,MZ,S} and the references therein).

\smallskip

Nevertheless, when using the above two approaches, we often need to impose some additional conditions on the solutions, which would not be necessary if we consider the pseudo-differential equation directly. Moreover, these two approaches do not work for fully nonlinear nonlocal operators, for instance, the fractional $p$-Laplacians (see e.g. \cite{CQ} for more details).

\smallskip

Therefore, it is desirable for us to develop direct methods without going through extension methods or integral representation formulae. Direct moving planes (spheres) method and sliding method have been introduced for fractional Laplacian $(-\Delta)^{s}$ ($s\in(0,1)$) in \cite{CLL,CW2,DSV} and for fractional $p$-Laplacians $(-\Delta)^{s}_{p}$ in \cite{CL2,CLiu,CW1,CW3}, and have been applied to obtain symmetry, monotonicity and uniqueness of solutions to various equations involving $(-\Delta)^{s}$ or $(-\Delta)^{s}_{p}$. The goal of this paper is to establish \emph{direct methods} for the general pseudo-relativistic Schr\"{o}dinger operators $(-\Delta+m^{2})^{s}$ and derive \emph{various applications}.

\subsection{Main results}
In this paper, inspired by the direct methods for $(-\Delta)^{s}$ and $(-\Delta)^{s}_{p}$ established in \cite{CL2,CLiu,CLL,CW1,CW2,CW3,DSV}, we will introduce \emph{direct moving planes} and \emph{sliding methods} for the pseudo-relativistic Schr\"{o}dinger operators $(-\Delta+m^{2})^{s}$ with general $s\in(0,1)$ and $m>0$.

\smallskip

Being essentially different from the homogeneous fractional Laplacians $(-\Delta)^{s}$ ($s\in(0,1)$), since there is a \emph{modified Bessel function of the second kind $K_{\frac{N}{2}+s}$} in the definition \eqref{Definition} of inhomogeneous fractional operators $(-\Delta+m^{2})^{s}$, it \emph{does not possess} any \emph{invariance properties} under \emph{Kelvin-type or scaling transforms}. In order to circumvent these difficulties, we need to implement some new arguments and careful analysis (see Section 2 and 3, especially, the proofs of \emph{Theorem \ref{MP_anti_ubdd}, Lemma \ref{lem-MP}, Theorem \ref{MP-Ubdd} and Theorem \ref{32NRP}}). In particular, one should note that, being different from the \emph{maximum principles (in unbounded open sets $D$)} for fractional Laplacians $(-\Delta)^{s}$ or fractional $p$-Laplacians $(-\Delta)^{s}_{p}$ established in Dipierro, Soave and Valdinoci \cite{DSV}, Chen and Liu \cite{CLiu} and Chen and Wu \cite{CW2} (some \emph{exterior conditions on $D$} are necessary therein), \emph{we do not need to impose any assumption on $D$ in the maximum principle in our Theorem \ref{MP-Ubdd}, i.e., $D\subseteq\mathbb{R}^{N}$ can be any open set (possibly unbounded and disconnected)}. The assumption \eqref{32NRP-con} on $D$ in Theorem \ref{32NRP} is also much weaker than the exterior conditions on $D$ assumed in \cite{CLiu,CW2,DSV}.

\medskip

The main contents and results in our paper are arranged as follows.

\smallskip

In Section 2, we will establish various maximum principles for anti-symmetric functions and develop the direct method of moving planes for the inhomogeneous fractional operators $(-\Delta+m^{2})^{s}$. As applications, symmetry, monotonicity and uniqueness results for solutions to various equations involving the general pseudo-relativistic Schr\"{o}dinger operators were derived. For more literatures on the methods of moving planes (spheres), please refer to \cite{CD,CDQ,CGS,CL,CL1,CL0,CL2,CLL,CLO,CLZ,CW3,CY,DFHQW,DFQ,DL,DLQ,DQ,GNN1,JLX,LiC,Li,Lin,LZ,LZ1,MZ,Pa,Serrin,WX} and the references therein. For the method of scaling spheres, please see \cite{DQ0,DQ3} and the references therein.

\smallskip

Section 3 is devoted to proving various maximum principles in both bounded or unbounded open sets and developing the direct sliding methods for the general pseudo-relativistic Schr\"{o}dinger operators $(-\Delta+m^{2})^{s}$. As applications, monotonicity and uniqueness results for solutions to various equations involving the general pseudo-relativistic Schr\"{o}dinger operators were derived. The sliding method was developed by Berestycki and Nirenberg (\cite{BN1,BN2,BN3}). It was used to establish qualitative properties of solutions for PDEs (mainly involving the regular Laplacian $-\Delta$), such as symmetry, monotonicity and uniqueness $\cdots$. The key ingredients are different forms of maximum principles. The main idea is to compare values of solution to the equation at two different points, between which one point is obtained from the other by sliding the region in a fixed direction, and then slide the region back to a critical position. While in the method of moving planes, one point is the reflection of the other. For more literatures on the sliding methods for $-\Delta$, $(-\Delta)^{s}$ or $(-\Delta)^{s}_{p}$, please refer to \cite{BCN1,BCN2,BHM,BN1,BN2,BN3,CLiu,CW1,CW2,DSV}.

\begin{rem}\label{rem18}
By using similar ideas and arguments, one can also develop the direct moving planes and sliding methods for the following general fully nonlinear nonlocal operators:
\begin{equation}\label{fno}
  F_{s,m}(u)(x):=c_{N,s} m^{\Ns}P.V.\int_{\mathbb{R}^{N}} \frac{G(u(x)-u(y))}{|x-y|^{\Ns}}K_{\Ns}(m|x-y|)d y+m^{2s}u(x),
\end{equation}
where $G$ is a local Lipschitz continuous function satisfying $G(0)=0$ and $u$ belongs to some appropriate function space. This kind of operators were introduced by Caffarelli and Silvestre in \cite{CS}. When $G(t)=|t|^{p-2}t$ and $m\rightarrow0$, $F_{s,m}\rightarrow F_{s,0}=(-\Delta)^{s}_{p}$. If $G(t)=|t|^{p-2}t$ and $m>0$, we denote $F_{s,m}:=(-\Delta+m^{2})^{s}_{p}$. It is clear that, when $G(t)=t$ and $m>0$, $F_{s,m}$ degenerates into the general pseudo-relativistic Schr\"{o}dinger operators $(-\Delta+m^{2})^{s}$. We leave these open problems to interested readers.
\end{rem}

In what follows, we will use $C$ to denote a general positive constant that may depend on $N$, $m$ and $s$, and whose value may differ from line to line.

\section{Direct method of moving planes for $(-\Delta+m^{2})^{s}$}
In this Section, inspired by \cite{CL2,CLL,CW3}, we will establish various maximum principles and develop the direct method of moving planes for pseudo-relativistic Schr\"{o}dinger operators $(-\Delta+m^{2})^{s}$ (with $s\in(0,1)$ and $m>0$) and apply it to investigate symmetry, monotonicity and uniqueness of solutions to the following fractional order semi-linear equation
$$(-\Delta+m^{2})^s u(x)=f(u(x))$$
in different type of regions, including bounded domains $B_{1}(0)$, coercive epigraph $\Omega$ and the whole-space $\mathbb{R}^{N}$.

\subsection{Maximum principles for anti-symmetric functions and immediate applications}
In this subsection, we will establish various maximum principles for anti-symmetric functions and give some immediate applications. These maximum principles are key ingredients in applying the direct method of moving planes.

Let $T$ be any given hyper-plane in $\mathbb{R}^{N}$ and $\Sigma$ be the half space on one side of the plane $T$ hereafter. Denote the reflection of a point $x$ with respect to $T$ by $\tilde{x}$.

First, we can prove the following strong maximum principle for anti-symmetric functions.
\begin{lem}(Strong maximum principle for anti-symmetric functions)\label{SMP-anti}
Suppose that $w\in\mathcal{L}_{s}(\mathbb{R}^{N})$, $w\left(\tilde{x}\right)=-w(x)$ and $w\geq0$ in $\Sigma$. If there exists $x_{0}\in\Sigma$ such that, $w(x_{0})=0$, $w$ is $C^{1,1}$ near $x_{0}$ and $(-\Delta+m^{2})^{s}w(x_{0})\geq0$, then $w=0$ a.e. in $\mathbb{R}^N$.
\end{lem}
\begin{proof}
	Since there exists $x_0\in\Sigma$ such that $w(x_0)=\min_{x\in\Sigma}w(x)=0$, it follows from \eqref{Definition} that
	\begin{align*}
	&0\leq\left(-\Delta+m^{2}\right)^{s}w(x_0) \\ \nonumber
	&\quad =c_{N,s} m^{\Ns} P.V. \int_{\mathbb{R}^{N}} \frac{w(x_0)-w(y)}{|x_{0}-y|^{\Ns}} K_{\Ns}(m|x_{0}-y|)dy \\ \nonumber
	&\quad =-c_{N, s} m^{\Ns} P.V. \int_{\mathbb{R}^{N}} \frac{w(y)}{|x_{0}-y|^{\Ns}} K_{\Ns}(m|x_{0}-y|)dy \\ \nonumber
    &\quad =c_{N, s} m^{\Ns} P.V. \int_{\Sigma} \left(\frac{K_{\Ns}(m|x_{0}-\tilde{y}|)}{|x_{0}-\tilde{y}|^{\Ns}}-\frac{K_{\Ns}(m|x_{0}-y|)}{|x_{0}-y|^{\Ns}}\right)w(y)dy\leq0.
	\end{align*}
	Thus we must have $w=0$ a.e. in $\Sigma$ and hence $w=0$ a.e. in $\mathbb{R}^{N}$. This finishes the proof of Lemma \ref{SMP-anti}.
\end{proof}

\subsubsection{Maximum principles in bounded open sets}
\begin{thm}[Maximum principle for anti-symmetric functions]\label{MP Anti}
Let $\Omega$ be a bounded open set in $\Sigma$. Assume that $w\in \mathcal{L}_{s}(\mathbb{R}^{N})\cap C_{\text {loc}}^{1,1}(\Omega)$ and is lower semi-continuous on $\overline{\Omega}$. If
\begin{equation}\label{MP_anti}
\left\{\begin{array}{ll}{(-\Delta+m^{2})^s w(x)+c(x)w(x)\geq 0} & {\text {at points} \,\, x\in\Omega \,\, \text{where} \,\, w(x)<0} \\ {w(x) \geq 0} & {\text { in } \Sigma \setminus \Omega} \\ {w\left(\tilde{x}\right)=-w(x)} & {\text { in } \Sigma,}\end{array}\right.
\end{equation}
where $c(x)\geq-m^{2s}$ for any $x\in\left\{x\in\Omega\,|\,w(x)<0\right\}$. Then $w(x) \geq 0$ in $\Omega$.

Furthermore, assume that
\begin{equation}\label{2MP-anti-2}
  (-\Delta+m^{2})^s{w}(x)\geq0 \quad \text{at points} \,\, x\in\Omega \,\, \text{where} \,\, w(x)=0,
\end{equation}
then either $w>0$ in $\Omega$ or $w=0$ almost everywhere in $\mathbb{R}^{N}$.

These conclusions hold for unbounded open set $\Omega$ if we further assume that
$$
\liminf_{|x| \rightarrow \infty} w(x) \geq 0.
$$
\end{thm}
\begin{proof}
If $w$ is not nonnegative, then the lower semi-continuity of $w$ on $\overline{\Omega}$ indicates that
there exists a $\hat{x}\in \overline{\Omega}$ such that
$$
w\left(\hat{x}\right)=\min _{\overline{\Omega}} w<0.
$$
One can further deduce from \eqref{MP_anti} that $\hat{x}$ is in the interior of $\Omega$. It follows that
\begin{equation}\label{P:MP_inequa}
\begin{aligned}
&(-\De+m^{2})^s w\left(\hat{x}\right)+c(\hat{x})w\left(\hat{x}\right) \\
=&c_{N, s}m^{\frac{N}{2}+s}P.V.\int_{\mathbb{R}^{N}}\frac{w\left(\hat{x}\right)
-w(y)}{\left|\hat{x}-y\right|^{\Ns}}K_{\Ns}\left(m\left|\hat{x}-y\right|\right)dy
+\left(c(\hat{x})+m^{2s}\right)w\left(\hat{x}\right) \\
=&c_{N, s}m^{\frac{N}{2}+s}P.V.\left\{\int_{\Sigma} \frac{w\left(\hat{x}\right)-w(y)}{\left|\hat{x}-y\right|^{\Ns}}K_{\Ns}\left(m\left|\hat{x}-y\right|\right) dy\right. \\
&+\left.\int_{\Sigma}\frac{w\left(\hat{x}\right)+w(y)}{\left|\hat{x}-\tilde{y}\right|^{\Ns}}K_{\Ns}\left(m\left|\hat{x}-\tilde{y}\right|\right)dy\right\}
+\left(c(\hat{x})+m^{2s}\right)w\left(\hat{x}\right) \\
\leq& c_{N,s}m^{\frac{N}{2}+s}
\int_{\Sigma}\left\{\frac{w\left(\hat{x}\right)-w(y)}{\left|\hat{x}-\tilde{y}\right|^{\Ns}}+\frac{w\left(\hat{x}\right)+w(y)}{\left|\hat{x}-\tilde{y}\right|^{\Ns}}\right\}
K_{\Ns}\left(m\left|\hat{x}-\tilde{y}\right|\right)dy \\
=&c_{N,s}m^{\frac{N}{2}+s}\int_{\Sigma}\frac{2w\left(\hat{x}\right)}{\left|\hat{x}-\tilde{y}\right|^{\Ns}}K_{\Ns}\left(m\left|\hat{x}-\tilde{y}\right|\right)dy<0,
\end{aligned}
\end{equation}
which contradicts \eqref{MP_anti}. Hence $w(x)\geq 0$ in $\Omega$.

Now we have proved that $w(x)\geq 0$ in $\Sigma$. If there is some point $\bar{x}\in\Omega$ such that $w\left(\bar{x}\right)=0$, then from \eqref{2MP-anti-2} and Lemma \ref{SMP-anti}, we derive immediately $w=0$ almost everywhere in $\mathbb{R}^{N}$. This completes the proof of Theorem \ref{MP Anti}.
\end{proof}
\begin{rem}\label{rem12}
It is clear from the proof that, in Theorem \ref{MP Anti}, the assumptions ``$w$ is lower semi-continuous on $\overline{\Omega}$" and ``$w\geq0$ in $\Sigma\setminus\Omega$" can be weaken into: ``if $w<0$ somewhere in $\Sigma$, then the negative minimum $\inf_{\Sigma}w(x)$ can be attained in $\Omega$", the same conclusions are still valid. One can also notice that, we only need to assume that $c(x)\geq0$ at points $x\in\Omega$ where $w(x)=inf_{\Sigma}w<0$ in Theorems \ref{MP Anti}.
\end{rem}

\begin{thm}[Narrow region principle]\label{2NRP}
Let $\Omega$ be a bounded open set in $\Sigma$ which can be contained in the region between $T$ and $T_{\Omega}$, where $T_{\Omega}$ is a hyper-plane that is parallel to $T$. Let $d(\Omega):=dist(T,T_{\Omega})$. Suppose that $w\in \mathcal{L}_{s}(\mathbb{R}^{N})\cap C_{loc}^{1,1}(\Omega)$ and is lower semi-continuous on $\overline{\Omega}$, and satisfies
\begin{equation}\label{NRP-anti}
\left\{\begin{array}{ll}{(-\De+m^{2})^s w(x)+c(x)w(x)\geq 0} & {\text {at points} \,\, x\in\Omega \,\, \text{where} \,\, w(x)<0} \\ {w(x) \geq 0} & {\text {in } \Sigma \backslash \Omega} \\ {w\left(\tilde{x}\right)=-w(x)} & {\text {in } \Sigma,}\end{array}\right.
\end{equation}
where $c(x)$ is uniformly bounded from below (w.r.t. $d(\Omega)$) in $\{x\in\Omega\,|\,w(x)<0\}$. There exists a constant $C_{N,s}>0$ such that, if we assume $\Omega$ is narrow in the sense that, $d(\Omega)\leq\frac{r_{0}}{4m}$ and 
\begin{equation}\label{2NRP-3}
  \left(-\inf_{\left\{x\in\Omega\,|\,w(x)<0\right\}}c(x)-m^{2s}\right)d(\Omega)^{2s}<C_{N,s},
\end{equation}
then, $w(x) \geq 0 \text { in } \Omega$. Furthermore, assume that
\begin{equation}\label{2NRP-4}
  (-\Delta+m^{2})^s{w}(x)\geq0 \quad \text{at points} \,\, x\in\Omega \,\, \text{where} \,\, w(x)=0,
\end{equation}
then either $w>0$ in $\Omega$ or $w=0$ almost everywhere in $\mathbb{R}^{N}$.

These conclusions hold for unbounded open set $\Omega$ if we further assume that
$$
\liminf _{|x| \rightarrow \infty} w(x) \geq 0.
$$
\end{thm}
\begin{proof}
Without loss of generalities, we may assume that
\[T=\{x\in\mathbb{R}^{N}\,|\,x_{1}=0\} \quad \text{and} \quad \Sigma=\{x\in\mathbb{R}^{N}\,|\,x_{1}<0\},\]
and hence $\Omega\subseteq\{x\in\mathbb{R}^{N}\,|-d(\Omega)<x_{1}<0\}$.

If $w$ is not nonnegative in $\Omega$, then the lower semi-continuity of $w$ on $\overline{\Omega}$ indicates that, there exists a $\bar{x}\in\overline{\Omega}$ such that
$$
w\left(\bar{x}\right)=\min_{\overline{\Omega}} w<0.
$$
One can further deduce from \eqref{NRP-anti} that $\bar{x}$ is in the interior of $\Omega$. It follows that
\begin{equation}\label{2NRP-1}
\begin{aligned}
&(-\De+m^{2})^s w\left(\bar{x}\right) \\
=&c_{N,s}m^{\frac{N}{2}+s}P.V.\int_{\mathbb{R}^{N}}\frac{w\left(\bar{x}\right)
-w(y)}{\left|\bar{x}-y\right|^{\Ns}}K_{\Ns}\left(m\left|\bar{x}-y\right|\right)dy+m^{2s}w\left(\bar{x}\right) \\
=&c_{N,s}m^{\frac{N}{2}+s}P.V.\left\{\int_{\Sigma} \frac{w\left(\bar{x}\right)-w(y)}{\left|\bar{x}-y\right|^{\Ns}}K_{\Ns}\left(m\left|\bar{x}-y\right|\right) dy\right. \\
&+\left.\int_{\Sigma}\frac{w\left(\bar{x}\right)+w(y)}{\left|\bar{x}-\tilde{y}\right|^{\Ns}}K_{\Ns}\left(m\left|\bar{x}-\tilde{y}\right|\right) dy\right\}+m^{2s}w\left(\bar{x}\right) \\
\leq& c_{N,s}m^{\frac{N}{2}+s}
\int_{\Sigma}\left\{\frac{w\left(\bar{x}\right)-w(y)}{\left|\bar{x}-\tilde{y}\right|^{\Ns}}
+\frac{w\left(\bar{x}\right)+w(y)}{\left|\bar{x}-\tilde{y}\right|^{\Ns}}\right\}K_{\Ns}\left(m\left|\bar{x}-\tilde{y}\right|\right)dy
+m^{2s}w\left(\bar{x}\right) \\
=&c_{N,s}m^{\frac{N}{2}+s}\int_{\Sigma}\frac{2w\left(\bar{x}\right)}{\left|\bar{x}-\tilde{y}\right|^{\Ns}}K_{\Ns}\left(m\left|\bar{x}-\tilde{y}\right|\right)dy
+m^{2s}w\left(\bar{x}\right).
\end{aligned}
\end{equation}
Let
$$D:=\left\{y=(y_1,y')\in\mathbb{R}^{N} \mid d(\Omega)<y_{1}-(\bar{x})_{1}<2d(\Omega),\left|y^{\prime}-\left(\bar{x}\right)^{\prime}\right|<\frac{r_{0}}{2m}\right\},$$
then $m|y-\bar{x}|<r_0$ for all $y\in D$. Denote $t:=y_{1}-(\bar{x})_{1}$, $\tau:=\left|y^{\prime}-\left(\bar{x}\right)^{\prime}\right|$. Using \eqref{Bessel_asy0}, we have
\begin{equation}\label{2NRP-2}
\begin{aligned}
&m^{\frac{N}{2}+s}\int_{\Sigma}\frac{1}{\left|\bar{x}-\tilde{y}\right|^{\Ns}}K_{\Ns}\left(m\left|\bar{x}-\tilde{y}\right|\right)dy \\
\geq &m^{\frac{N}{2}+s}\int_{D}\frac{1}{\left|\bar{x}-y\right|^{\Ns}}K_{\Ns}\left(m\left|\bar{x}-y\right|\right)dy
\geq c_{0}\int_{D}\frac{1}{\left|\bar{x}-y\right|^{N+2s}}dy \\
=&c_{0}\int_{d(\Omega)}^{2d(\Omega)}\int_{0}^{\frac{r_0}{2m}}\frac{\sigma_{N-1}\tau^{N-2}d\tau}{\left(t^{2}+\tau^{2}\right)^{\Ns}}dt
=c_{0}\int_{d(\Omega)}^{2d(\Omega)}\int_{0}^{\frac{r_0}{2mt}}\frac{\sigma_{N-1}(t\rho)^{N-2}t d\rho}{t^{N+2s}\left(1+\rho^{2}\right)^{\Ns}}dt \\
=&c_{0}\int_{d(\Omega)}^{2d(\Omega)}\frac{1}{t^{1+2s}}\int_{0}^{\frac{r_0}{2mt}}\frac{\sigma_{N-1}\rho^{N-2} d\rho}{\left(1+\rho^{2}\right)^{\Ns}}dt
\geq c_{0}\int_{d(\Omega)}^{2d(\Omega)}\frac{1}{t^{1+2s}}\int_{0}^{1}\frac{\sigma_{N-1}\rho^{N-2} d\rho}{\left(1+\rho^{2}\right)^{\Ns}}dt \\
\geq&C_{N,s}\int_{d(\Omega)}^{2d(\Omega)}\frac{1}{t^{1+2s}}dt=\frac{C_{N,s}}{d(\Omega)^{2s}},
\end{aligned}
\end{equation}
where we have used the substitution $\rho:=\tau/t$ and $\sigma_{N-1}$ denotes the area of the unit sphere in $\mathbb{R}^{N-1}$. Since $c(x)$ is uniformly bounded from below (w.r.t. $d(\Omega)$) in $\{x\in\Omega\,|\,w(x)<0\}$, then, from \eqref{2NRP-3}, \eqref{2NRP-1} and \eqref{2NRP-2}, we get
$$
(-\De+m^{2})^s w\left(\bar{x}\right)+c\left(\bar{x}\right)w\left(\bar{x}\right)
\leq\left[\frac{C_{N,s}}{d(\Omega)^{2s}}+\inf_{\{x\in\Omega\,|\,w(x)<0\}}c(x)+m^{2s}\right]w\left(\bar{x}\right)<0,
$$
which contradicts \eqref{NRP-anti}.

Now we have proved that $w(x)\geq 0$ in $\Sigma$. If there is some point $\bar{x}\in\Omega$ such that $w\left(\bar{x}\right)=0$, then from \eqref{2NRP-4} and Lemma \ref{SMP-anti}, we derive immediately $w=0$ almost everywhere in $\mathbb{R}^{N}$. This finishes the proof of Theorem \ref{2NRP}.
\end{proof}
\begin{rem}\label{rem13}
It is clear from the proof that, in Theorem \ref{2NRP}, the assumptions ``$w$ is lower semi-continuous on $\overline{\Omega}$" and ``$w\geq0$ in $\Sigma\setminus\Omega$" can be weaken into: ``if $w<0$ somewhere in $\Sigma$, then the negative minimum $\inf_{\Sigma}w(x)$ can be attained in $\Omega$", the same conclusions are still valid. One can also notice that, in Theorem \ref{2NRP}, we only need to assume that $c(x)$ is uniformly bounded from below at the negative minimum points of $w$ and $\inf_{\left\{x\in\Omega\,|\,w(x)<0\right\}}c(x)$ can be replaced by the infimum of $c(x)$ over the set of negative minimum points of $w$ in \eqref{2NRP-3}.
\end{rem}

\subsubsection{Maximum principles in unbounded open sets and immediate applications}
\begin{thm}[Decay at infinity (I)]\label{P:decay}
Suppose $0\notin\Sigma$. Let $\Omega$ be an unbounded open set in $\Sigma$. Assume $w\in \mathcal{L}_{s}(\mathbb{R}^{N})\cap C_{loc}^{1,1}(\Omega)$ is a solution of
\begin{equation}\label{P:decay_eq}
\left\{\begin{array}{ll}{(-\De+m^{2})^s w(x)+c(x)w(x)\geq 0} & {\text {at points} \,\, x\in\Omega \,\, \text{where} \,\, w(x)<0} \\ {w(x) \geq 0} & {\text { in } \Sigma \backslash \Omega} \\ {w\left(\tilde{x}\right)=-w(x)} & {\text { in } \Sigma}\end{array}\right.
\end{equation}
with
\begin{equation}\label{decay_con}
\liminf_{\substack{x\in\Omega,\,w(x)<0 \\ |x| \rightarrow+\infty}}|x|^{s+\frac{1}{2}-\frac{N}{2}}e^{3m|x|}\left(c(x)+m^{2s}\right)\geq 0,
\end{equation}
then there exists a constant $R_{0}>0$ (depending only on $c(x)$, $m$, $N$ and $s$, but independent of $w$ and $\Sigma$) such that, if $\hat{x}\in\Omega$ satisfying
$$
w\left(\hat{x}\right)=\min_{\overline{\Omega}} w(x)<0,
$$
then $\left|\hat{x}\right|\leq R_{0}$.
\end{thm}
\begin{proof}
Without loss of generalities, we may assume that, for some $\lambda\leq0$,
\[T=\{x\in\mathbb{R}^{N}\,|\,x_{1}=\lambda\} \quad \text{and} \quad \Sigma=\{x\in\mathbb{R}^{N}\,|\,x_{1}<\lambda\}.\]

Since $w\in \mathcal{L}_{s}(\mathbb{R}^{N})\cap C_{loc}^{1,1}(\Omega)$ and $\hat{x}\in\Omega$ satisfying $w\left(\hat{x}\right)=\min _{\overline{\Omega}} w(x)<0$, through similar calculations as \eqref{2NRP-1}, we get
\begin{equation}\label{Decay-1}
(-\De+m^{2})^s w\left(\hat{x}\right)
\leq c_{N,s}m^{\frac{N}{2}+s}\int_{\Sigma}\frac{2w\left(\hat{x}\right)}{\left|\hat{x}-\tilde{y}\right|^{\Ns}}K_{\Ns}\left(m\left|\hat{x}-\tilde{y}\right|\right) dy+m^{2s}w\left(\hat{x}\right).
\end{equation}
Note that $\lambda\leq0$ and $\hat{x}\in\Omega$, it follows that $B_{\left|\hat{x}\right|}\left(\bar{x}\right) \subset \left\{x \in \mathbb{R}^{N} | x_{1}>\lambda\right\}$, where $\bar{x}:=\left(2\left|\hat{x}\right|+(\hat{x})_{1},\left(\hat{x}\right)^{\prime}\right)$. Thus we derive that, if $|\hat{x}|\geq\frac{R_{\infty}}{3m}$,
\begin{equation}\label{Decay-2}
\begin{aligned}
& m^{\frac{N}{2}+s}\int_{\Sigma}\frac{K_{\Ns}\left(m\left|\hat{x}-\tilde{y}\right|\right)}{\left|\hat{x}-\tilde{y}\right|^{\Ns}}dy \\
\geq &m^{\frac{N}{2}+s}\int_{B_{|\hat{x}|}\left(\bar{x}\right)}\frac{K_{\Ns}\left(m\left|\hat{x}-y\right|\right)}
{\left|\hat{x}-y\right|^{\Ns}} dy\geq m^{\frac{N}{2}+s}\int_{B_{\left|\hat{x}\right|}\left(\bar{x}\right)}\frac{K_{\Ns}\left(3m\left|\hat{x}\right|\right)}
{3^{\Ns}\left|\hat{x}\right|^{\Ns}} dy \\
\geq&\frac{c_{\infty}m^{\frac{N-1}{2}+s}\omega_N}{3^{\frac{N+1}{2}+s}\left|\hat{x}\right|^{s+\frac12-\frac{N}{2}}e^{3m\left|\hat{x}\right|}},
\end{aligned}
\end{equation}
where $\omega_{N}:=|B_{1}(0)|$ denotes the volume of unit ball in $\mathbb{R}^{N}$. Then we can deduce from \eqref{P:decay_eq}, \eqref{Decay-1} and \eqref{Decay-2} that
\begin{align}\label{Decay-3}
0&\leq(-\De+m^{2})^s w(\hat{x})+c(\hat{x})w(\hat{x}) \\
\nonumber &\leq\left[\frac{2c_{N,s}c_{\infty}\omega_{N}m^{\frac{N-1}{2}+s}}{3^{\frac{N+1}{2}+s}\left|\hat{x}\right|^{s+\frac12-\frac{N}{2}}e^{3m\left|\hat{x}\right|}}
+c(\hat{x})+m^{2s}\right]w(\hat{x}).
\end{align}
It follows from $w(\hat{x})<0$ and \eqref{Decay-3} that
\begin{equation}\label{Decay-4}
\left|\hat{x}\right|^{s+\frac{1}{2}-\frac{N}{2}}e^{3m\left|\hat{x}\right|}\left[c\left(\hat{x}\right)+m^{2s}\right]
\leq-\frac{2c_{N,s}c_{\infty}\omega_{N}m^{\frac{N-1}{2}+s}}{3^{\frac{N+1}{2}+s}}.
\end{equation}
From \eqref{decay_con}, we infer that there exists a $R_{1}$ sufficiently large such that, for any $x\in\Omega$ satisfying $|x|>R_{1}$ and $w(x)<0$,
\begin{equation}\label{Decay-5}
  |x|^{s+\frac{1}{2}-\frac{N}{2}}e^{3m|x|}\left(c(x)+m^{2s}\right)>-\frac{2c_{N,s}c_{\infty}\omega_{N}m^{\frac{N-1}{2}+s}}{3^{\frac{N+1}{2}+s}}.
\end{equation}
Combining \eqref{Decay-4} and \eqref{Decay-5}, we arrive at $|\hat{x}|\leq R_{1}$. Therefore, take $R_{0}:=\max\{\frac{R_{\infty}}{3m},R_{1}\}$, we must have $|\hat{x}|\leq R_{0}$. This completes the proof of Theorem \ref{P:decay}.
\end{proof}
\begin{rem}\label{rem14}
It is clear from the proofs of Theorems \ref{MP Anti}, \ref{2NRP} and \ref{P:decay} that, the assumption ``$(-\De+m^{2})^s w(x)+c(x)w(x)\geq 0$ at points $x\in\Omega$ where $w(x)<0$" can be weaken into: ``$(-\De+m^{2})^s w(x)+c(x)w(x)\geq 0$ at points $x\in\Omega$ where $w(x)=inf_{\Sigma}w<0$", the same conclusions in Theorems \ref{MP Anti}, \ref{2NRP} and \ref{P:decay} are still valid.
\end{rem}

\begin{thm}[Maximum principle for anti-symmetric functions in unbounded domains]\label{MP_anti_ubdd}
Assume that $w \in \mathcal{L}_{s}(\mathbb{R}^{N})\cap C_{\text {loc}}^{1,1}(\Sigma)$ is bounded from above and $w\left(\tilde{x}\right)=-w(x)$ in $\Sigma$, where $\tilde{x}$ is the reflection of $x$ with respect to $T$. Suppose that, at any points $x \in\Sigma$ such that $w(x)>0$, $w$ satisfies
\begin{equation}\label{MP-1}
\left(-\Delta+m^{2}\right)^{s} w(x)+c(x)w(x)\leq 0,
\end{equation}
where $c(x)$ satisfies
\begin{equation}\label{MP-assumption}
  \inf_{\{x\in\Sigma \mid w(x)>0\}}\,c(x)>-m^{2s}.
\end{equation}
Then
\begin{equation}\label{MP-2}
w(x) \leq 0, \quad \forall x \in\Sigma.
\end{equation}
Furthermore, assume that
\begin{equation}\label{MP-condition}
  (-\Delta+m^{2})^s{w}(x)\leq0 \quad \text{at points} \,\, x\in\Sigma \,\, \text{where} \,\, w(x)=0,
\end{equation}
then either $w<0$ in $\Sigma$ or $w=0$ in $\mathbb{R}^{N}$.
\end{thm}		
\begin{rem}\label{rem-MP}
One can observe that, we allow the function $c(x)$ in Theorem \ref{MP Anti}, \ref{P:decay} and \ref{MP_anti_ubdd} to be negative, which is different from corresponding assumptions on $c(x)$ in those maximum principles for fractional Laplacians $(-\Delta)^{s}$ and fractional $p$-Laplacians $(-\Delta)^{s}_{p}$ ($0<s<1$) (see, e.g. Theorem 1 and 3 in \cite{CLL}, Theorem 2 in \cite{CW3}).
\end{rem}
\begin{proof}
Suppose that \eqref{MP-2} is false, since $w$ is bounded from above, we have $M:=\sup_{\Sigma} w(x)>0$. Hence, there exists sequences $x^k\in\Sigma$ and $0<\beta_k<1$ with $\beta_k\rightarrow 1$ as $k\rightarrow \infty$ such that
	\begin{equation}\label{MP1-4}
	w(x^k)\geq \beta_k M.
	\end{equation}
	We may assume that $$T=\{x\in\mathbb{R}^N| x_1=0\}, \quad \Sigma=\{x\in\mathbb{R}^N| x_1<0\}.$$
	Then $\tilde{x}=(-x_1,x_2,\cdots,x_N)$. We denote $d_k:=\frac{1}{2}dist(x^k,T)$.
	Let
	\begin{equation*}
		\psi(x)=\begin{cases}e^{\frac{|x|^{2}}{|x|^2-1}}, \quad\,|x|<1\\ 0, \qquad\,\, \quad|x|\geq 1.\end{cases}
	\end{equation*}
	It is well known that $\psi\in C_0^\infty(\mathbb{R}^N)$, thus $|\left(-\Delta+m^{2}\right)^{s}\psi(x)|\leq C$ for all $x \in \mathbb{R}^N$. Moreover, $\left(-\Delta+m^{2}\right)^{s}\psi(x)\sim|x|^{-\frac{N+1}{2}-s}e^{-|x|}$ as $|x|\rightarrow +\infty$.
	
	Set $$\psi_k(x):=\psi\left(\frac{x-\widetilde{(x^k)}}{d_k}\right) \quad \text{and} \quad \tilde{\psi_k}(x)=\psi_k(\tilde{x})=\psi\left(\frac{x-x^k}{d_k}\right).$$
	Then $\tilde{\psi_k}-\psi_k$ is anti-symmetric with respect to $T$. Now pick $\varepsilon_k=(1-\beta_k)M$, then we have
	$$w(x^k)+\varepsilon_k[\tilde{\psi_k}-\psi_k](x^k)\geq M.$$
	We denote $$w_k(x):=w(x)+\varepsilon_k[\tilde{\psi_k}-\psi_k](x).$$ Then $w_k$ is also anti-symmetric with respect to $T$.
	
	Since for any $x\in\Sigma\setminus B_{d_k}(x^k)$, $w(x)\leq M$ and $\tilde{\psi_k}(x)=\psi_k(x)=0$, we have
    $$w_k(x^k)\geq w_k(x), \quad \forall \,\,x\in\Sigma\setminus B_{d_k}(x^k).$$
	Hence the supremum of $w_k(x)$ in $\Sigma$ is achieved in $B_{d_k}(x^k)$. Consequently, there exists a point $\overline{x}^k\in B_{d_k}(x^k)$ such that
	\begin{equation}\label{MP-3}
	w_k(\overline{x}^k)=\sup_{x\in\Sigma} w_k(x)\geq M.
	\end{equation}
	By the choice of $\varepsilon_k$, it is easy to verify that $w(\bar{x}^k)\geq \beta_k M>0$.
	
	Next, we will evaluate the upper bound and the lower bound of $\left(-\Delta+m^{2}\right)^{s}w_k(\bar{x}^k)$. To this end, we need the following Lemma.
	\begin{lem}\label{lem-MP}
		For any function $\phi\in C^{1,1}(\mathbb{R}^{N})$ with $\|\phi\|_{C^{1,1}(\mathbb{R}^{N})}<+\infty$ and any $r>0$, set $\phi_r(x):=\phi(\frac{x}{r})$. Then,
		\begin{equation*}
		\left|\left[\left(-\Delta+m^{2}\right)^{s}-m^{2s}\right]\phi_r(x)\right|\leq \frac{C}{r^{2s}}, \qquad \forall\,\, x\in\mathbb{R}^{N},
		\end{equation*}	
		where $C>0$ is a constant depending only on $N$, $s$ and $\|\phi\|_{C^{1,1}(\mathbb{R}^{N})}$, but independent of $r$.
	\end{lem}
\begin{proof}
From \eqref{Bessel_asy0} and \eqref{Bessel_asy1}, one can easily infer that, for any $r>0$ and $\nu>0$,
\begin{equation}\label{asymptotic}
  0<K_{\nu}(r)\leq\frac{C_{\nu}}{r^{\nu}}.
\end{equation}
Through straightforward calculations, we deduce that
	\begin{align*}
	&\quad \left|\left[\left(-\Delta+m^{2}\right)^{s}-m^{2s}\right]\phi_r(x)\right|\\ \nonumber
	&\leq c_{N, s} m^{\Ns}\left|P.V. \int_{\mathbb{R}^{N}} \frac{\phi_r(x)-\phi_r(y)}{|x-y|^{\Ns}} K_{\Ns}\left(m|x-y|\right)dy\right|\\ \nonumber
	&\leq c_{N, s} m^{\Ns}\left|P.V. \int_{B_r(x)}\frac{\phi_r(x)-\phi_r(y)}{|x-y|^{\Ns}} K_{\Ns}\left(m|x-y|\right) dy\right| \\ \nonumber
& \quad +c_{N, s}m^{\Ns}\left|P.V. \int_{\left(B_r(x)\right)^{c}}\frac{\phi_r(x)-\phi_r(y)}{|x-y|^{\Ns}} K_{\Ns}\left(m|x-y|\right) dy\right|\\ \nonumber
	&\leq c_{N, s}m^{\Ns}P.V.\int_{B_r(x)}\frac{\|\phi\|_{C^{1,1}(\mathbb{R}^{N})}\left|\frac{x}{r}-\frac{y}{r}\right|^2}{\left|x-y\right|^{\Ns}}\cdot \frac{C_{N,s}}{\left(m\left|x-y\right|\right)^{\Ns}}dy \\ \nonumber
& \quad +c_{N, s}m^{\Ns}\int_{\left(B_r(x)\right)^{c}}\frac{2\|\phi\|_{L^{\infty}(\mathbb{R}^{N})}}{|x-y|^{\Ns}}\cdot \frac{C_{N,s}}{\left(m\left|x-y\right|\right)^{\Ns}}dy\\ \nonumber
	&\leq \frac{C}{r^{2s}}.
	\end{align*}
	This concludes the proof of Lemma \ref{lem-MP}.	
\end{proof}	

\begin{rem}\label{rem7}
Lemma \ref{lem-MP} reveals one of the essential differences between the inhomogeneous fractional operators $(-\Delta+m^{2})^{s}$ and the fractional Laplacians $(-\Delta)^{s}$. For fractional Laplacians $(-\Delta)^{s}$, one can easily verify that $(-\Delta)^{s}\phi_{r}(x)=\frac{1}{r^{2s}}\phi\left(\frac{x}{r}\right)$ for any function $\phi\in \mathcal{L}_{s}(\mathbb{R}^{N})\cap C^{1,1}_{loc}(\mathbb{R}^{N})$.
\end{rem}

	As a consequence of \eqref{MP-1} and Lemma \ref{lem-MP}, we obtain the upper bound
    \begin{equation}\label{MP-4}
        \left(-\Delta+m^{2}\right)^{s}w_{k}(\bar{x}^k)\leq -c(\bar{x}^k)w(\bar{x}^k)+C\left(\frac{1}{d_k^{2s}}+1\right)\varepsilon_k.
    \end{equation}

	On the other hand, we have the following lower bound
	\begin{align}\label{MP-5}
	&\quad\left(-\Delta+m^{2}\right)^{s}w_k(\bar{x}^k)\\ \nonumber
	&=c_{N,s} m^{\Ns} P.V.\int_{\mathbb{R}^{N}}\frac{w_k(\bar{x}^k)-w_k(y)}{|\bar{x}^k-y|^{\Ns}} K_{\Ns}\left(m|\bar{x}^k-y|\right)dy+m^{2 s}
w_k(\bar{x}^k)\\ \nonumber
	&=c_{N, s}m^{\Ns}P.V. \int_{\Sigma}\left[\frac{w_k(\bar{x}^k)-w_k(y)}{|\bar{x}^k-y|^{\Ns}} K_{\Ns}(m|\bar{x}^k-y|)\right.\\ \nonumber
&\quad \left.+\frac{w_k(\bar{x}^k)+w_k(y)}{|\bar{x}^k-\tilde{y}|^{\Ns}}K_{\Ns}(m|\bar{x}^k-\tilde{y}|)\right]dy+m^{2s}w_k(\bar{x}^k)\\ \nonumber	
	&=c_{N,s} m^{\Ns} P.V.\int_{\Sigma}\left(\frac{K_{\Ns}(m|\bar{x}^k-y|)}{|\bar{x}^k-y|^{\Ns}}  -\frac{K_{\Ns}(m|\bar{x}^k-\tilde{y}|)}{|\bar{x}^k-\tilde{y}|^{\Ns}}\right)(w_k(\bar{x}^k)-w_k(y))dy\\ \nonumber
	&\quad +2c_{N,s}m^{\Ns} w_k(\bar{x}^k)\int_{\Sigma} \frac{K_{\Ns}(m|\bar{x}^k-\tilde{y}|)}{|\bar{x}^k-\tilde{y}|^{\Ns}}dy+m^{2s}w_k(\bar{x}^k)\\ \nonumber
	&>m^{2s}w_k(\bar{x}^k).
	\end{align}
	Here we have used the following facts:
	\begin{align*}
	&\frac{K_{\Ns}(m|\bar{x}^k-y|)}{|\bar{x}^k-y|^{\Ns}}>\frac{K_{\Ns}(m|\bar{x}^k-\tilde{y}|)}{|\bar{x}^k-\tilde{y}|^{\Ns}}, \quad \forall \,\, y\in \Sigma,\\
	&w_k(\bar{x}^k)\geq w_k(y),\quad \forall \,\, y\in \Sigma,
	\end{align*}
	which can be seen directly.
	
Next, we will carry out our proof by discussing two different cases and derive contradictions in both of these two cases.

\emph{Case (i).} There exists a $d_0>0$ such that $d_k\geq d_0$ for all $k$. Then \eqref{MP-4} implies
    \begin{equation}\label{MP-4'}
    \left(-\Delta+m^{2}\right)^{s} w_{k}(\bar{x}^k)\leq -c(\bar{x}^k)w(\bar{x}^k)+C\varepsilon_k.
    \end{equation}
	Combining \eqref{MP-5} and \eqref{MP-4'}, we derive
	\begin{equation*}
	m^{2 s}w(\bar{x}^k)\leq m^{2 s}w_k(\bar{x}^k)\leq -\left(\inf_{\{x\in\Sigma \mid w(x)>0\}}\, c(x)\right)w(\bar{x}^k)+C(1-\beta_{k})M,
	\end{equation*}
	which implies that
$$\left(\inf_{\{x\in\Sigma \mid w(x)>0\}}c(x)+m^{2s}\right)\beta_{k}\leq C(1-\beta_k).$$
This will lead to a contradiction if we take the limit $k\rightarrow+\infty$.
	
\emph{Case (ii).} Up to a subsequence (still denote by $d_{k}$), $0<d_k\leq\frac{r_{0}}{4m}$ for every $k\geq 1$. Let
$$D_k:=\left\{x=(x_1,x')\in\mathbb{R}^N| -d_k<x_1<0, |x'-(\bar{x}^k)'|<\frac{r_{0}}{m}\right\}.$$
Through further calculations, we can get the following refinement of \eqref{MP-5}:
	\begin{align}\label{MP-5'}
	&\quad\left(-\Delta+m^{2}\right)^{s}w_k(\bar{x}^k)\\\nonumber
	&\geq 2c_{N, s} m^{\Ns} w_k(\bar{x}^k)\int_{\Sigma} \frac{K_{\Ns}(m|\bar{x}^k-\tilde{y}|)}{|\bar{x}^k-\tilde{y}|^{\Ns}}dy+m^{2s}w_k(\bar{x}^k)\\\nonumber
	&\geq 2c_{N, s}c_{0}w_k(\bar{x}^k)\int_{D_k} \frac{1}{|\bar{x}^k-\tilde{y}|^{N+2s}}dy+m^{2 s} w_k(\bar{x}^k)\\\nonumber
	&\geq \frac{C_{N,s}}{d_k^{2s}}w_k(\bar{x}^k)+m^{2 s} w_k(\bar{x}^k),
	\end{align}	
	where the last inequality can be derived in similar way as \eqref{2NRP-2}. Therefore, we infer from \eqref{MP-4} and \eqref{MP-5'} that
	\begin{equation}\label{MP1-6}
	\frac{C_{N,s}}{d_{k}^{2s}}w_k(\bar{x}^k)+m^{2 s} w_k(\bar{x}^k)\leq -c(\bar{x}^k)w(\bar{x}^k)+C\left(\frac{1}{d_k^{2s}}+1\right)\varepsilon_k,
	\end{equation}
	which implies
\begin{equation}\label{2MP-1}
  \frac{C_{N,s}}{d_{k}^{2s}}+m^{2s}\leq -\inf_{\{x\in\Sigma \mid w(x)>0\}}\,c(x)+C\left(\frac{1}{d_k^{2s}}+1\right)\frac{1-\beta_{k}}{\beta_{k}}.
\end{equation}
Now we choose $k$ large enough such that $\beta_{k}>\frac{1}{2}$ and $1-\beta_k<\frac{C_{N,s}}{4C}$, then we have
\begin{equation}\label{2MP-2}
  \frac{C_{N,s}}{2d_{k}^{2s}}+m^{2s}\leq -\inf_{\{x\in\Sigma \mid w(x)>0\}}\,c(x)+2C(1-\beta_k).
\end{equation}

From \eqref{2MP-2}, we can deduce the following:

\emph{(i).} If
\begin{equation}\label{2MP-4}
  \inf_{\{x\in\Sigma \mid w(x)>0\}}\,c(x)>-\left(1+\frac{2^{4s-1}C_{N,s}}{r_{0}^{2s}}\right)m^{2s},
\end{equation}
since $0<d_{k}\leq\frac{r_{0}}{4m}$, then \eqref{2MP-2} yields a contradiction if we let $k\rightarrow+\infty$.

\emph{(ii).} If we only assume $c(x)$ is bounded from below in $\{x\in\Sigma \mid w(x)>0\}$, then we can also get a contradiction from \eqref{2MP-2} by taking the limit $k\rightarrow+\infty$, provided that there is a $0<\delta\leq\frac{r_{0}}{4m}$ such that, $0<d_{k}\leq\delta$ for $k$ large enough and
\begin{equation}\label{2MP-3}
  -\delta^{2s}\left(\inf_{\{x\in\Sigma \mid w(x)>0\}}\, c(x)+m^{2s}\right)<\frac{C_{N,s}}{2}.
\end{equation}

Now we have proved that $w(x)\leq 0$ in $\Sigma$. If there is some point $\bar{x}\in\Omega$ such that $w\left(\bar{x}\right)=0$, then from \eqref{MP-condition} and Lemma \ref{SMP-anti}, we derive immediately $w=0$ almost everywhere in $\mathbb{R}^{N}$ and hence $w=0$ in $\mathbb{R}^{N}$. This concludes our proof of Theorem \ref{MP_anti_ubdd}.
\end{proof}	

\begin{rem} \label{eg} 
Being essentially different from fractional Laplacians $(-\Delta)^{s}$ ($s\in(0,1)$), since there is a \emph{modified Bessel function of the second kind $K_{\frac{N}{2}+s}$} in the definition of inhomogeneous fractional operators $(-\Delta+m^{2})^{s}$, $(-\Delta+m^{2})^{s}$ ($s\in(0,1)$) \emph{do not possess} any \emph{invariance properties} under \emph{Kelvin-type or scaling transforms}. In order to overcome these difficulties, we need to derive fine estimates on the upper and lower bounds of $\left(-\Delta+m^{2}\right)^{s}w_k(\bar{x}_k)$ (see Lemma \ref{lem-MP}, \eqref{MP-4}, \eqref{MP-5} and \eqref{MP-5'} in the proof of Theorem \ref{MP_anti_ubdd}).
\end{rem}

\begin{rem}\label{rem19}
It can be seen from the proof that, in Theorem \ref{MP_anti_ubdd}, the assumption \eqref{MP-assumption} on $c(x)$ can be weaken into
\[\liminf_{\substack{x\in\Sigma, \,w(x)>0 \\ |x|\rightarrow+\infty}}c(x)>-m^{2s} \quad \text{and} \quad \inf_{\{x\in\Sigma \mid w(x)>0\}}\,c(x)\geq-m^{2s},\]
the same conclusions are still valid. In fact, we only need to consider, instead of Cases (i) and (ii) in the proof of Theorem \ref{MP_anti_ubdd}, the following two cases: \\
Case (a). Up to a subsequence (still denote by $d_{k}$), $d_{k}\rightarrow+\infty$ as $k\rightarrow+\infty$. \\
Case (b). There exists a $D_{0}$ such that $d_{k}\leq D_{0}$ for every $k\geq1$.
\end{rem}

From the proof of Theorem \ref{MP_anti_ubdd}, we can deduce the following narrow region principle in unbounded open sets, which improves the narrow region principle (Theorem \ref{2NRP}).
\begin{thm}[Narrow region principle in unbounded open sets]\label{NRP-Anti-ubdd}
Let $\Omega\subseteq\Sigma$ be an open set (possibly unbounded and disconnected) which can be contained in the region between $T$ and $T_{\Omega}$, where $T_{\Omega}$ is a hyper-plane that is parallel to $T$. Let $d(\Omega):=dist(T,T_{\Omega})\leq\frac{r_{0}}{2m}$. Suppose that $w\in \mathcal{L}_{s}(\mathbb{R}^{N})\cap C_{loc}^{1,1}(\Omega)$ is bounded from below and satisfies
\begin{equation}\label{NRP-anti-ubdd}
\left\{\begin{array}{ll}{(-\De+m^{2})^s w(x)+c(x)w(x)\geq 0} & {\text {at points} \,\, x\in\Omega \,\, \text{where} \,\, w(x)<0} \\ {w(x) \geq 0} & {\text {in } \Sigma \backslash \Omega} \\ {w\left(\tilde{x}\right)=-w(x)} & {\text {in } \Sigma,}\end{array}\right.
\end{equation}
where $c(x)$ is uniformly bounded from below (w.r.t. $d(\Omega)$) in $\{x\in\Omega\,|\,w(x)<0\}$. If we assume that either
\begin{equation}\label{2NRP-4-ubdd}
  \inf_{\{x\in\Omega \,\mid\, w(x)<0\}}\,c(x)>-\left(1+\frac{2^{4s-1}C_{N,s}}{r_{0}^{2s}}\right)m^{2s},
\end{equation}
or
\begin{equation}\label{2NRP-3-ubdd}
  \left(-\inf_{\left\{x\in\Omega\,|\,w(x)<0\right\}}c(x)-m^{2s}\right)d(\Omega)^{2s}<\frac{C_{N,s}}{2},
\end{equation}
where $C_{N,s}$ is the same as in \eqref{2MP-4} and \eqref{2MP-3}. Then, $w(x) \geq 0 \text { in } \Omega$. Furthermore, assume that
\begin{equation}\label{2NRP-5-ubdd}
  (-\Delta+m^{2})^s{w}(x)\geq0 \quad \text{at points} \,\, x\in\Omega \,\, \text{where} \,\, w(x)=0,
\end{equation}
then either $w>0$ in $\Omega$ or $w=0$ almost everywhere in $\mathbb{R}^{N}$.
\end{thm}
\begin{proof}
By replacing $w$ by $-w$, Theorem \ref{NRP-Anti-ubdd} can be proved by using quite similar arguments as in the proof of Theorem \ref{MP_anti_ubdd}. More precisely, note that $d_{k}\leq \frac{1}{2}d(\Omega)\leq\frac{r_{0}}{4m}$, Theorem \ref{NRP-Anti-ubdd} follows immediately from \emph{Case (ii)} in the proof of Theorem \ref{MP_anti_ubdd}. We omit the details here.
\end{proof}
\begin{rem}\label{rem8}
In Theorem \ref{NRP-Anti-ubdd}, we allow the open set $\Omega$ to be unbounded without the additional assumption $\liminf_{|x|\rightarrow+\infty}w(x)\geq0$ in Theorem \ref{2NRP}.
\end{rem}

From the proof of Theorem \ref{MP_anti_ubdd}, we can also deduce the following maximum principle in unbounded domains, which develops the \emph{Decay at infinity (I)} (Theorem \ref{P:decay}).
\begin{thm}[Decay at infinity (II)]\label{dati}
Let $\Omega$ be an unbounded open set in $\Sigma$. Assume $w\in \mathcal{L}_{s}(\mathbb{R}^{N})\cap C_{loc}^{1,1}(\Omega)$ is bounded from below and satisfies
\begin{equation}\label{dati-eq}
\left\{\begin{array}{ll}{(-\De+m^{2})^s w(x)+c(x)w(x)\geq 0} & {\text {at points} \,\, x\in\Omega \,\, \text{where} \,\, w(x)<0} \\ {w(x) \geq 0} & {\text { in } \Sigma \backslash \Omega} \\ {w\left(\tilde{x}\right)=-w(x)} & {\text { in } \Sigma}\end{array}\right.
\end{equation}
with
\begin{equation}\label{dati-con}
\liminf_{\substack{x\in\Omega,\,w(x)<0 \\ |x|\rightarrow \infty}}c(x)>-m^{2s},
\end{equation}
then there exists a $R_{0}>0$ large enough and $\gamma_{0}\in (0,1)$ close enough to $1$ ($R_{0}$ and $\gamma_{0}$ are independent of $w$ and $\Sigma$) such that, if $\hat{x}\in\Omega$ satisfying
$$
w\left(\hat{x}\right)\leq\gamma_{0}\inf _{\Omega}w(x)<0,
$$
then $\left|\hat{x}\right|\leq R_{0}$.
\end{thm}
\begin{proof}
Theorem \ref{dati} can be proved via similar contradiction arguments as Theorem \ref{MP_anti_ubdd}.

Indeed, suppose on the contrary that there exists sequences $\{x^{k}\}\in\Omega$ and $\{\gamma_{k}\}\in(0,1)$ such that
\begin{equation}\label{dati-1}
  |x^{k}|\rightarrow+\infty, \quad \gamma_{k}\rightarrow 1, \quad \text{and} \quad w(x^{k})\leq\gamma_{k}\inf _{\Omega}w(x)<0.
\end{equation}
Then, we can derive a contradiction through entirely similar way as in the proof of Theorem \ref{MP_anti_ubdd} by simply replacing $w$ with $-w$ and $\beta_{k}$ with $\gamma_{k}$. We omit the details here.
\end{proof}
\begin{rem}\label{rem11}
We say \emph{Decay at infinity (II)} Theorem \ref{dati} developed \emph{Decay at infinity (I)} Theorem \ref{P:decay} in the sense that, not only the positions of minimal points but also the positions of ``almost" negative minimal points were controlled by a radius $R_{0}$ in Theorem \ref{dati}. Moreover, Theorem \ref{dati} also tell us that, if $\inf _{\Omega}w(x)<0$, then $\Omega\cap B_{R_{0}}(0)\neq\emptyset$ and the negative minimum can be attained in $\Omega\cap B_{R_{0}}(0)$.
\end{rem}

As an immediate application of Theorem \ref{MP_anti_ubdd}, we can obtain the following Liouville type Theorem in $\mathbb{R}^{N}$. For Liouville theorem on $s$-harmonic functions in $\mathbb{R}^{N}$, please refer to \cite{F} and the references therein.
\begin{thm}(Liouville Theorem)\label{Liouville}
	Assume that $u\in\mathcal{L}_{s}(\mathbb{R}^{N})\cap C_{\text {loc}}^{1,1}(\mathbb{R}^{N})$ is bounded, if
	\begin{equation}\label{MP-6}
	\left(-\Delta+m^{2}\right)^{s}u(x)=f(u(x)) \quad \text{in}\,\,\mathbb{R}^{N},
	\end{equation}
where the function $f(\cdot)$ satisfies
\begin{equation}\label{MP-con2}
 \sup_{\substack{t_{1},\,t_{2}\in[\inf u,\,\sup u] \\ t_1>t_2}}\frac{f(t_1)-f(t_2)}{t_1-t_2}<m^{2s}.
 \end{equation}
Then
$$u\equiv C \qquad \text{with} \,\, C \,\, \text{satisfying} \,\, f(C)=m^{2s}C.$$
\end{thm}
\begin{proof}
Let $T$ be any hyper-plane, $\Sigma$ be the half space on one side of the plane $T$. Set $w(x)=u(\tilde{x})-u(x)$ for all $x\in \Sigma$, where $\tilde{x}$ is the reflection of $x$ with respect to $T$. Then, $w \in\mathcal{L}_{s}(\mathbb{R}^{N})\cap C_{\text {loc}}^{1,1}(\mathbb{R}^{N})$ is bounded, and at any points $x\in\Sigma$ where $w(x)>0$, one has $\left(-\Delta+m^{2}\right)^{s}w(x)=f(u(\tilde{x}))-f(u(x))=c(x)w(x)$, where $c(x):=\frac{f(u(\tilde{x}))-f(u(x))}{u(\tilde{x})-u(x)}$ satisfies
\[\sup_{\{x\in\Sigma \mid w(x)>0\}}c(x)<m^{2s}.\]
Therefore, applying Theorem \ref{MP_anti_ubdd}, we arrive immediately $w\leq 0$ in $\Sigma$. Similarly, we can prove that $w\leq 0$ in $\mathbb{R}^N\setminus\Sigma$. Hence $w\equiv0$ in $\mathbb{R}^{N}$, and $u$ is symmetric with respect to $T$. Since $T$ is arbitrary, we must have $u\equiv C$ in $\mathbb{R}^{N}$. By the definition of $(-\Delta+m^{2})^{s}$, one has
	\begin{equation*}
	f(C)=\left(-\Delta+m^{2}\right)^{s} u(x)=m^{2s}u(x)=m^{2s}C,
	\end{equation*}
thus $C$ is determined by the equation $f(C)=m^{2s}C$. This finishes the proof of Theorem \ref{Liouville}.
\end{proof}

Next, let us consider the following equation
\begin{equation}\label{MP-7}
\left(-\Delta+m^{2}\right)^{s} u(x)=f(u(x)), \quad \forall \,\, x\in \mathbb{R}^N.
\end{equation}

As another application of Theorem \ref{MP_anti_ubdd}, we derive the following monotonicity result on \eqref{MP-7}.
\begin{thm}\label{Mono_1}
	Suppose $u\in\mathcal{L}_{s}(\mathbb{R}^{N})\cap C_{\text {loc}}^{1,1}(\mathbb{R}^{N})$ is a solution of \eqref{MP-7}, and
    $$|u(x)|\leq 1,\quad\forall \, x\in \mathbb{R}^{N},$$
	\begin{equation}\label{MP-8}
	\lim_{x_N\rightarrow\pm \infty}u(x',x_N)=\pm1 \quad \text{uniformly w.r.t.} \,\, x' \in \mathbb{R}^{N-1}.
	\end{equation}
Assume there exists a $\delta>0$ such that
	\begin{equation}\label{MP-9}
	\sup_{\substack{-1\leq t_{1}<t_{2}\leq-1+\delta \\ 1-\delta\leq t_{1}<t_{2}\leq1}}\frac{f(t_{2})-f(t_{1})}{t_{2}-t_{1}}<m^{2s},
	\end{equation}
	then there exists $M>0$ such that, $u(x)$ is strictly monotone increasing w.r.t. $x_{N}$ provided that $|x_{N}|>M$.
\end{thm}
\begin{proof}
For arbitrary $\lambda\in\mathbb{R}$, let $T_{\lambda}:=\left\{x \in \mathbb{R}^{N} \,|\, x_{N}=\lambda\right\}$, $\Sigma_{\lambda}:=\left\{x \in \mathbb{R}^{N} \,|\, x_{N}>\lambda\right\}$ be the region above the plane, and $x^{\lambda}:=\left(x_{1}, x_{2}, \ldots, 2 \lambda-x_{N}\right)$ be the reflection of point $x$ about the plane $T_{\lambda}$.

We only need to show that $w_{\lambda}(x):=u(x^{\lambda})-u(x)\leq 0$ in $\Sigma_{\lambda}$ for any $\lambda$ with $|\lambda|$ sufficiently large. By the assumption \eqref{MP-8}, there exists $M>0$ such that $u(x)\in[-1,-1+\delta]\cup[1-\delta,1]$ for any $x$ with $|x_N|>M$. Consequently, for any $|\lambda|>M$, at any point $x\in \Sigma_{\lambda}$ where $w_\lambda(x)=u(x^\lambda)-u(x)>0$, we infer from assumption \eqref{MP-9} that $\left(-\Delta+m^{2}\right)^{s} w_\lambda(x)=f(u(x^\lambda))-f(u(x))=c_{\lambda}(x)w_\lambda(x)$ with $c_{\lambda}(x):=\frac{f(u(x^\lambda))-f(u(x))}{u(x^{\lambda})-u(x)}$ satisfying
\[\sup_{\{x\in\Sigma_{\lambda} \mid w_{\lambda}(x)>0\}}c_{\lambda}(x)<m^{2s}.\]
Therefore, we deduce from Theorem \ref{MP_anti_ubdd} that $w_{\lambda}(x):=u(x^{\lambda})-u(x)\leq 0$ in $\Sigma_{\lambda}$ for all $\lambda$ with $|\lambda|>M$.

Now, suppose that there exists a $\widetilde{\lambda}\in(-\infty,-M)\cup(M,+\infty)$ and a point $\hat{x}\in\Sigma_{\widetilde{\lambda}}$ such that $w_{\widetilde{\lambda}}(\hat{x})=0$. Then, it follows that
\begin{equation}\label{MP-10'}
  (-\Delta+m^{2})^{s}w_{\widetilde{\lambda}}(\hat{x})=f(u((\hat{x})^{\widetilde{\lambda}}))-f(u(\hat{x}))=0,
\end{equation}
and hence we can derive from Lemma \ref{SMP-anti} immediately that $w_{\widetilde{\lambda}}(x)=0$ almost everywhere in $\mathbb{R}^{N}$, which contradicts assumption \eqref{MP-8}. Thus $w_{\lambda}(x):=u(x^{\lambda})-u(x)<0$ in $\Sigma_{\lambda}$ for all $\lambda$ with $|\lambda|>M$. This finishes our proof of Theorem \ref{Mono_1}.
\end{proof}
\begin{rem}\label{rem1}
One should note that the De Giorgi type nonlinearity $f(u)=u-u^3$ satisfies condition \eqref{MP-9}. The assumption \eqref{MP-9} in Theorem \ref{Mono_1} is weaker than those assumed in \cite{BHM,CLiu,CW1,CW2,CW3,DSV} for Laplacian $-\Delta$ and fractional Laplaicans $(-\Delta)^{s}$, which requires $f(\cdot)$ to be non-increasing on $[-1,-1+\delta]$ and $[1-\delta,1]$.
\end{rem}

\subsection{Direct method of moving planes and its applications}
By using various maximum principles for anti-symmetric functions established in subsection 2.1, we will apply the direct method of moving planes to investigate symmetry and monotonicity of solutions to various problems involving the pseudo-relativistic Schr\"{o}dinger operators $(-\Delta+m^{2})^{s}$.

\subsubsection{Bounded domains}
Consider the following Dirichlet problem on unit ball:
\begin{equation}\label{Ball}
	\begin{cases} \left(-\Delta+m^{2}\right)^{s}u(x)=f(u(x)),\quad \forall \,\,x\in B_{1}(0),\\ u(x)>0,\quad \forall \,\, x\in B_{1}(0),\\
	u=0, \quad \forall \,\, x\in\mathbb{R}^N\setminus B_{1}(0).
	\end{cases}
	\end{equation}
By applying the direct method of moving planes, we will prove the following symmetry and monotonicity result. For related results for $-\Delta$ or $(-\Delta)^{s}$, please refer to \cite{CLM,GNN1}.
\begin{thm}\label{thm-221}
Assume that $u\in C^{1,1}_{loc}(B_{1}(0))\cap C(\overline{B_{1}(0)})$ is a positive solution to \eqref{Ball} with $f(\cdot)$ being Lipschitz continuous. Then $u$ must be radially symmetric and strictly monotone decreasing w.r.t. the origin $0$.
\end{thm}
\begin{proof}
In order to carry out the moving plane procedure, we choose an arbitrary direction to be the $x_{1}$-direction. For arbitrary $\lambda\in(-1,1)$, let
$$
T_{\lambda}:=\left\{x \in \mathbb{R}^{N} | x_{1}=\lambda\right\}
$$
be the moving planes,
\begin{equation}
\Sigma_{\lambda}:=\left\{x \in B_{1}(0) \mid x_{1}<\lambda\right\}
\end{equation}
be the region to the left of the plane in $B_{1}(0)$, and
$$
x^{\lambda}:=\left(2 \lambda-x_{1}, x_{2}, \ldots, x_{N}\right)
$$
be the reflection of $x\in B_{1}(0)$ about the plane $T_{\lambda}$.
To compare the values of $u(x)$ with $u\left(x^{\lambda}\right)$, we define
$$
w_{\lambda}(x):=u\left(x^{\lambda}\right)-u(x), \quad \forall \,\, x\in\Sigma_{\lambda}.
$$
Our goal is to show that $w_{\lambda}\geq0$ in $\Sigma_{\lambda}$ for any $\lambda\in(-1,0]$.

By equation \eqref{Ball}, it is easy to verify that, for any $\lambda\in(-1,0]$,
\begin{equation}\label{221-0}
  (-\Delta+m^{2})^{s}w_{\lambda}(x)+c_{\lambda}(x)w_{\lambda}(x)=0, \quad \forall \,\, x\in\Sigma_{\lambda}
\end{equation}
with
\begin{equation*}
	c_{\lambda}(x):=\begin{cases} -\frac{f(u(x))-f\left(u(x^{\lambda})\right)}{u(x)-u(x^{\lambda})},\quad \text{if}\,\,u(x)\not=u(x^{\lambda}),\\ \\ 0,\quad \text{if}\,\,u(x)=u(x^{\lambda})\\
	\end{cases}
	\end{equation*}
satisfying
\begin{equation}\label{221-1}
  \|c_{\lambda}\|_{L^{\infty}(\Sigma_{\lambda})}\leq L,
\end{equation}
where $L$ (independent of $\lambda$) is the Lipschitz constant for the function $f(\cdot)$.

We will carry out the moving plane procedure in two steps.

\emph{Step 1.} Start moving the plane $T_{\lambda}$ from near $\lambda=-1$ to the right along $x_{1}$-axis. Observe that $\Sigma_{\lambda}$ is a narrow region for $\lambda$ close enough to $-1$ and
\[w_{\lambda}(x)\geq0, \quad \forall \,\, x\in H_{\lambda}\setminus\Sigma_{\lambda},\]
where $H_{\lambda}:=\{x\in\mathbb{R}^{N} \mid x_{1}<\lambda\}$. Thus \eqref{221-0}, \eqref{221-1} and the narrow region principle Theorem \ref{2NRP} imply that, for $\lambda$ close to $-1$,
\begin{equation}\label{221-aim}
  w_{\lambda}(x)>0, \quad \forall \,\, x\in\Sigma_{\lambda}.
\end{equation}

\emph{Step 2.} Move the plane to continuously to the right until its limiting position. Step 1 provides a start point for us to move planes, next we will continue to move $T_{\lambda}$ to the right as long as \eqref{221-aim} holds.

To this end, let us define
\begin{equation}\label{221-def}
  \lambda_{0}:=\sup\left\{\lambda\in(-1,0] \mid w_{\mu}>0 \,\, \text{in} \,\, \Sigma_{\mu}, \,\, \forall \, -1<\mu<\lambda\right\}.
\end{equation}
We aim to show that $\lambda_{0}=0$ via contradiction arguments.

Suppose on the contrary that $\lambda_{0}<0$, then we will be able to move $T_{\lambda}$ to the right a little bit further while \eqref{221-aim} still holds, which contradicts the definition \eqref{221-def} of $\lambda_{0}$.

Indeed, due to $\lambda_{0}<0$, one can infer from \eqref{Ball} that $w_{\lambda_{0}}>0$ in $\left((B_{1}(0))^{\lambda_{0}}\cap H_{\lambda_{0}}\right)\setminus\Sigma_{\lambda_{0}}$ ($A^{\lambda}$ denotes the reflection of a set $A$ w.r.t. $T_{\lambda}$), and hence the strong maximum principle Lemma \ref{SMP-anti} yields that
\begin{equation}\label{221-2}
  w_{\lambda_{0}}(x)>0, \quad \forall \,\, x\in\Sigma_{\lambda_{0}}.
\end{equation}
Now select $0<\epsilon_{1}<\min\{-\lambda_{0},\lambda_{0}+1\}$ small enough such that $\Sigma_{\lambda_{0}+\epsilon_{1}}\setminus\overline{\Sigma_{\lambda_{0}-\epsilon_{1}}}$ is a narrow region. Since $w_{\lambda_{0}}>0$ in $(B_{1}(0))^{\lambda_{0}}\cap H_{\lambda_{0}}$, so there exists a constant $c>0$ such that
\begin{equation}\label{221-3}
  w_{\lambda_{0}}(x)\geq c>0, \quad \forall \,\, x\in\overline{\Sigma_{\lambda_{0}-\epsilon_{1}}}.
\end{equation}
Due to the continuity of $w_{\lambda}$ w.r.t. $\lambda$, we get, there exists a sufficiently small $0<\epsilon_{2}<\epsilon_{1}$ such that, for any $\lambda\in[\lambda_{0},\lambda_{0}+\epsilon_{2}]$,
\begin{equation}\label{221-4}
  w_{\lambda}(x)>0, \quad \forall \,\, x\in\overline{\Sigma_{\lambda_{0}-\epsilon_{1}}}.
\end{equation}
For any $\lambda\in[\lambda_{0},\lambda_{0}+\epsilon_{2}]$, note that $\Sigma_{\lambda}\setminus\overline{\Sigma_{\lambda_{0}-\epsilon_{1}}}$ is a narrow region, we can deduce from the narrow region principle Theorem \ref{2NRP} that
\begin{equation}\label{221-5}
  w_{\lambda}(x)>0, \quad \forall \,\, x\in\Sigma_{\lambda}\setminus\overline{\Sigma_{\lambda_{0}-\epsilon_{1}}},
\end{equation}
and hence, for any $\lambda\in[\lambda_{0},\lambda_{0}+\epsilon_{2}]$,
\begin{equation}\label{221-6}
  w_{\lambda}(x)>0, \quad \forall \,\, x\in\Sigma_{\lambda}.
\end{equation}
This contradicts with the definition \eqref{221-def} of $\lambda_{0}$. Thus $\lambda_{0}=0$, or more precisely,
\begin{equation}\label{221-7}
  u(-x_{1},x_{2},\cdots,x_{N})\geq u(x_{1},x_{2},\cdots,x_{N}), \quad \forall \,\, x\in\Sigma_{0}.
\end{equation}
Through similar arguments, we can also move the plane $T_{\lambda}$ from near $\lambda=1$ to the left and derive
\begin{equation}\label{221-8}
  u(-x_{1},x_{2},\cdots,x_{N})\leq u(x_{1},x_{2},\cdots,x_{N}), \quad \forall \,\, x\in B_{1}(0) \,\, \text{with} \,\, x_{1}<0.
\end{equation}
Combining \eqref{221-7} with \eqref{221-8}, we derive that $u(x)$ is symmetric about the plane $T_{0}$. Since the $x_{1}$-direction is arbitrarily chosen, $u(x)$ is radially symmetric w.r.t. the origin $0$. The strict monotonicity comes directly from the moving planes procedure. This completes the proof of Theorem \ref{thm-221}.
\end{proof}

\subsubsection{Coercive epigraph $\Omega$}
A domain $\Omega \subseteq \mathbb{R}^N$ is a coercive epigraph if there exists a continuous function $\varphi : \mathbb{R}^{N-1}\rightarrow \mathbb{R}$ satisfying
\begin{equation}\label{M-1}
\lim_{|x'|\rightarrow +\infty}\varphi(x')=+\infty,
\end{equation}
such that $\Omega=\{x=(x',x_N)\in\mathbb{R}^N|x_N>\varphi(x')\}$.

In this setting, we can prove the following monotonicity result via the direct method of moving planes.
\begin{thm}\label{Mono-eg}
	Let $\Omega$ be a coercive epigraph, and let $u\in\mathcal{L}_{s}(\mathbb{R}^{N})\cap C^{1,1}_{loc}(\Omega)\cap C(\overline{\Omega})$ be a solution to
	\begin{equation}\label{M-2}
	\begin{cases} \left(-\Delta+m^{2}\right)^{s} u=f(x,u)\quad \text{in}\,\, \Omega,\\ u>0\quad \text{in}\,\, \Omega,\\
	u=0 \quad \text{in}\,\, \mathbb{R}^N\setminus\Omega,
	\end{cases}
	\end{equation}
	with $f(x,t)$ continuous in $\overline{\Omega}\times \mathbb{R}$, non-decreasing in $x_N$, and locally Lipschitz continuous in $t$, locally uniformly in $x$, in the following sense: for any $M>0$ and any compact set $K \subset \overline{\Omega}$, there exists $C>0$ depending on $M$ and $K$ such that
	\begin{equation}\label{M-3}
	\sup_{x\in K}\frac{|f(x,t)-f(x,\tau)|}{|t-\tau|}\leq C, \quad \forall \,\, t,\,\tau \in [-M, M].
	\end{equation}
	Then $u$ is strictly monotone increasing in $x_N$.
\end{thm}
\begin{proof}
	Without loss of generality, we assume $$\inf_{x\in\Omega}x_N=0.$$
	For arbitrary $\lambda>0$, let
	$$
	T_{\lambda}:=\left\{x \in \mathbb{R}^{N} | x_{N}=\lambda\right\}
	$$
	be the moving planes,
	\begin{equation}
	\Sigma_{\lambda}:=\left\{x \in \mathbb{R}^{N} | x_{N}<\lambda\right\}
	\end{equation}
	be the region below the plane, and
	$$
	x^{\lambda}:=\left(x_{1}, x_{2}, \ldots, 2 \lambda-x_{N}\right)
	$$
	be the reflection of $x$ about the plane $T_{\lambda}$.
	
	Assume that $u$ is a solution to problem \eqref{M-2}. To compare the values of $u(x)$ with $u_{\lambda}(x):=u\left(x^{\lambda}\right)$, we denote
	$$
	w_{\lambda}(x):=u_{\lambda}(x)-u(x).
	$$
	Since $\Omega$ is a coercive epigraph, $\Sigma_{\lambda}\cap \Omega$ is always bounded for every $\lambda>0$.
	We aim at proving that $w_{\lambda}>0$ in $\Sigma_{\lambda}\cap\Omega$ for every $\lambda>0$, which gives the desired monotonicity.
	
For $\lambda>0$,  since $2\lambda-x_N>x_N$ in $\Sigma_{\lambda}$, we infer from the  monotonicity of $f$ in $x_N$ that
	\begin{align}\label{M-4}
	\left(-\Delta+m^{2}\right)^{s}	w_{\lambda}(x)&=f\left(x^{\lambda},u_{\lambda}(x)\right)-f(x,u(x))\\\nonumber
	&\geq f\left(x,u_{\lambda}(x)\right)-f(x,u(x))\\\nonumber
	&=-c_{\lambda}(x)w_{\lambda}(x),
	\end{align}
	in $\Sigma_{\lambda}\cap \Omega$, with
	\begin{equation*}
	c_{\lambda}(x):=\begin{cases} \frac{f(x,u(x))-f\left(x,u_{\lambda}(x)\right)}{u_{\lambda}(x)-u(x)},\quad \text{if}\,\,u_{\lambda}(x)\not=u(x),\\ \\ 0,\quad \text{if}\,\,u_{\lambda}(x)=u(x).\\
	\end{cases}
	\end{equation*}
	Notice that, thanks to the locally Lipschitz continuity of $f$ and the local boundedness of $u$, for any $\overline{\lambda}>0$, there exists $C>0$ depending on $\overline{\lambda}$ such that $\|c_\lambda\|_{L^{\infty}(\Sigma_{\lambda}\cap\Omega)}\leq C$ for all $0<\lambda\leq\overline{\lambda}$. We also have
	\begin{equation}\label{M-5}
	w_{\lambda}(x)\geq 0 \qquad \text{in}\,\, \Sigma_{\lambda}\setminus \Omega.
	\end{equation}
	
	We will carry out the method of moving planes in two steps.

	\emph{Step 1.} We will first show that, for $\lambda>0$ sufficiently close to $0$,
	\begin{eqnarray}\label{M-6}
	w_{\lambda}>0 \qquad \text{in} \,\, \Sigma_{\lambda}\cap\Omega.
	\end{eqnarray}
	Indeed, when $\lambda$ sufficiently close to $0$, $\Sigma_{\lambda}\cap \Omega$ is a narrow region. By \eqref{M-4} and \eqref{M-5}, \eqref{M-6} follows directly from the narrow region principle Theorem \ref{2NRP}.
	
	\emph{Step 2.} Inequality \eqref{M-6} provides a starting point for us to carry out the moving planes procedure. Now we increase $\lambda$ from close to $0$ to $+\infty$ as long as inequality \eqref{M-6} holds until its limiting position. Define
	\begin{equation}\label{M-7}
	\lambda_{0}:=\sup \left\{\lambda>0 \mid w_{\mu}> 0 \,\, \text{in} \,\, x \in\Sigma_{\mu}\cap\Omega, \,\, \forall \, 0<\mu<\lambda\right\}.
	\end{equation}
	We aim to prove that
	$$
	\lambda_{0}=+\infty.
	$$
	
	Otherwise, suppose on the contrary that $0<\lambda_{0}<+\infty$, we will show that the plane $T_{\lambda_0}$ can be moved upward a little bit more, that is, there exists an $\varepsilon>0$ small enough such that
	\begin{eqnarray}\label{M-8}
	w_{\lambda}>0  \quad \text{in} \,\,\Sigma_{\lambda}\cap\Omega, \qquad \forall \,\, \lambda_{0}\leq\lambda\leq\lambda_{0}+\varepsilon,
	\end{eqnarray}
	which contradicts the definition \eqref{M-7} of $\lambda_{0}$.
	
First, by the definition of $\lambda_{0}$, we have $w_{\lambda_{0}}\geq 0$ in $\Sigma_{\lambda_0}\cap\Omega$. Since $u>0$ in $\Omega$ and $u\equiv 0$ in $\mathbb{R}^N\setminus\Omega$, we have $w_{\lambda_{0}}(x)> 0$ for any $x\in\Omega^{\lambda_{0}}\setminus\Omega$, where the notation $A^{\lambda}$ denotes the reflection of a given set $A$ w.r.t. the plane $T_{\lambda}$. Then, we obtain from Theorem \ref{MP Anti} that actually $w_{\lambda_{0}}>0$ in $\Sigma_{\lambda_0}\cap\Omega$.

Next, we choose $\varepsilon_1>0$ sufficiently small such that $\left(\Sigma_{\lambda_{0}+\varepsilon_1}\setminus\overline{\Sigma_{\lambda_{0}-\varepsilon_1}}\right)\cap \Omega$ is a bounded narrow region. By the fact that $w_{\lambda_{0}}(x)>0$ for $x\in\Omega^{\lambda_{0}}\cap\Sigma_{\lambda_{0}}$ and the continuity of $w_{\lambda_{0}}$, there exists $c_0>0$ such that
$$w_{\lambda_{0}}(x)> c_0,\qquad \forall \,\, x\in\overline{\Sigma_{\lambda_0-\varepsilon_1}}\cap \Omega.$$
Therefore, we can choose $0<\varepsilon_2<\varepsilon_1$ sufficiently small such that
	\begin{equation}\label{M-9}
	w_{\lambda}(x)>\frac{c_0}{2}>0,\qquad \forall \,\, x\in\overline{\Sigma_{\lambda_0-\varepsilon_1}}\cap \Omega,
	\end{equation}
	for every $\lambda_0\leq\lambda\leq\lambda_0+\varepsilon_2$. Since $\left(\Sigma_{\lambda_{0}+\varepsilon_2}\setminus\overline{\Sigma_{\lambda_{0}-\varepsilon_1}}\right)\cap \Omega$ is also a narrow region, we deduce from \eqref{M-5}, \eqref{M-9} and the narrow region principle Theorem \ref{2NRP} that
	\begin{eqnarray*}
	w_{\lambda}>0  \quad \text{in} \,\,\Sigma_{\lambda}\cap\Omega, \qquad \forall \,\, \lambda_{0}\leq\lambda\leq\lambda_{0}+\varepsilon_2,
	\end{eqnarray*}
    which contradicts the definition \eqref{M-7} of $\lambda_{0}$. Thus, we must have $\lambda_0=+\infty$. This completes the proof of Theorem \ref{Mono-eg}.
\end{proof}	
\begin{rem}\label{rem9}
Theorem \ref{Mono-eg} is the inhomogeneous counterpart for the monotonicity results in Theorem 1.3 in Dipierro, Soave and Valdinoci \cite{DSV} for $(-\Delta)^{s}$, Theorem 1.3 in Berestycki, Caffarelli and Nirenberg \cite{BCN2} and Proposition II.1 in Esteban and Lions \cite{EL} for $-\Delta$.
\end{rem}

\subsubsection{Whole space $\mathbb{R}^{N}$}
Consider \emph{generalized static pseudo-relativistic Schr\"odinger equations}:
\begin{equation}\label{sta_poly}
	\lr-\De+m^{2}\rr^{s}u(x)+\omega u(x)=u^p(x), \qquad \forall \,\,x\in\mathbb{R}^{N}.
\end{equation}

We will prove the following symmetry and monotonicity result for nonnegative solution to \eqref{sta_poly} via direct method of moving planes for $(-\Delta+m^{2})^{s}$.
\begin{thm}\label{T:sub_symm}
	Assume that $u \in\mathcal{L}_{s}(\mathbb{R}^{N})\cap C_{loc}^{1,1}(\mathbb{R}^{N})$ is a nonnegative solution of \eqref{sta_poly} with $\omega>-m^{2s}$ and $1<p<+\infty$. If
	\begin{equation}\label{sub_con}
		\limsup_{|x|\rightarrow+\infty}u(x)=l<\left(\frac{\omega+m^{2s}}{p}\right)^{\frac{1}{p-1}},
	\end{equation}
	then $u$ must be radially symmetric and monotone decreasing about some point in $\mathbb{R}^{N}$.
\end{thm}
\begin{proof}
Choose an arbitrary direction to be the $x_{1}$-direction. In order to apply the method of moving planes, we need some notations. For arbitrary $\lambda\in\mathbb{R}$, let
$$
T_{\lambda}:=\left\{x \in \mathbb{R}^{N} | x_{1}=\lambda\right\}
$$
be the moving planes,
\begin{equation}
\Sigma_{\lambda}:=\left\{x \in \mathbb{R}^{N} | x_{1}<\lambda\right\}
\end{equation}
be the region to the left of the plane, and
$$
x^{\lambda}:=\left(2 \lambda-x_{1}, x_{2}, \cdots, x_{N}\right)
$$
be the reflection of $x$ about the plane $T_{\lambda}$.

Assume that $u$ is a nonnegative solution of the pseudo-relativistic Schr\"odinger equations \eqref{sta_poly}. To compare the values of $u(x)$ with $u\left(x^{\lambda}\right)$, we define
$$
w_{\lambda}(x):=u\left(x^{\lambda}\right)-u(x), \quad \forall \,\, x\in\Sigma_{\lambda}.
$$
Then, for any $\lambda\in\mathbb{R}$, at points $x\in\Sigma_{\lambda}$ where $w_{\lambda}(x)<0$, we have
\begin{equation}\label{sub_diff}
(-\De+m^{2})^s w_{\lambda}(x)+c(x)w_{\lambda}(x)\geq 0,
\end{equation}
where $c(x):=\omega-pu^{p-1}(x)$. From the assumption \eqref{sub_con}, we infer that, for any $\lambda\in\mathbb{R}$,
\begin{equation}\label{223-0}
  \liminf_{\substack{x\in\Sigma_{\lambda},\, w_{\lambda}(x)<0 \\ |x|\rightarrow+\infty}}c(x)>-m^{2s}.
\end{equation}

We carry out the moving planes procedure in two steps.

\emph{Step 1.} We use Theorem \ref{dati} (\emph{Decay at infinity (II)}) to show that, for sufficiently negative $\lambda$,
\begin{equation}\label{223-aim}
  w_{\lambda}(x)\geq 0, \quad \forall \,\, x\in\Sigma_{\lambda}.
\end{equation}

In fact, from assumption \eqref{sub_con}, we know that $u$ is bounded from above and hence $w_{\lambda}$ is bounded from below for any $\lambda\in\mathbb{R}$. Suppose that $\inf_{\Sigma_{\lambda}}w_{\lambda}<0$. By \eqref{sub_diff} and \eqref{223-0}, we can deduce from Theorem \ref{dati} (\emph{Decay at infinity (II)}) that, there exist $R_{0}>0$ large and $0<\gamma_{0}<1$ close to $1$ (independent of $\lambda$) such that, if $\hat{x}\in\Sigma_{\lambda}$ satisfying $w_{\lambda}(\hat{x})\leq\gamma_{0}\inf_{\Sigma_{\lambda}}w_{\lambda}<0$, then $|\hat{x}|\leq R_{0}$. This will lead to a contradiction provided that $\lambda\leq-R_{0}$. Thus we have, for any $\lambda\leq-R_{0}$, $w_{\lambda}\geq0$ in $\Sigma_{\lambda}$.

\emph{Step 2.} Step 1 provides a starting point, from which we can now move the plane $T_\la$ to the right as long as \eqref{223-aim} holds to its limiting position.

To this end, let us define
\begin{equation}\label{223-5}
  \lambda_{0}:=\sup\left\{\lambda\in\mathbb{R} \mid w_{\mu}\geq 0 \,\, \text{in} \,\, \Sigma_{\mu}, \,\, \forall \,\mu \leq \lambda\right\}.
\end{equation}
It follows from Step 1 that $-R_{0}\leq\lambda_{0}<+\infty$. One can easily verify that
\begin{equation}\label{223-1}
  w_{\lambda_{0}}(x)\geq0, \quad \forall \,\, x\in\Sigma_{\lambda_{0}}.
\end{equation}

Next, we are to show via contradiction arguments that
\begin{equation}\label{223-2}
  w_{\lambda_{0}}(x) \equiv 0, \quad \forall \,\, x \in \Sigma_{\lambda_{0}}.
\end{equation}
Suppose on the contrary that
\begin{equation}\label{223-11}
  w_{\lambda_{0}}\geq 0 \,\, \text { but } \,\, w_{\lambda_{0}}\not\equiv 0 \quad \text { in } \,\, \Sigma_{\lambda_{0}},
\end{equation}
then we must have
\begin{equation}\label{223-3}
  w_{\lambda_{0}}(x)>0, \quad \forall \,\, x\in\Sigma_{\lambda_{0}}.
\end{equation}
In fact, if \eqref{223-3} is violated, then there exists a point $\hat{x}\in \Sigma_{\lambda_{0}}$ such that
$$
w_{\lambda_{0}}(\hat{x})=\min_{\Sigma_{\lambda_{0}}}w_{\lambda_{0}}=0.
$$
Then it follows from \eqref{sta_poly} that
\begin{equation}\label{223-4}
(-\De+m^{2})^s w_{\lambda_{0}}(\hat{x})=0,
\end{equation}
and hence Lemma \ref{SMP-anti} implies that $w_{\lambda_{0}}\equiv0$ in $\Sigma_{\lambda_{0}}$, which contradicts \eqref{223-11}. Thus $w_{\lambda_{0}}(x)>0$ in $\Sigma_{\lambda_{0}}$.

Then we will show that the plane $T_{\lambda}$ can be moved a little bit further from $T_{\lambda_{0}}$ to the right. More precisely, there exists an $\delta>0$, such that for any $\lambda\in\left[\lambda_{0},\lambda_{0}+\delta\right]$, we have
\begin{equation}\label{223-6}
  w_{\lambda}(x)\geq 0, \quad \forall \,\, x \in \Sigma_{\lambda}.
\end{equation}

In fact, \eqref{223-6} can be achieved by using the \emph{Narrow region principle} Theorem \ref{2NRP} and the \emph{Decay at infinity (II)} Theorem \ref{dati}. First, since $c(x):=\omega-pu^{p-1}(x)$ is uniformly bounded, we can choose $\delta_{1}>0$ small enough such that $\left(\Sigma_{\lambda_{0}+\delta_{1}}\setminus\overline{\Sigma_{\lambda_{0}-\delta_{1}}}\right)\cap B_{R_{\ast}}(0)$ is a narrow region, where $R_{\ast}:=R_{0}+|\lambda_{0}|\geq R_{0}$ with $R_{0}$ given by \emph{Decay at infinity (II)} Theorem \ref{dati}. From \eqref{223-3}, we deduce that, there exists a $c_{0}>0$ such that
\begin{equation}\label{223-7}
  w_{\lambda_{0}}(x)\geq c_{0}>0, \qquad \forall \,\, x\in\overline{\Sigma_{\lambda_{0}-\delta_{1}}\cap B_{R_{\ast}}(0)}.
\end{equation}
As a consequence, due to the continuity of $w_{\lambda}$ w.r.t. $\lambda$, there exists a $0<\delta_{2}<\delta_{1}$ sufficiently small such that, for any $\lambda\in[\lambda_{0},\lambda_{0}+\delta_{2}]$,
\begin{equation}\label{223-8}
  w_{\lambda}(x)>0, \qquad \forall \,\, x\in\overline{\Sigma_{\lambda_{0}-\delta_{1}}\cap B_{R_{\ast}}(0)}.
\end{equation}
For any $\lambda\in[\lambda_{0},\lambda_{0}+\delta_{2}]$, if we suppose that $\inf_{\Sigma_{\lambda}}w_{\lambda}(x)<0$, then the \emph{Decay at infinity (II)} Theorem \ref{dati} implies that
\[w_{\lambda}(x)>\gamma_{0}\inf_{\Sigma_{\lambda}}w_{\lambda}(x), \quad \forall \,\, x\in\Sigma_{\lambda}\setminus\overline{B_{R_{0}}(0)},\]
and hence the negative minimum $\inf_{\Sigma_{\lambda}}w_{\lambda}(x)$ can be attained in $B_{R_{0}}(0)\cap\Sigma_{\lambda}$. Then, from \eqref{223-8}, we infer that, if $\inf_{\Sigma_{\lambda}}w_{\lambda}(x)<0$, then the negative minimum $\inf_{\Sigma_{\lambda}}w_{\lambda}(x)$ can be attained in the narrow region $\left(\Sigma_{\lambda}\setminus\overline{\Sigma_{\lambda_{0}-\delta_{1}}}\right)\cap B_{R_{\ast}}(0)$. Therefore, from Narrow region principle Theorem \ref{2NRP} (see Remark \ref{rem13}), we get, for any $\lambda\in[\lambda_{0},\lambda_{0}+\delta_{2}]$,
\begin{equation}\label{223-9}
  w_{\lambda}(x)>0, \qquad \forall \,\, x\in\left(\Sigma_{\lambda}\setminus\overline{\Sigma_{\lambda_{0}-\delta_{1}}}\right)\cap B_{R_{\ast}}(0),
\end{equation}
and hence
\begin{equation}\label{223-10}
  w_{\lambda}(x)\geq0, \qquad \forall \,\, x\in\Sigma_{\lambda}.
\end{equation}
Thus \eqref{223-6} holds, which contradicts the definition \eqref{223-5} of $\lambda_{0}$. Hence \eqref{223-2} must be valid.

The arbitrariness of the $x_{1}$-direction leads to the radial symmetry and monotonicity of $u(x)$ about some point $x_{0}\in\mathbb{R}^{N}$. This completes the proof of Theorem \ref{T:sub_symm}.
\end{proof}
\begin{rem}\label{rem15}
If we use \emph{Decay at infinity (I)} Theorem \ref{P:decay} in the proof of Theorem \ref{T:sub_symm}, then we will need the stronger assumption $$\lim_{|x|\rightarrow+\infty}u(x)=l<\left(\frac{\omega+m^{2s}}{p}\right)^{\frac{1}{p-1}}$$
instead of \eqref{sub_con}. One can observe that, by using \emph{Decay at infinity (II)} Theorem \ref{dati} instead of Theorem \ref{P:decay}, the ``limit" can be weaken into ``superior limit" in assumption \eqref{sub_con}.

Similar to \emph{Decay at infinity (II)} Theorem \ref{dati}, the \emph{Decay at infinity} Theorem 3 for $(-\Delta)^{s}$ in Chen, Li and Li \cite{CLL} can also be improved to control the locations of ``almost" negative minimal points. By doing this, the assumption (59) in Theorem 3.3 in \cite{CLL} can be weaken into:
\[\limsup_{|x|\rightarrow+\infty}u(x)=a<\left(\frac{1}{p}\right)^{\frac{1}{p-1}}.\]
We leave the details to interested readers.
\end{rem}

\begin{rem}\label{rem20}
By applying the direct method of moving planes for $(-\Delta+m^{2})^{s}$ developed here, we can also derive symmetry and monotonicity of nonnegative solutions to the following generalized equations involving pseudo-relativistic Schr\"{o}dinger operators:
\begin{equation}\label{gPRSE}
\left(-\Delta+m^{2}\right)^{s}u(x)=f(u(x)), \qquad \forall \,\, x\in\mathbb{R}^{N},
\end{equation}
under assumptions that, either
\begin{equation*}
  u \,\, \text{is bounded from above} \quad \text{and} \quad \sup_{\substack{t_{1},\,t_{2}\in[\inf u,\,\sup u] \\ t_{1}>t_{2}}}\frac{f(t_{1})-f(t_{2})}{t_{1}-t_{2}}<m^{2s},
\end{equation*}
or
\begin{equation*}
  \lim_{|x|\rightarrow+\infty}u(x)=0 \quad \text{and} \quad f'(t)\leq m^{2s}, \,\,\, \forall \,\, t \,\,\text{sufficiently small}.
\end{equation*}
We leave these open problems to interested readers.
\end{rem}

\subsubsection{Generalized stationary boson star equations in $\mathbb{R}^{3}$}
Consider the following physically interesting \emph{3D static pseudo-relativistic Hartree or Choquard equations}:
\begin{equation}\label{RSO-HC}
\left(-\Delta+m^{2}\right)^su+\omega u=\left(\frac{1}{|x|}*u^{2}\right)u \qquad \text{in}\,\,\mathbb{R}^3,
\end{equation}
where $\omega>-m^{2s}$. In the special case $s=\frac{1}{2}$, equation \eqref{RSO-HC} describes pseudo-relativistic boson star (see \cite{FL}).

By applying the direct method of moving planes, we prove the following Theorem on symmetry and monotonicity of ground state solutions to \eqref{RSO-HC}.
\begin{thm}\label{RSO-HC-Symm}
	Assume that $u \in {\mathcal{L}}_{s}(\mathbb{R}^{3})\cap C_{loc}^{1,1}(\mathbb{R}^{3})$ is a nonnegative solution of \eqref{RSO-HC} satisfying
	\begin{equation}\label{HC-0}
    \int_{\mathbb{R}^{3}}\frac{u(x)}{|x|^{2}}dx<+\infty \qquad \text{and} \qquad u(x)=o\left(\frac{1}{|x|}\right) \quad \text{as}\,\,|x| \rightarrow \infty.
	\end{equation}
	
	Then $u$ must be radially symmetric and monotone decreasing about some point in $\mathbb{R}^{3}$.
\end{thm}

\begin{proof}
Choose an arbitrary direction to be the $x_{1}$-direction. In order to apply the direct method of moving planes along $x_{1}$-axis, we need some notations. For any $\lambda\in\mathbb{R}$, let $T_{\lambda}$, $\Sigma_{\lambda}$, $x^{\lambda}$ and $w_{\lambda}$ be defined the same as in the proof of Theorem \ref{T:sub_symm} in subsubsection 2.2.3. Define $u_{\lambda}(x):=u(x^{\lambda})$. Set
$$\Sigma_\lambda^-:=\{x\in\Sigma_\lambda \mid w_\lambda(x)<0\}.$$

Since the assumption \eqref{HC-0} implies that $u$ is bounded, so $w_{\lambda}$ is also bounded. For $x\in \Sigma_\lambda^-$, we infer from \eqref{RSO-HC} that
\begin{align}\label{HC-1}
&\quad\left(-\Delta+m^{2}\right)^{s}	w_{\lambda}(x)+\omega w_{\lambda}(x)\\\nonumber
&=\left(\frac{1}{|x|}*u^{2}\right)(x^\lambda)u_\lambda(x)-\left(\frac{1}{|x|}*u^{2}\right)u(x)\\\nonumber
&=\left(\frac{1}{|x|}*u^{2}\right)(x)w_{\lambda}(x)+\Big(\int_{\mathbb{R}^{3}}\frac{u^{2}(y)}{|x^{\lambda}-y|}dy-\int_{\mathbb{R}^{3}}\frac{u^{2}(y)}{|x-y|}dy\Big)u_{\lambda}(x) \\\nonumber
&=\left(\frac{1}{|x|}*u^{2}\right)(x)w_{\lambda}(x)+ u_{\lambda}(x)\int_{\Sigma_\lambda}\Big(\frac{1}{|x-y|}-\frac{1}{|x^{\lambda}-y|}\Big)\big(u_\lambda^2(y)-u^2(y)\big)dy \\\nonumber
&\geq \left(\frac{1}{|x|}*u^{2}\right)(x)w_{\lambda}(x)+2u(x)\int_{\Sigma_\lambda^-}\Big(\frac{1}{|x-y|}-\frac{1}{|x^{\lambda}-y|}\Big)u(y)w_\lambda(y)dy \\\nonumber
&\geq \int_{\mathbb{R}^3}\frac{u^2(y)}{|x-y|}dy\,w_{\lambda}(x)+2u(x)\inf_{\Sigma_\lambda^-}w_\lambda\int_{\Sigma_\lambda^-}\Big(\frac{1}{|x-y|}-\frac{1}{|x^{\lambda}-y|}\Big)u(y)dy.
\end{align}
Thus, at points $x\in\Sigma_\lambda^-$ where $w_{\lambda}(x)=\inf_{\Sigma_\lambda^-}w_\lambda$, we obtain
\begin{equation}\label{HC-2}
\left(-\Delta+m^{2}\right)^{s}w_{\lambda}(x)+c_\lambda(x)w_\lambda(x)\geq 0,
\end{equation}
where $c_\lambda(x):=\omega-\int_{\mathbb{R}^3}\frac{u^2(y)}{|x-y|}dy-2u(x)\int_{\Sigma_\lambda^-}\Big(\frac{1}{|x-y|}-\frac{1}{|x^{\lambda}-y|}\Big)u(y)dy$.

Next, we will prove that
\begin{equation}\label{HC-3}
\lim_{x\in\Sigma_{\lambda}^{-},\,|x|\rightarrow+\infty} c_\lambda(x)>-m^{2s},
\end{equation}
for any $\lambda\leq0$. Since $\omega>-m^{2s}$, we only need to prove that, for any $\lambda\leq0$,
\begin{equation}\label{HC-4}
  \lim_{x\in\Sigma_{\lambda}^{-},\,|x|\rightarrow+\infty}\left\{\int_{\mathbb{R}^3}\frac{u^2(y)}{|x-y|}dy+2u(x)\int_{\Sigma_\lambda^-}\Big(\frac{1}{|x-y|}-\frac{1}{|x^{\lambda}-y|}\Big)u(y)dy\right\}=0.
\end{equation}

To this end, we first infer from the assumption \eqref{HC-0} that
$$\int_{\mathbb{R}^{3}}\frac{u^{2}(x)}{|x|}dx<+\infty,$$
and for any $\epsilon>0$, there exists $R_{\epsilon}>0$ large enough such that
\[u(x)\leq\frac{\epsilon}{|x|}, \quad \forall \,\, |x|\geq R_{\epsilon}.\]
It is also easy to check that $|x-y|\geq\frac{|x|}{2}$ implies $|x-y|\geq\frac{|y|}{3}$. Consequently, for any $\lambda\leq 0$ and $x\in\Sigma_{\lambda}^{-}$ with $|x|>2R_{\epsilon}$ sufficiently large, by straightforward calculations, we have
\begin{align}\label{HC-5}
&\quad\int_{\mathbb{R}^3}\frac{u^2(y)}{|x-y|}dy\\ \nonumber
&\leq\frac{2}{\sqrt{|x|}}\int_{\substack{|y-x|\geq\frac{|x|}{2} \\
|y|<\sqrt{|x|}}}\frac{u^{2}(y)}{|y|}dy
+3\int_{\substack{|x-y|\geq\frac{|x|}{2} \\ |y|\geq\sqrt{|x|}}}\frac{u^{2}(y)}{|y|}dy
+\frac{4\epsilon^{2}}{|x|^{2}}\int_{|x-y|<\frac{|x|}{2}}\frac{1}{|x-y|}dy\\ \nonumber
&\leq \frac{2}{\sqrt{|x|}}\int_{\mathbb{R}^{3}}\frac{u^{2}(x)}{|x|}dx+3\int_{\substack{|y|\geq\sqrt{|x|}}}\frac{u^{2}(y)}{|y|}dy+C\epsilon^{2} \\ \nonumber
&\leq \frac{C}{\sqrt{|x|}}+o_{|x|}(1)+C\epsilon^{2},
\end{align}
\begin{align}\label{HC-6}
&\quad 2u(x)\int_{\Sigma_\lambda^-}\left(\frac{1}{|x-y|}-\frac{1}{|x^{\lambda}-y|}\right)u(y)dy\\ \nonumber
&\leq 2u(x)\int_{\Sigma_\lambda^-}\frac{|x^{\lambda}-y|-|x-y|}{|x-y|\cdot|x^{\lambda}-y|}u(y)dy \\ \nonumber
&\leq 4|x|u(x)\int_{\mathbb{R}^{3}}\frac{u(y)}{|x-y|^{2}}dy \\ \nonumber
&\leq\frac{16\epsilon}{|x|}\int_{\substack{|y-x|\geq\frac{|x|}{2} \\
|y|<\sqrt{|x|}}}\frac{u(y)}{|y|^{2}}dy
+9\int_{\substack{|x-y|\geq\frac{|x|}{2} \\ |y|\geq\sqrt{|x|}}}\frac{u(y)}{|y|^{2}}dy
+\frac{8\epsilon^{2}}{|x|}\int_{|x-y|<\frac{|x|}{2}}\frac{1}{|x-y|^{2}}dy\\ \nonumber
&\leq\frac{16\epsilon}{|x|}\int_{\mathbb{R}^{3}}\frac{u(x)}{|x|^{2}}dx+9\int_{\substack{|y|\geq\sqrt{|x|}}}\frac{u(y)}{|y|^{2}}dy+C\epsilon^{2} \\ \nonumber
&\leq \frac{C\epsilon}{|x|}+o_{|x|}(1)+C\epsilon^{2}.
\end{align}
From \eqref{HC-5} and \eqref{HC-6}, we deduce that, for arbitrary $\epsilon>0$,
\begin{equation}\label{HC-7}
  \lim_{x\in\Sigma_{\lambda}^{-},\,|x|\rightarrow+\infty}\left\{\int_{\mathbb{R}^3}\frac{u^2(y)}{|x-y|}dy+2u(x)\int_{\Sigma_\lambda^-}\Big(\frac{1}{|x-y|}-\frac{1}{|x^{\lambda}-y|}\Big)u(y)dy\right\}\leq C\epsilon^{2}.
\end{equation}
This indicates \eqref{HC-4} holds and hence \eqref{HC-3} holds for any $\lambda\leq0$. Besides, from \eqref{HC-5} and \eqref{HC-6}, one can also derive that $c_{\lambda}(x)$ is uniformly bounded from below (independent of $\lambda$).

We will carry out the direct method of moving planes in two steps.

\emph{Step 1.} We use Theorem \ref{P:decay} (\emph{Decay at infinity (I)}) to show that, for sufficiently negative $\lambda$,
\begin{equation}\label{HC-aim}
  w_{\lambda}(x)\geq 0, \quad \forall \,\, x\in\Sigma_{\lambda}.
\end{equation}

In fact, from the assumption \eqref{HC-0}, we know that $u$ is bounded from above and
$$\lim_{|x|\rightarrow+\infty}u(x)=0,$$
and hence $w_{\lambda}$ is bounded from below and $\lim_{|x|\rightarrow+\infty}w_{\lambda}(x)=0$ for any $\lambda\in\mathbb{R}$. Now suppose that $\inf_{\Sigma_{\lambda}}w_{\lambda}<0$, then the negative minimal $\inf_{\Sigma_{\lambda}}w_{\lambda}<0$ can be attained in $\Sigma_{\lambda}^{-}$, i.e. there exists $\hat{x}\in\Sigma_{\lambda}^{-}$ such that $w_{\lambda}(\hat{x})=\inf_{\Sigma_{\lambda}}w_{\lambda}<0$. By \eqref{HC-2} and \eqref{HC-3}, we can deduce from Theorem \ref{P:decay} (\emph{Decay at infinity (I)}) (Remark \ref{rem14}) that, there exists $R_{0}>0$ large (independent of $\lambda$) such that $|\hat{x}|\leq R_{0}$. This will lead to a contradiction provided that $\lambda\leq-R_{0}$. Thus we have, for any $\lambda\leq-R_{0}$, $w_{\lambda}\geq0$ in $\Sigma_{\lambda}$.

\emph{Step 2.} Step 1 provides a starting point, from which we can now move the plane $T_\la$ along $x_{1}$-axis to the right as long as \eqref{HC-aim} holds to its limiting position.

To this end, let us define
\begin{equation}\label{HC-def}
  \lambda_{0}:=\sup\left\{\lambda\leq0 \mid w_{\mu}\geq 0 \,\, \text{in} \,\, \Sigma_{\mu}, \,\, \forall \,\mu \leq \lambda\right\}.
\end{equation}
It follows from Step 1 that $-R_{0}\leq\lambda_{0}\leq0$. One can easily verify that
\begin{equation}\label{HC-8}
  w_{\lambda_{0}}(x)\geq0, \quad \forall \,\, x\in\Sigma_{\lambda_{0}}.
\end{equation}

If $\lambda_{0}<0$, being quite similar to \emph{step 2} in the proof of Theorem \ref{T:sub_symm} in subsubsection 2.2.3, we can show via contradiction arguments, \emph{Narrow region principle} Theorem \ref{2NRP} and \emph{Decay at infinity (I)} Theorem \ref{P:decay} (see Remark \ref{rem13} and \ref{rem14}) that
\begin{equation}\label{HC-9}
  w_{\lambda_{0}}(x) \equiv 0, \quad \forall \,\, x \in \Sigma_{\lambda_{0}}.
\end{equation}

If $\lambda_0=0$, we move the plane in the opposite direction along $x_{1}$-direction until the limiting position $\bar{\lambda}_0\geq\lambda_0=0$. Again, if $\bar{\lambda}_0>0$, by using \emph{Narrow region principle} Theorem \ref{2NRP} and \emph{Decay at infinity (I)} Theorem \ref{P:decay} (see Remark \ref{rem13} and \ref{rem14}), we can deduce that $u_{\bar{\lambda}_{0}}\equiv u$ as in \emph{step 2} in the proof of Theorem \ref{T:sub_symm}. If $\bar{\lambda}_{0}=\lambda_{0}=0$, we immediately have $u_0\equiv u$. All in all, $u$ is symmetric w.r.t. the plane $T_{\lambda_0}$ or $T_{\bar{\lambda}_{0}}$. Since the $x_{1}$-direction is arbitrarily chosen, we must have $u$ is radially symmetric and monotone decreasing about some point $x_{0}\in\mathbb{R}^{N}$. This concludes our proof of Theorem \ref{RSO-HC-Symm}.
\end{proof}

\begin{rem}\label{rem2}
By applying the direct method of moving planes for $(-\Delta+m^{2})^{s}$ developed here, under appropriate decay conditions on nonnegative solution $u$ as $|x|\rightarrow+\infty$, we can also derive symmetry and monotonicity of nonnegative solutions to the following generalized pseudo-relativistic Hartree or Choquard equations
\begin{equation}\label{gRSO-HC}
\left(-\Delta+m^{2}\right)^{s}u+\omega u=\left(\frac{1}{|x|^{\gamma}}*u^{p}\right)u^{q} \qquad \text{in} \,\, \mathbb{R}^{N},
\end{equation}
where $\omega>-m^{2s}$. We leave these open problems to interested readers. For more literatures on Hartree or Choquard equations, please refer to \cite{CD,DFHQW,DFQ,DL,DLQ,DQ,DSS,Le1,Le,MS,MZ,W,Wu1} and the references therein.
\end{rem}

\section{Direct sliding methods for $(-\Delta+m^{2})^{s}$}
In this Section, inspired by \cite{CLiu,CW1,CW2,DSV}, we will establish various maximum principles and develop the direct sliding methods for pseudo-relativistic Schr\"{o}dinger operators $(-\Delta+m^{2})^{s}$ (with $s\in(0,1)$ and $m>0$) and apply it to investigate monotonicity and uniqueness of solutions to the following fractional order semi-linear equation
\begin{equation}\label{PDE}
  (-\Delta+m^{2})^s u(x)=f(u(x))
\end{equation}
in different type of regions, including bounded domains $\Omega$, epigraph $D$ and the whole-space $\mathbb{R}^{N}$.

\subsection{Narrow Region Principle}
Similar to the method of moving planes, the \emph{narrow region principle} is a key ingredient in the sliding method and it provides a starting position to slide the domain $\Omega$. Hence, in this subsection, we will establish the narrow region principle for $(-\Delta+m^{2})^{s}$.

First, we can prove the following strong maximum principle.
\begin{lem}(Strong maximum principle)\label{SMP-Ubdd}
Suppose that $u\in\mathcal{L}_{s}(\mathbb{R}^{N})$ and $u\geq0$ in $\mathbb{R}^{N}$. If there exists $x_{0}\in\mathbb{R}^{N}$ such that, $u(x_{0})=0$, $u$ is $C^{1,1}$ near $x_{0}$ and $(-\Delta+m^{2})^{s}u(x_{0})\geq0$, then $u=0$ a.e. in $\mathbb{R}^N$.
\end{lem}
\begin{proof}
	Since there exists $x_0\in\mathbb{R}^{N}$ such that $u(x_0)=\min_{x\in\mathbb{R}^N}u(x)=0$, it follows from \eqref{Definition} that
	\begin{align*}
	&0\leq\left(-\Delta+m^{2}\right)^{s}u(x_0)\\\nonumber
	&\quad =c_{N,s} m^{\Ns} P.V. \int_{\mathbb{R}^{N}} \frac{u(x_0)-u(y)}{|x_{0}-y|^{\Ns}} K_{\Ns}(m|x_{0}-y|)dy\\\nonumber
	&\quad =-c_{N,s} m^{\Ns} P.V. \int_{\mathbb{R}^{N}} \frac{u(y)}{|x_{0}-y|^{\Ns}} K_{\Ns}(m|x_{0}-y|)dy\leq0.
	\end{align*}
	Thus we must have $u=0$ a.e. in $\mathbb{R}^{N}$. This finishes the proof of Lemma \ref{SMP-Ubdd}.
\end{proof}

We have the following narrow region principle for $(-\Delta+m^{2})^{s}$.
\begin{thm}[Narrow region principle]\label{thm1}
Let $D$ be a bounded narrow region in $\mathbb{R}^N$ and $d_N(D)$ be the width of $D$ in $x_N$ direction. Suppose that $w\in {{\mathcal{L}}_{s}(\mathbb{R}^N)} \cap C_{loc}^{1,1}(D)$ is lower semi-continuous on $\overline{D}$, and satisfies
\begin{eqnarray}\label{3NRP-1}\left\{\begin{array}{ll}
 (-\Delta+m^{2})^s{w}(x) + c(x){w}(x) \geq 0 &\quad \text{at points} \,\, x\in D \,\, \text{where} \,\, w(x)<0, \\
{w}(x)\geq 0 &\quad \mbox{ in } D^c,
\end{array} \right. \end{eqnarray}
where $c(x)$ is uniformly bounded from below (w.r.t. $d_{N}(D)$) in $\left\{x\in D \mid w(x)<0\right\}$. We assume $D$ is narrow in the sense that, $d_N(D)\leq\frac{r_{0}}{4m}$, and there exists a constant $C_{N,s}>0$ such that
\begin{eqnarray}\label{3NRP-2}
\left(-\inf_{\left\{x\in D \mid w(x)<0\right\}}c(x)-m^{2s}\right)d_N(D)^{2s}<C_{N,s}.
\end{eqnarray}
Then
\begin{eqnarray}\label{3NRP-3}
 w(x) \geqslant 0 \quad \mbox{ in } D.
\end{eqnarray}
Furthermore, assume that
\begin{equation}\label{3NRP-7}
  (-\Delta+m^{2})^s{w}(x)\geq0 \quad \text{at points} \,\, x\in D \,\, \text{where} \,\, w(x)=0,
\end{equation}
then we have
\begin{eqnarray}\label{3NRP-4}
 \mbox{ either } \quad w(x) > 0 \mbox{ in } D \quad \mbox{ or } \quad w(x)=0 \mbox{ a.e. in } \mathbb{R}^N. \end{eqnarray}
\end{thm}
\begin{proof}
Suppose \eqref{3NRP-3} is not valid, then the lower semi-continuity of $w$ in $\bar{D}$ implies that there exists a point $\bar{x}\in D$ such that
\begin{eqnarray}\label{3NRP-5}
 w(\bar{x})=\min_{\bar{D}}w(x) <0.
\end{eqnarray}
Since $d_N(D)\leq\frac{r_{0}}{4m}$, by \eqref{3NRP-1} and \eqref{3NRP-2}, we have
\begin{eqnarray}\label{3NRP-6}
 &&(-\Delta+m^{2})^s w(\bar{x})+c(\bar{x})w(\bar{x}) \\
 &=&c_{N,s}m^{\frac{N}{2}+s}P.V.\int_{\mathbb{R}^N}\frac{w(\bar{x})-w(y)}{|\bar{x}-y|^{\frac{N}{2}+s}}K_{\frac{N}{2}+s}(m|\bar{x}-y|)dy
 +m^{2s}w(\bar{x})+c(\bar{x})w(\bar{x})\nonumber\\
 &\leq& c_{N,s}m^{\frac{N}{2}+s}w(\bar{x})\int_{D^c}\frac{K_{\frac{N}{2}+s}(m|\bar{x}-y|)}{|\bar{x}-y|^{\frac{N}{2}+s}}dy+\left(\inf_{\{x\in D \mid w(x)<0\}} c(x)+m^{2s}\right)w(\bar{x})\nonumber\\
 &\leq& \left[c_{0}c_{N,s}\int_{\mathcal{C}}\frac{1}{|\bar{x}-y|^{N+2s}}dy+\left(\inf_{\{x\in D \mid w(x)<0\}} c(x)+m^{2s}\right)\right]w(\bar{x})\nonumber\\
 &\leq& \left[c_{0}c_{N,s}c_{N}\sigma_{N}\int_{2d_{N}(D)}^{4d_{N}(D)}\frac{1}{r^{1+2s}}dr+\left(\inf_{\{x\in D \mid w(x)<0\}} c(x)+m^{2s}\right)\right]w(\bar{x})\nonumber\\
 &\leq& w(\bar{x})\left[\frac{C_{N,s}}{[d_N(D)]^{2s}}+\left(\inf_{\{x\in D \mid w(x)<0\}} c(x)+m^{2s}\right)\right]<0, \nonumber
\end{eqnarray}
where $\mathcal{C}:=\{x\in\mathbb{R}^{N}\,|\,2d_{N}(D)<|x-\bar{x}|<4d_{N}(D),\,x_{1}-(\bar{x})_{1}>\frac{1}{2}|x-\bar{x}|\}\subset D^{c}$.

Inequality \eqref{3NRP-6} contradicts \eqref{3NRP-1}, and we thus conclude that
$$w(x)\geq 0, \qquad \forall \,x \in D.$$

Furthermore, if there exists a point $\tilde{x}\in D$ such that $w(\tilde{x})=0$, then it follows immediately from \eqref{3NRP-7} and Lemma \ref{SMP-Ubdd} that $w=0$ a.e. in $\mathbb{R}^N$. Therefore, we have
$$\mbox{ either } \quad {w}(x) > 0 \mbox{ in } D \quad \mbox{ or } \quad w(x)=0 \mbox{ a.e. in } \mathbb{R}^N.$$
This completes the proof of Theorem \ref{thm1}.
\end{proof}
\begin{rem}\label{rem16}
It is clear from the proof that, in Theorem \ref{thm1}, the assumptions ``$w$ is lower semi-continuous on $\overline{D}$" and ``$w\geq0$ in $D^{c}$" can be weaken into: ``if $w<0$ somewhere in $\mathbb{R}^{N}$, then the negative minimum $\inf_{\mathbb{R}^{N}}w(x)$ can be attained in $D$", the same conclusions are still valid. It can also be seen from the proof that, the assumption ``$(-\De+m^{2})^s w(x)+c(x)w(x)\geq 0$ at points $x\in D$ where $w(x)<0$" can be weaken into: ``$(-\De+m^{2})^s w(x)+c(x)w(x)\geq 0$ at points $x\in D$ where $w(x)=inf_{\mathbb{R}^{N}}w<0$", the same conclusions in Theorems \ref{thm1} are still valid.
\end{rem}

\subsection{Maximum principles in unbounded open sets and immediate applications}
In order to apply the sliding method on unbounded open sets, maximum principles play an important role.

We prove
\begin{thm}(Maximum principle in unbounded open sets)\label{MP-Ubdd}
Let $D$ be an open set in $\mathbb{R}^N$, possibly unbounded and disconnected. Suppose that $u\in \mathcal{L}_{s}(\mathbb{R}^{N})\cap C_{\text {loc}}^{1,1}(D)$ is bounded from above, and satisfies
	\begin{equation}\label{MP2-2}
	\begin{cases} \left(-\Delta+m^{2}\right)^{s} u(x)+c(x)u(x)\leq 0 \quad \text{at points} \,\, x\in D \,\, \text{where} \,\,u(x)>0,\\ u(x)\leq 0, \quad x\in \mathbb{R}^{N}\setminus D,
	\end{cases}
	\end{equation}
where $c(x)$ satisfies
\begin{equation}\label{MP2-assumption}
  \inf_{\{x\in D \mid u(x)>0\}}c(x)>-m^{2s}.
\end{equation}
Then $u\leq 0$ in $D$.

Furthermore, assume that
\begin{equation}\label{MP2-CON}
  (-\Delta+m^{2})^s{u}(x)\leq0 \quad \text{at points} \,\, x\in D \,\, \text{where} \,\, u(x)=0,
\end{equation}
then we have
\begin{eqnarray}\label{3NRP-4}
 \mbox{ either } \quad u(x)<0 \mbox{ in } D \quad \mbox{ or } \quad u(x)=0 \mbox{ a.e. in } \mathbb{R}^N. \end{eqnarray}
\end{thm}
\begin{rem}\label{rem-MP-Ubdd}
One can observe that, we allow the function $c(x)$ in Theorem \ref{MP-Ubdd} to be negative, which is different from corresponding assumptions on $c(x)$ in those maximum principles in unbounded open sets for fractional Laplacians $(-\Delta)^{s}$ or fractional $p$-Laplacians $(-\Delta)^{s}_{p}$ ($0<s<1$) (see, e.g. Theorem 1 in \cite{CLiu} and Theorem 2 in \cite{CW2}).
\end{rem}
\begin{proof}
	Suppose on the contrary that there exists one point $x\in D$ such that $u(x)>0$, then we have
	\begin{equation}\label{MP2-3}
	0<M:=\sup_{x\in \mathbb{R}^{N}} u(x) <\infty.
	\end{equation}
	There exists sequences $x^k\in D$ and $0<\alpha_k<1$ with $\alpha_k\rightarrow 1$ as $k\rightarrow \infty$ such that
	\begin{equation}\label{MP2-4}
	u(x^k)\geq \alpha_k M.
	\end{equation}
	Let
	\begin{equation*}
	\gamma(x)=\begin{cases}e^{\frac{|x|^{2}}{|x|^2-1}}, \qquad |x|<1\\ 0, \qquad \qquad\, |x|\geq 1.\end{cases}
	\end{equation*}
	It is well known that $\gamma\in C_0^\infty(\mathbb{R}^N)$, therefore $|\left(-\Delta+m^{2}\right)^{s}\gamma(x)|\leq C_0$ for any $x \in \mathbb{R}^N$. Moreover, $\left(-\Delta+m^{2}\right)^{s}\gamma(x)\sim|x|^{-\frac{N+1}{2}-s}e^{-|x|}$ as $|x|\rightarrow +\infty$.
	
	Define
$$ \Gamma_k(x):=\gamma(x-x^k).$$
We can take $\epsilon_k>0$ such that, for any $x\in \mathbb{R}^N\setminus B_{1}(x^k)$,
	\begin{equation}\label{MP2-5}
	u(x^k)+\epsilon_k\Gamma_{k}(x^k)\geq M\geq u(x)+\epsilon_k\Gamma_{k}(x).
	\end{equation}
	In fact, we only need to let
$$\epsilon_k:=(1-\alpha_k)M.$$
Consequently, there exists $\bar{x}^k\in B_{1}(x^k)$ such that
	\begin{equation}\label{MP2-6}
	u(\bar{x}^k)+\epsilon_k\Gamma_{k}(\bar{x}^k)= \max_{x\in\mathbb{R}^{N}}[u(x)+\epsilon_k\Gamma_{k}(x)]\geq M,
	\end{equation}
	which also implies that
\begin{equation}\label{MP2-7}
  u(\bar{x}^k)\geq u(x^k)+\epsilon_k\Gamma_{k}(x^k)-\epsilon_k\Gamma_{k}(\bar{x}^k)\geq u(x^k)\geq \alpha_{k}M>0.
\end{equation}

	Therefore, we deduce from \eqref{MP2-6} that
	\begin{align}\label{MP2-8}
	&\quad\left(-\Delta+m^{2}\right)^{s}[u+\epsilon_k\Gamma_{k}](\bar{x}^k)\\ \nonumber
&=c_{N,s}m^{\Ns} P.V. \int_{\mathbb{R}^{N}} \frac{[u(\bar{x}^k)+\epsilon_k\Gamma_{k}(\bar{x}^k)]-[u(y)+\epsilon_k\Gamma_{k}(y)]}{|\bar{x}^k-y|^{\Ns}} K_{\Ns}(m|\bar{x}^k-y|) dy\\ \nonumber
&\quad+m^{2s}[u(\bar{x}^k)+\epsilon_k\Gamma_{k}(\bar{x}^k)]\\ \nonumber
	&\geq m^{2s}[u(\bar{x}^k)+\epsilon_k\Gamma_{k}(\bar{x}^k)].
	\end{align}
	
	Next, we will evaluate the upper bound of $\left(-\Delta+m^{2}\right)^{s}[u+\epsilon_k\Gamma_{k}](\bar{x}^k)$.
	
Indeed, since \eqref{MP2-7} implies $u(\bar{x}^{k})>0$ and hence $\bar{x}^{k}\in D$, we conclude from \eqref{MP2-2} and  $|\left(-\Delta+m^{2}\right)^{s}\Gamma_{k}(x)|\leq C_0$ for any $x\in\mathbb{R}^{N}$ that
	\begin{equation}\label{MP2-9}
	\left(-\Delta+m^{2}\right)^{s}[u+\epsilon_k\Gamma_{k}](\bar{x}^k)\leq -c(\bar{x}^{k})u(\bar{x}^k)+C_0\epsilon_{k}.
	\end{equation}
	
	Combining \eqref{MP2-8} and \eqref{MP2-9}, we derive
	\begin{eqnarray}\label{MP2-11}
	&&m^{2s}u(\bar{x}^k)\leq m^{2s}[u(\bar{x}^k)+\epsilon_k\Gamma_{k}(\bar{x}^k)]\leq -c(\bar{x}^{k})u(\bar{x}^k)+C_{0}(1-\alpha_{k})M \\
    \nonumber &&\qquad\qquad\, \leq -\left(\inf_{\{x\in D \mid u(x)>0\}}c(x)\right)u(\bar{x}^k)+C_{0}(1-\alpha_{k})\frac{u(\bar{x}^k)}{\alpha_{k}}.
	\end{eqnarray}
Thus if $k$ is large enough such that $\alpha_{k}>\frac{1}{2}$, then \eqref{MP2-11} yields that
$$m^{2s}\leq-\inf_{\{x\in D \mid u(x)>0\}}c(x)+2C_{0}(1-\alpha_{k}).$$
Due to assumption \eqref{MP2-assumption}, this will lead to a contradiction if we let $k\rightarrow+\infty$.

Furthermore, if there exists a point $\tilde{x}\in D$ such that $u(\tilde{x})=0$, then it follows immediately from \eqref{MP2-CON} and Lemma \ref{SMP-Ubdd} that $u=0$ a.e. in $\mathbb{R}^N$. Therefore, we have
$$\mbox{ either } \quad {u}(x)<0 \mbox{ in } D \quad \mbox{ or } \quad u(x)=0 \mbox{ a.e. in } \mathbb{R}^N.$$
This completes our proof of Theorem \ref{MP-Ubdd}.
\end{proof}
\begin{rem} \label{eg} 
For fractional Laplacians $(-\Delta)^{s}$ ($0<s<1$), Dipierro, Soave and Valdinoci proved in \cite{DSV} Maximum Principles in unbounded open set $D$ by using Silvestre's growth lemma (\cite{S}) under the {\em exterior cone condition} that the complement of $D$ contains an infinite open connected cone $\Sigma$. Subsequently, Chen and Liu \cite{CLiu}, Chen and Wu \cite{CW2} introduced new ideas in the proof and thus significantly weakens the {\em exterior cone condition} to the following condition:
\begin{eqnarray}\label{domain}
 \mathop{\underline{lim}}\limits_{k\rightarrow \infty}\frac{|D^c\cap (B_{2^{k+1}r}(q)\backslash B_{2^kr}(q))|}{|B_{2^{k+1}r}(q)\backslash B_{2^kr}(q)|}=c_{0}>0, \qquad \forall \,\, q\in D
 \end{eqnarray}
for some $c_{0}>0$ and $r>0$. Typical examples of $D$ which satisfy condition (\ref{domain}) but does not satisfy the {\em exterior cone condition} include: stripes, annulus and Archimedean spiral (refer to \cite{CLiu,CW2} for details).

Being essentially different from fractional Laplacians $(-\Delta)^{s}$ ($s\in(0,1)$), since there is a \emph{modified Bessel function of the second kind $K_{\frac{N}{2}+s}$} in the definition of inhomogeneous fractional operators $(-\Delta+m^{2})^{s}$, $(-\Delta+m^{2})^{s}$ ($s\in(0,1)$) \emph{do not possess} any \emph{invariance properties} under \emph{scaling transforms}. In order to circumvent these difficulties, we need to derive upper and lower bounds on $\left(-\Delta+m^{2}\right)^{s}[u+\varepsilon_k\Phi_{k}](\bar{x}^k)$ (see \eqref{MP2-8} and \eqref{MP2-9} in the proof of Theorem \ref{MP-Ubdd}). One should note that, \emph{we do not impose any assumptions on $D$ in our Theorem \ref{MP-Ubdd}, $D\subseteq\mathbb{R}^{N}$ can be any open set (possibly unbounded and disconnected)}.
\end{rem}

We consider the following equation
\begin{equation}\label{MP-11}
\begin{cases}\left(-\Delta+m^{2}\right)^{s} u(x)=f(u(x)), \quad \forall \,\, x\in \Omega,\\ u(x)=\phi(x), \quad \forall \,\, x\in \mathbb{R}^{N}\setminus\Omega.\end{cases}
\end{equation}
As an application of Theorem \ref{MP-Ubdd} and Lemma \ref{SMP-Ubdd}, we can prove the following theorem.
\begin{thm}\label{thm-ub}
Let $u\in C_{\text {loc}}^{1,1}(\Omega)\cap\mathcal{L}_{s}(\mathbb{R}^{N})$ be a solution of \eqref{MP-11} such that $u$ is bounded from above and
$$u(x)=\phi(x)<\mu, \quad \forall \,\, x\in \mathbb{R}^{N}\setminus\Omega.$$
Assume $f$ satisfies $f(t)\leq m^{2s}\mu$ for $t\geq \mu$ and $f(\mu)=m^{2s}\mu$.
	
Then $u<\mu$ in $\Omega$.
\end{thm}
\begin{proof}
We first show that $u\leq \mu$ in $\Omega$. To this end, define $w(x)=u(x)-\mu$, then $w$ is bounded from above and $w(x)<0$ in $\mathbb{R}^{N}\setminus\Omega$. From equation \eqref{MP-11}, we infer that, at points $x\in\Omega$ where $u(x)>\mu$,
	\begin{equation}
	\left(-\Delta+m^{2}\right)^{s} w(x)=f(u(x))-m^{2s}\mu=f(u(x))-f(\mu)\leq 0.
	\end{equation}
It follows from Theorem \ref{MP-Ubdd} that $w(x)\leq 0$ in $\Omega$. Thus we arrive at $u\leq\mu$ in $\Omega$.

Furthermore, by strong maximum principle Lemma \ref{SMP-Ubdd}, we conclude that $u<\mu$ in $\Omega$. This finishes the proof of Theorem \ref{thm-ub}.
\end{proof}

We can also prove the following narrow region principle in unbounded open sets.
\begin{thm}[Narrow region principle in unbounded open sets]\label{32NRP}
Let $D$ be a open set in $\mathbb{R}^{N}$ (possibly unbounded and disconnected) and $d(D):=\sup_{x\in D}dist(x,D^{c})$ be the width of $D$. Assume that $D$ satisfies
\begin{equation}\label{32NRP-con}
\frac{\big|\big(B_{c_{2}r_{x}}(x)\setminus{B_{r_{x}}(x)}\big)\cap D^{c}\big|}
{|B_{c_{2}r_{x}}(x)\setminus B_{r_{x}}(x)|}\geq c_{1}, \qquad \forall \,\, x\in D
\end{equation}
for some constants $c_{1}>0$ and $c_{2}>1$, where $r_{x}=dist(x,D^{c})\leq d(D)$. Suppose that $w\in {{\mathcal{L}}_{s}(\mathbb{R}^{N})}\cap C_{loc}^{1,1}(D)$ is bounded from below and satisfies
\begin{eqnarray}\label{32NRP-1}\left\{\begin{array}{ll}
 (-\Delta+m^{2})^s{w}(x)+c(x){w}(x) \geq 0 &\quad \text{at points} \,\, x\in D \,\, \text{where} \,\, w(x)<0, \\
{w}(x)\geq 0 &\quad \mbox{ in } D^c,
\end{array} \right. \end{eqnarray}
where $c(x)$ is uniformly bounded from below (w.r.t. $d(D)$) in $\left\{x\in D \mid w(x)<0\right\}$. If we assume that $d(D)\leq\frac{r_{0}}{c_{2}m}$, and either
\begin{equation}\label{32NRP-5}
  \inf_{\{x\in D \mid w(x)<0\}}\,c(x)>-\left\{1+\frac{c_{2}^{2s}C_{N,s}}{r_{0}^{2s}}\left[\frac{1}{c_{2}(1+c_{2})}\right]^{N+2s}\left(c_{2}^{N}-1\right)c_{1}\right\}m^{2s},
\end{equation}
or
\begin{eqnarray}\label{32NRP-2}
-[d(D)]^{2s}\left(\inf_{\{x\in D \mid w(x)<0\}}\, c(x)+m^{2s}\right)<C_{N,s}\left[\frac{1}{c_{2}(1+c_{2})}\right]^{N+2s}\left(c_{2}^{N}-1\right)c_{1}
\end{eqnarray}
for some constant $C_{N,s}>0$. Then
\begin{eqnarray}\label{32NRP-3}
 w(x) \geqslant 0 \quad \mbox{ in } D.
\end{eqnarray}
Furthermore, assume that
\begin{equation}\label{32NRP-7}
  (-\Delta+m^{2})^s{w}(x)\geq0 \quad \text{at points} \,\, x\in D \,\, \text{where} \,\, w(x)=0,
\end{equation}
then we have
\begin{eqnarray}\label{32NRP-4}
 \mbox{ either } \quad w(x) > 0 \mbox{ in } D \quad \mbox{ or } \quad w(x)=0 \mbox{ a.e. in } \mathbb{R}^N. \end{eqnarray}.
\end{thm}
\begin{proof}
Suppose on the contrary, there is some point $x$ such that $w(x)<0$ in $D$, then
\begin{equation}\label{32NRP-6}
-\infty<-L:=\inf_{x\in\mathbb{R}^{N}}w(x)<0.
\end{equation}
There exist sequences $\{x^k\}\subseteq D$ and $\{\theta_k\}\subseteq(0,1)$ with $\lim_{k\rightarrow\infty}\theta_k\rightarrow1$ such that
\begin{equation}\label{32NRP-8}
-L\leq w(x^k)\leq-\theta_k L.
\end{equation}
Let
	\begin{equation}\label{32NRP-9}
	\Psi(x)=\begin{cases}e^{\frac{|x|^{2}}{|x|^2-1}}, \qquad |x|<1\\ 0, \qquad \qquad\, |x|\geq 1.\end{cases}
	\end{equation}
	It is well known that $\Psi\in C_0^\infty(\mathbb{R}^N)$, therefore $|\left(-\Delta+m^{2}\right)^{s}\Psi(x)|\leq C_0$ for any $x \in \mathbb{R}^N$. Moreover, $\left(-\Delta+m^{2}\right)^{s}\Psi(x)$ is monotone decreasing with respect to $|x|$ and $\left(-\Delta+m^{2}\right)^{s}\Psi(x)\sim|x|^{-\frac{N+1}{2}-s}e^{-|x|}$ as $|x|\rightarrow +\infty$.

Define
\begin{equation}\label{32NRP-10}
\Psi_{k}(x):=\Psi\left(\frac{x-x^k}{d_k}\right),
\end{equation}
where $d_{k}:=dist(x^{k},D^{c})$. We can pick $\mu_k:=(1-\theta_{k})L>0$ such that, for any $x\in\mathbb{R}^{N}\setminus{B_{d_k}(x^k)}$,
\begin{equation}\label{32NRP-11}
w(x^k)-\mu_k\Psi_{k}(x^k)\leq -L\leq w(x)-\mu_k\Psi_{k}(x).
\end{equation}
Therefore, there exists $\bar{x}^k\in B_{d_k}(x^k)$ such that
\begin{equation}\label{32NRP-12}
w(\bar{x}^k)-\mu_k\Psi_{k}(\bar{x}^k)=\min_{x\in\mathbb{R}^{N}}[w(x)-\mu_k\Psi_{k}(x)].
\end{equation}
As a consequence, one has
$$w(\bar{x}^k)-\mu_k\Psi_{k}(\bar{x}^k)\leq w(x^k)-\mu_k\Psi_{k}(x^k)\leq-L,$$
and hence
\begin{equation}\label{32NRP-13}
  w(\bar{x}^k)\leq w(x^k)-\mu_k\Psi_{k}(x^k)+\mu_k\Psi_{k}(\bar{x}^k)\leq w(x^k)\leq-\theta_{k}L<0.
\end{equation}

Next, we will estimate the upper bound and lower bound for $(-\Delta+m^{2})^{s}\left[w-\mu_{k}\Psi_{k}\right](\bar{x}^{k})$.

By definition and direct calculations, we have
\begin{equation}\label{32NRP-14}\begin{split}
&\quad (-\Delta+m^{2})^s\left[w-\mu_{k}\Psi_{k}\right](\bar{x}^{k})\\
&=c_{N,s}m^{\frac{N}{2}+s}P.V.\int_{\mathbb{R}^{N}}
\frac{w(\bar{x}^{k})-\mu_{k}\Psi_{k}(\bar{x}^{k})-w(y)+\mu_{k}\Psi_{k}(y)}{|\bar{x}^k-y|^{\frac{N}{2}+s}}K_{\frac{N}{2}+s}(m|\bar{x}^k-y|)dy \\
&\quad +m^{2s}\left(w-\mu_{k}\Psi_{k}\right)(\bar{x}^{k})\\
&\leq c_{N,s}m^{\frac{N}{2}+s}\int_{\mathbb{R}^{N}\setminus{B_{d_k}(x^k)}}\frac{w(\bar{x}^{k})-\mu_{k}\Psi_{k}(\bar{x}^{k})-w(y)}{|\bar{x}^k-y|^{\frac{N}{2}+s}}K_{\frac{N}{2}+s}(m|\bar{x}^k-y|)dy\\
&\quad +m^{2s}\left(w-\mu_{k}\Psi_{k}\right)(\bar{x}^{k}).
\end{split}\end{equation}
From the assumption \eqref{32NRP-con}, note that $d(D)\leq\frac{r_{0}}{c_{2}m}$, we get the following estimate:
\begin{equation}\label{32NRP-15}\begin{split}
&\quad c_{N,s}m^{\frac{N}{2}+s}\int_{\mathbb{R}^{N}\setminus{B_{d_k}(x^k)}}\frac{w(\bar{x}^{k})-\mu_{k}\Psi_{k}(\bar{x}^{k})-w(y)}{|\bar{x}^k-y|^{\frac{N}{2}+s}}K_{\frac{N}{2}+s}(m|\bar{x}^k-y|)dy\\
&\leq c_{N,s}m^{\frac{N}{2}+s}\int_{D^{c}}\frac{K_{\frac{N}{2}+s}(m|\bar{x}^k-y|)}{|\bar{x}^k-y|^{\frac{N}{2}+s}}dy
\left[w(\bar{x}^{k})-\mu_{k}\Psi_{k}(\bar{x}^{k})\right]\\
&\leq c_{N,s}c_{0}\int_{D^{c}\cap\left(B_{c_{2}d_{k}}(x_{k})\setminus B_{d_{k}}(x_{k})\right)}\frac{1}{|\bar{x}^k-y|^{N+2s}}dy
\left[w(\bar{x}^{k})-\mu_{k}\Psi_{k}(\bar{x}^{k})\right]\\
&\leq c_{N,s}c_{0}\left(\frac{1}{1+c_{2}}\right)^{N+2s}\int_{D^{c}\cap\left(B_{c_{2}d_{k}}(x_{k})\setminus B_{d_{k}}(x_{k})\right)}\frac{1}{|x^k-y|^{N+2s}}dy\cdot w_{k}(\bar{x}^{k})\\
&\leq c_{N,s}c_{0}\left[\frac{1}{c_{2}(1+c_{2})}\right]^{N+2s}\frac{|\left(B_{c_{2}d_{k}}(x_{k})\setminus B_{d_{k}}(x_{k})\right)\cap D^{c}|}{d_{k}^{N+2s}}\cdot w_{k}(\bar{x}^{k}) \\
&\leq C_{N,s}\left[\frac{1}{c_{2}(1+c_{2})}\right]^{N+2s}\left(c_{2}^{N}-1\right)\frac{|\left(B_{c_{2}d_{k}}(x_{k})\setminus B_{d_{k}}(x_{k})\right)\cap D^{c}|}{|B_{c_{2}d_{k}}(x_{k})\setminus B_{d_{k}}(x_{k})|}\cdot\frac{w_{k}(\bar{x}^{k})}{d_{k}^{2s}} \\
&\leq C_{N,s}\left[\frac{1}{c_{2}(1+c_{2})}\right]^{N+2s}\frac{\left(c_{2}^{N}-1\right)c_{1}}{d_{k}^{2s}}w_{k}(\bar{x}^{k}),
\end{split}\end{equation}
where we have used the notations
\begin{equation*}
w_{k}(x):=w(x)-\mu_{k}\Psi_{k}(x), \qquad k=1,2,3,\cdots.
\end{equation*}
Consequently, combining \eqref{32NRP-14} with \eqref{32NRP-15} yield the upper bound:
\begin{eqnarray}\label{32NRP-ub}
 && \quad (-\Delta+m^{2})^s\left[w-\mu_{k}\Psi_{k}\right](\bar{x}^{k})\\
 \nonumber &&\leq C_{N,s}\left[\frac{1}{c_{2}(1+c_{2})}\right]^{N+2s}\frac{\left(c_{2}^{N}-1\right)c_{1}}{d_{k}^{2s}}w_{k}(\bar{x}^{k})+m^{2s}w_{k}(\bar{x}^{k}).
\end{eqnarray}

On the other hand, note that $d(D)\leq\frac{r_{0}}{c_{2}m}$, by \eqref{32NRP-1} and Lemma \ref{lem-MP}, we can deduce the following lower bound:
\begin{equation}\label{32NRP-lb}
(-\Delta+m^{2})^s\left[w-\mu_{k}\Psi_{k}\right](\bar{x}^{k})\geq-c(\bar{x}^{k})w(\bar{x}^{k})-\frac{C}{d_{k}^{2s}}\mu_{k}.
\end{equation}

As a consequence of \eqref{32NRP-ub} and \eqref{32NRP-lb}, we arrive at, for any $k\geq1$,
\begin{equation}\label{32NRP-16}
-c(\bar{x}^{k})w(\bar{x}^{k})-\frac{C}{d_{k}^{2s}}\mu_{k}\leq C_{N,s}\left[\frac{1}{c_{2}(1+c_{2})}\right]^{N+2s}\frac{\left(c_{2}^{N}-1\right)c_{1}}{d_{k}^{2s}}w_{k}(\bar{x}^{k})+m^{2s}w_{k}(\bar{x}^{k}).
\end{equation}
Then it follows immediately that
\begin{equation}\label{32NRP-17}
 -\left\{C_{N,s}\left[\frac{1}{c_{2}(1+c_{2})}\right]^{N+2s}\frac{\left(c_{2}^{N}-1\right)c_{1}}{d_{k}^{2s}}+c(\bar{x}^{k})+m^{2s}\right\}w(\bar{x}^{k})
\leq \frac{C}{d_{k}^{2s}}\mu_{k}.
\end{equation}
Note that $\theta_{k}L\leq -w(\bar{x}^{k})\leq L$ and $\mu_{k}=(1-\theta_{k})L$, we infer from \eqref{32NRP-17} that
\begin{equation}\label{32NRP-18}
 C_{N,s}\left[\frac{1}{c_{2}(1+c_{2})}\right]^{N+2s}\frac{\left(c_{2}^{N}-1\right)c_{1}}{d_{k}^{2s}}+\inf_{\{x\in D \mid w(x)<0\}}c(x)+m^{2s}
\leq\frac{C}{d_{k}^{2s}}\frac{1-\theta_{k}}{\theta_{k}}.
\end{equation}
Now we choose $k$ large enough such that $\theta_{k}>\frac{1}{2}$ and
$$1-\theta_k<\frac{C_{N,s}}{2C}\left[\frac{1}{c_{2}(1+c_{2})}\right]^{N+2s}\left(c_{2}^{N}-1\right)c_{1},$$
then we have
\begin{eqnarray}\label{32NRP-c}
  &&\quad C_{N,s}\left[\frac{1}{c_{2}(1+c_{2})}\right]^{N+2s}\left(c_{2}^{N}-1\right)c_{1}-2C(1-\theta_{k}) \\
 \nonumber &&\leq-[d(D)]^{2s}\left(\inf_{\{x\in D \mid w(x)<0\}}c(x)+m^{2s}\right).
\end{eqnarray}

From \eqref{32NRP-c}, we can deduce the following:

\emph{(i).} If
\begin{equation}\label{32NRP-19}
  \inf_{\{x\in D \mid w(x)<0\}}\,c(x)>-\left\{1+\frac{c_{2}^{2s}C_{N,s}}{r_{0}^{2s}}\left[\frac{1}{c_{2}(1+c_{2})}\right]^{N+2s}\left(c_{2}^{N}-1\right)c_{1}\right\}m^{2s},
\end{equation}
since $0<d(D)\leq\frac{r_{0}}{c_{2}m}$, then \eqref{32NRP-c} yields a contradiction if we let $k\rightarrow+\infty$.

\emph{(ii).} If we only assume $c(x)$ is bounded from below in $\{x\in D \mid w(x)<0\}$, then we can also get a contradiction from \eqref{32NRP-c} by taking the limit $k\rightarrow+\infty$, provided that $D$ is narrow enough in the sense that its width $d(D)$ satisfies
\begin{equation}\label{32NRP-20}
  -[d(D)]^{2s}\left(\inf_{\{x\in D \mid w(x)<0\}}\, c(x)+m^{2s}\right)<C_{N,s}\left[\frac{1}{c_{2}(1+c_{2})}\right]^{N+2s}\left(c_{2}^{N}-1\right)c_{1}.
\end{equation}

Furthermore, if there exists a point $\tilde{x}\in D$ such that $w(\tilde{x})=0$, then it follows immediately from \eqref{32NRP-7} and Lemma \ref{SMP-Ubdd} that $w=0$ a.e. in $\mathbb{R}^N$. Therefore, we have
$$\mbox{ either } \quad {w}(x) > 0 \mbox{ in } D \quad \mbox{ or } \quad w(x)=0 \mbox{ a.e. in } \mathbb{R}^N.$$
This completes the proof of Theorem \ref{32NRP}.
\end{proof}
\begin{rem}\label{rem17}
Theorem \ref{32NRP} improves the \emph{Narrow Region Principle} Theorem \ref{thm1} at least in two aspects. First, we allow the open set $D$ to be unbounded or disconnected. Second, $D$ is not required to be narrow in certain direction, and hence more general sets such as ``narrow annulus" and ``narrow Archimedean spirals" ... are admissible in Theorem \ref{32NRP}.
\end{rem}

\subsection{Direct sliding methods and its applications}
By using various maximum principles established in subsections 3.1 and 3.2, we will apply the direct sliding method to investigate monotonicity and uniqueness of solutions to various problems involving the pseudo-relativistic Schr\"{o}dinger operators $(-\Delta+m^{2})^{s}$.

\subsubsection{Bounded domain $\Omega$}
Assume $u\in\mathcal{L}_{s}(\mathbb{R}^{N})\cap C_{loc}^{1,1}(\Omega)\cap C(\overline{\Omega})$ is a solution to equation \eqref{PDE} in $\Omega$. For any $x=(x', {x_N})$ with $x':= ({x_1},...,{x_{N-1}}) \in {\mathbb{R}^{N-1}}$ and $\tau\in \mathbb{R}$, let
$$u^\tau(x):=u(x', x_N+\tau)$$
and
$$w^\tau(x):=u^\tau(x)-u(x).$$
For a bounded region $\Omega$, we define
$$\Omega^\tau:=\Omega-\tau e_N \mbox{ with } e_N =(0, \cdots, 0, 1),$$
which is obtained by sliding $\Omega$ downward $\tau$ units.

It is obvious that when $\tau$ is sufficiently close to the width of $\Omega$ in $x_N$-direction, $\Omega \cap \Omega^\tau$ is a narrow region. Then under some monotonicity conditions on the solution $u$ in the complement of $\Omega$, we will be able to apply the Narrow region principle (Theorem \ref{thm1}) to conclude that
\begin{eqnarray}\label{331-1}
w^\tau(x)\geq 0,\;\; x \in \Omega \cap \Omega^\tau.
\end{eqnarray}

This provides a starting position to slide the domain $\Omega^\tau$. Then in the second step, we slide $\Omega^\tau$ back upward as long as inequality \eqref{331-1} holds to its limiting position. If we can slide the domain all the way to $\tau=0$, then we conclude that the solution is monotone increasing in
$x_N$-direction.

In order to guarantee the above two steps, we need to impose the exterior condition on $u$. Assume
\begin{eqnarray}\label{331-2}
u(x)=\varphi(x) \;\mbox{  in  } \;\Omega^c,
\end{eqnarray}
and the following monotonicity hypothesis:\\
{\textbf{(H):}} For any three points $x=(x', x_N), y=(x', y_N)$ and $z=(x', z_N)$ lying on a segment parallel to the $x_N$-axis, $y_N<x_N<z_N,$ with $y, z \in \Omega^c$, we have
\begin{eqnarray}\label{H-1}
\varphi(y)<u(x)<\varphi(z),\;\; \mbox{ if } x\in \Omega
\end{eqnarray}
and
\begin{eqnarray}\label{H-2}
\varphi(y)\leq \varphi(x)\leq \varphi(z), \;\;\mbox{ if } x\in \Omega^c.
\end{eqnarray}
\begin{rem}
The same monotonicity conditions \eqref{H-1} and \eqref{H-2} were also assumed in \cite{BN3,CW1,CW2}.
\end{rem}

By applying the Narrow region principle (Theorem \ref{thm1}) and employing the sliding method, we derive the monotonicity of solutions for fractional equations \eqref{PDE} in bounded domains $\Omega$.
\begin{thm}\label{thm2}
Let $\Omega\subset\mathbb{R}^{N}$ be a bounded domain. Assume that $u \in {\mathcal{L}_{s}}(\mathbb{R}^N)\cap C_{loc}^{1,1}(\Omega)\cap C(\overline{\Omega})$ is a solution of
\begin{eqnarray}\label{331-PDE}
\left\{\begin{array}{ll}
(-\Delta+m^{2})^s u(x)=f(u(x))  &\mbox{ in } \Omega, \\
u(x)=\varphi(x)  & \mbox{ on } \Omega^c,
\end{array}
\right.\end{eqnarray}
and satisfies the hypothesis (H). The function $f(\cdot)$ is supposed to be Lipschitz continuous. Then $u$ is monotone increasing with respect to $x_N$ in $\Omega$, i.e., for any $\tau>0$,
$$
u(x',x_N+\tau)>u(x',x_N), \qquad \forall \,\, (x', x_N), \, (x', x_N+\tau) \in \Omega.
$$
\end{thm}
\begin{proof}
For the sake of simplicity, we may assume that $\Omega$ is an ellipsoid or a rectangle. If $\Omega\subset\mathbb{R}^{N}$ is an arbitrary bounded domain, the proof is entirely similar.

For $\tau\geq 0$, $u^\tau(x)=u(x', x_N+\tau)$ is defined and satisfies the same equation \eqref{331-PDE} (as $u$ does in $\Omega$) on the set $\Omega^\tau=\Omega-\tau e_N$. Set
$$D^\tau:=\Omega^\tau \cap \Omega,$$
$$\tau_{0}:=\sup\{\tau \mid \tau>0, D^\tau\neq \emptyset\},$$
and
$$
w^\tau(x):=u^\tau(x)-u(x), \quad x \in D^\tau.
$$
Then $w^\tau$ satisfies
\begin{eqnarray}\label{331-3}
 (-\Delta+m^{2})^s{w^\tau}(x)=c^\tau(x){w^\tau}(x) \quad \mbox{ in } D^\tau,
\end{eqnarray}
where
\begin{equation}\label{331-5}
	c^{\tau}(x):=\begin{cases} \frac{f(x,u^{\tau}(x))-f\left(x,u(x)\right)}{u^{\tau}(x)-u(x)}, \quad \text{if}\,u^{\tau}(x)\not=u(x),\\ \\ 0,\qquad \text{if}\,u^{\tau}(x)=u(x)\\
	\end{cases}
\end{equation}
is a $L^\infty$ function satisfying
\begin{equation}\label{331-6}
  c^\tau(x)\leq C, \quad \forall \, x \in D^\tau.
\end{equation}

Our goal is to show that
\begin{eqnarray}\label{331-conclusion}
 w^\tau(x)>0 \quad \text{in} \,\, D^\tau, \qquad \forall \,\, 0<\tau<\tau_{0},
\end{eqnarray}
which indicates immediately that $u$ is strictly increasing in the $x_N$ direction.

\emph{Step 1.} We will first show that, for $\tau$ sufficiently close to $\tau_{0}$,
\begin{eqnarray}\label{331-4}
 w^\tau(x) \geq 0 \qquad \text{in} \,\, D^{\tau}.
\end{eqnarray}
Indeed, when $\tau$ sufficiently close to $\tau_{0}$, $D^{\tau}$ is a narrow region. By \eqref{331-3}, \eqref{331-6} and hypothesis (H), \eqref{331-4} follows directly from the Narrow region principle Theorem \ref{thm1}.

\emph{Step 2.} Inequality \eqref{331-4} provides a starting point for us to carry out the sliding procedure. Now we decrease $\tau$ from close to $\tau_{0}$ to $0$ as long as inequality \eqref{331-4} holds until its limiting position. Define
\begin{equation}\label{331-7}
  \tau_{\ast}:=\inf \left\{\tau\in(0,\tau_{0}) \,\mid\, w^{\mu}(x)\geq 0 \,\,\text{in}\, D^{\mu}, \,\forall\,\tau\leq\mu<\tau_{0}\right\}.
\end{equation}
We aim to prove that
$$
 \tau_{\ast}=0.
$$

Otherwise, suppose on the contrary that $\tau_{\ast}>0$, we will show that the domain $\Omega^\tau$ can be slid upward a little bit more, that is, there exists an $0<\varepsilon<\tau_{\ast}$ small enough such that
\begin{eqnarray}\label{331-8}
w^\tau(x)\geq 0 \quad \text{in} \,\, D^\tau, \qquad \forall \,\, \tau_{\ast}-\varepsilon\leq\tau\leq\tau_{\ast},
\end{eqnarray}
which contradicts the definition \eqref{331-7} of $\tau_{\ast}$.

Since by definition of $\tau_{\ast}$,
$$w^{\tau_{\ast}}(x)\geq 0, \;\;\; \forall \,\, x \in D^{\tau_{\ast}},$$
and by the hypothesis (H),
$$w^{\tau_{\ast}}(x)>0, \;\;\; \forall \,\, x \in \Omega \cap \partial D^{\tau_{\ast}},$$
then it follows that
$$w^{\tau_{\ast}}(x) \not\equiv 0 \qquad \text{in} \,\, D^{\tau_{\ast}}.$$

If there exists a point $\hat{x}\in D^{\tau_{\ast}}$ such that $w^{\tau_{\ast}}(\hat{x})=0$, then $\hat{x}$ is the minimum point for $w^{\tau_{\ast}}$ in $\mathbb{R}^{N}$ and
$$
(-\Delta+m^{2})^{s}{w}^{\tau_{\ast}}(\hat{x})=c_{N,s}m^{\frac{N}{2}+s}P.V.\int_{\mathbb{R}^{N}}\frac{-{w}^{\tau_{\ast}}(y)}{|\hat{x}-y|^{\frac{N}{2}+s}}
K_{\frac{N}{2}+s}(m|\hat{x}-y|)dy<0.
$$
This contradicts
$$
(-\Delta+m^{2})^{s}w^{\tau_{\ast}}(\hat{x})=f(u^{\tau_{\ast}}(\hat{x}))-f(u(\hat{x}))=0.
$$
Therefore, we infer that
\begin{eqnarray}\label{331-9}
w^{\tau_{\ast}}(x)>0, \quad \forall \,\, x\in D^{\tau_{\ast}}.
\end{eqnarray}

Now we choose a closed set $K\subset D^{\tau_{\ast}}$ such that $D^{\tau_{\ast}}\backslash K$ is a narrow region. By \eqref{331-9}, there exists a $C_{0}>0$ such that
\begin{equation}\label{331-10}
  w^{\tau_{\ast}}(x)\geq C_0>0 \quad \mbox{ in } K.
\end{equation}
By the continuity of $w^\tau$ with respect to $\tau$, we have, there is an $\varepsilon\in(0,\tau_{\ast})$ sufficiently small, such that, for any $0<\epsilon\leq\varepsilon$,
\begin{equation}\label{331-11}
  w^{{\tau_{\ast}}-\epsilon}(x)\geq 0 \quad \mbox{ in } K.
\end{equation}
Moreover, $D^{\tau_{\ast}-\epsilon}\backslash K$ is still a narrow region. In addition, we obtain from the hypothesis (H) that
\begin{equation}\label{331-12}
  w^{{\tau_{\ast}}-\epsilon}(x) \geq 0 \quad \mbox{ in } (D^{\tau_{\ast}-\epsilon})^c.
\end{equation}
Note that $(D^{{\tau_{\ast}}-\epsilon}\backslash K)^c=K\cup (D^{\tau_{\ast}-\epsilon})^c$, then from \eqref{331-11} and \eqref{331-12} we get
\begin{eqnarray}\label{eqww}
\left\{
\begin{array}{ll}
(-\Delta+m^{2})^s{w^{{\tau_{\ast}}-\epsilon}}(x)-c^{{\tau_{\ast}}-\epsilon}(x){w^{{\tau_{\ast}}-\epsilon}}(x)=0, & x \in D^{{\tau_{\ast}}-\epsilon}\backslash K,\\
{w^{{\tau_{\ast}}-\epsilon}}(x)\geq 0, & x \in  (D^{{\tau_{\ast}}-\epsilon}\backslash K)^c.
\end{array}
\right.
\end{eqnarray}
Then, it follows from the Narrow region principle Theorem \ref{thm1} (more precisely, the proof of Theorem \ref{thm1}) that \eqref{331-8} holds. As a consequence, \eqref{331-8} leads to a contradiction with the definition \eqref{331-7} of $\tau_{\ast}$, thus we must have $\tau_{\ast}=0$ and hence
\begin{eqnarray}\label{331-13}
 w^\tau (x)\geq 0 \quad \text{in} \,\, D^\tau, \quad \forall \; 0<\tau<\tau_{0}.
 \end{eqnarray}

It follows immediately from the hypothesis (H) that
$$w^{\tau}(x) \not\equiv 0 \quad \text{in} \,\, D^{\tau}, \quad \forall \; 0<\tau<\tau_{0}.$$
Thus if there exists a point $x^{\tau}$ such that $w^{\tau}(x^{\tau})=0$, then $x^{\tau}$ is the minimum point and
$$
(-\Delta+m^{2})^{s}w^{\tau}(x^{\tau})=c_{N,s}m^{\frac{N}{2}+s}P.V.\int_{\mathbb{R}^N}\frac{-w^{\tau}(y)}{|x^{\tau}-y|^{\frac{N}{2}+s}}
K_{\frac{N}{2}+s}(m|x^{\tau}-y|)dy<0.
$$
This contradicts $$
(-\Delta+m^{2})^{s}w^{\tau}(x^{\tau})=f(u^{\tau}(x^{\tau}))-f(u(x^{\tau}))=0.
$$
Hence, we arrived at \eqref{331-conclusion}, which implies that $u$ is strictly increasing in the $x_N$ direction. This completes the proof of Theorem \ref{thm2}.
\end{proof}

\subsubsection{Epigraph $D$}
Let the epigraph
$$D:=\left\{x=(x',x_N)\in\mathbb{R}^{N} \mid x_N>\varphi(x')\right\},$$
where $\varphi:\,\mathbb{R}^{N-1}\rightarrow\mathbb{R}$ is a continuous function. A typical example of epigraph $D$ is the upper half-space $\mathbb{R}^{N}_{+}$ ($\varphi\equiv0$).

We have the following monotonicity result on solutions to equation \eqref{PDE} on the epigraph $D$.
\begin{thm}\label{A}
 Let $u \in {\mathcal{L}_{s}}(\mathbb{R}^N)\cap C_{loc}^{1,1}(D)$ be a bounded solution of
 \begin{eqnarray}\label{332-PDE}
 \left\{\begin{array}{ll}
 (-\Delta+m^{2})^s u(x)=f(u(x)), &x \in D;\\
 u(x)=0, & x \in \mathbb{R}^N\setminus D,
 \end{array}
 \right.
 \end{eqnarray}
 where $f(\cdot)$ satisfies
 \begin{equation}\label{332-0}
 \sup_{\substack{t_{1},\,t_{2}\in[\inf u,\, \sup u] \\ t_1>t_2}}\frac{f(t_1)-f(t_2)}{t_1-t_2}<m^{2s}.
 \end{equation}
Assume that there exists $l>0$ such that
\begin{equation}\label{332-con}
  u\geq0 \,\,\,\, \text{in} \,\, \left\{x=(x',x_N)\in D\,|\,\varphi(x')<x_{N}<\varphi(x')+l\right\}.
\end{equation}
Then, either $u\equiv0$ in $\mathbb{R}^{N}$ and $f(0)=0$, or $u$ is strictly monotone increasing in the $x_N$ direction and hence $u>0$ in $D$.

If, in addition, $D$ is contained in a half-space, the same conclusion can be reached without the assumption \eqref{332-con}. Furthermore, if $D$ itself is exactly a half-space, then
$$u(x)=u\left(\langle\left(x',x_{N}-\varphi(0')\right),{\bf \nu}\rangle\right),$$
where ${\bf \nu}$ is the unit inner normal vector to the hyper-plane $\partial D$ and $\langle\cdot,\cdot\rangle$ denotes the inner product in Euclidean space. In particular, if $D=\mathbb{R}^{N}_{+}$, then $u(x)=u(x_{N})$.
\end{thm}
\begin{proof}
 For any $0<\tau<l$, let
 $$u^\tau(x):=u(x', x_N+\tau)$$
 and
 $$w^\tau(x):=u(x)-u^\tau(x).$$

 Since $f(u)$ satisfies \eqref{332-0}, we have
 $$f(u(x))-f(u^\tau(x))= c^\tau(x)w^\tau(x) \quad \mbox{ at points } x\in D \mbox{ where } w^{\tau}(x)>0,$$
 with $c^\tau(x)=\frac{f(u(x))-f(u^\tau(x))}{u(x)-u^\tau(x)}$ satisfying
 $$\sup_{\{x\in D \,|\, w^{\tau}(x)>0\}}c^{\tau}(x)<m^{2s}.$$
 Consequently, \eqref{332-PDE} implies that
 $$
 (-\Delta+m^{2})^{s}w^\tau(x)-c^\tau(x)w^\tau(x)=0 \quad \mbox{ at points } x\in D \mbox{ where } w^\tau(x)>0.
 $$

 In addition, for any $0<\tau<l$, we have
 $$w^\tau(x)\leq 0, \quad \forall \,\, x \in \mathbb{R}^N\setminus D.$$
 Thus it follows immediately from Theorem \ref{MP-Ubdd} that, for any $0<\tau<l$,
 $$w^\tau(x)\leq 0, \quad \forall \,\, x\in D.$$

 Now, suppose that $u\not\equiv0$ in $D$, then there exists a $\hat{x}\in D$ such that $u(\hat{x})>0$. We are to show that, for any $0<\tau<l$,
 \begin{equation}\label{332-1}
 w^\tau(x)<0, \quad \forall \,\, x \in D.
 \end{equation}
 If not, there exists a point $x^{\tau}\in D$ such that
$$
 w^\tau(x^{\tau})=0=\max_{\mathbb{R}^N}w^\tau(x).
 $$
 Then we have
 $$
 (-\Delta+m^{2})^s{w}^\tau(x^{\tau})=f(u(x^{\tau}))-f(u^\tau(x^{\tau}))=0,
 $$
 it follows immediately from Lemma \ref{SMP-Ubdd} that $w^\tau=0$ a.e. in $\mathbb{R}^{N}$. This contradicts $u(\hat{x})>0$ and $u=0$ in $\mathbb{R}^{N}\setminus D$. Therefore, \eqref{332-1} holds and hence $u$ is strictly monotone increasing in the $x_N$ direction. In particular, $u>0$ in $D$.

If, in addition, $D$ is contained in a half space, we will prove that $$u\geq 0\qquad \text{in}\,\,D$$ and hence the assumption \eqref{332-con} is redundant.

Without loss of generalities, we may assume that $D\subseteq \mathbb{R}^N_+$, let
\begin{equation}\label{332-2}
  T_{0}:=\left\{x \in \mathbb{R}^{N} | x_{N}=0\right\},
\end{equation}
\begin{equation}\label{332-3}
 \Sigma_{0}:=\left\{x \in \mathbb{R}^{N} | x_{N}>0\right\}
 \end{equation}
 be the region above the plane $T_0$, and
 $$
 x^{0}:=\left(x_{1}, x_{2}, \ldots, -x_{N}\right)
 $$
 be the reflection of $x$ about the plane $T_{0}$. We denote $u_{0}(x):=u\left(x^{0}\right)$ and $w_{0}(x)=u_{0}(x)-u(x)$. For $x\in \Sigma_{0}$ where $w_{0}(x)>0$, we derive by \eqref{332-PDE} and \eqref{332-0} that, $x\in D$ and
 \begin{equation*}
 (-\Delta+m^{2})^s{w}_0(x)=f(u_0(x))-f(u(x))=c_0(x)w_0(x),
 \end{equation*}
 where $c_0(x)=\frac{f(u_0(x))-f(u(x))}{u_0(x)-u(x)}$ satisfying
 $$\sup_{\left\{x\in\Sigma_{0} \mid  w_{0}(x)>0\right\}} c_0(x)<m^{2s}.$$
 Hence, we obtain from Theorem \ref{MP_anti_ubdd} that $w_0\leq 0$ in $\Sigma_{0}$, which implies immediately $u\geq 0$ in $D$.

Furthermore, suppose $D$ itself is exactly a half space. Without loss of generalities, we may assume that $D=\mathbb{R}^N_+$. We will show that $u(x)$ depends on $x_N$ only.

In fact, when $D=\mathbb{R}^N_+$, it can be seen from the above sliding procedure that the methods should still be valid if we replace $u^\tau(x):=u(x+\tau e_{N})$ by $u(x+\tau\nu)$, where $\nu=(\nu_1,\cdots,\nu_N)$ is an arbitrary vector such that $\langle\nu,e_{N}\rangle=\nu_{N}>0$. Applying similar sliding methods as above, we can derive that, for arbitrary such vector $\nu$,
 $$u(x+\tau\nu)>u(x) \quad \text{in} \,\, \mathbb{R}^N_+, \quad \forall \,\, \tau>0.$$
 Let $\nu_N\rightarrow0+$, from the continuity of $u$, we deduce that
 $$u(x+\tau\nu)\geq u(x)$$
 for arbitrary vector $\nu$ with $\nu_N=0$. By replacing $\nu$ by $-\nu$, we arrive at
 $$u(x+\tau\nu)=u(x)$$
 for arbitrary vector $\nu$ with $\nu_N=0$, this means that $u(x)$ is independent of $x'$, hence $u(x)=u(x_N)$. This finishes the proof of Theorem \ref{A}.
\end{proof}
\begin{rem}\label{A1}
If $\varphi\equiv0$, then $D$ is the half space $\mathbb{R}^N_+$, which is an important special case of Theorem \ref{A}. Our assumptions in Theorem \ref{A} are much weaker than that in Lemma 3.1 in \cite{CW2} for fractional Laplacians $(-\Delta)^{s}$ and $D=\mathbb{R}^{N}_{+}$, which requires $f(\cdot)$ to be non-increasing and $u>0$ in $\mathbb{R}^{N}_{+}$.
\end{rem}

\subsubsection{Whole-space $\mathbb{R}^{N}$}

Next, as an application of the sliding method, we derive the monotonicity of solutions for fractional semi-linear equations \eqref{PDE} in the whole space $\mathbb{R}^{N}$.

For arbitrary $H>0$, let $D_{H}:=\mathbb{R}^{N-1}\times[-H,H]\subset\mathbb{R}^{N}$. For $0<\alpha\leq1$, we say that $u\in C^{0,\alpha}(D_{H})$, if there exists a constant $C_{H}$ such that
\begin{equation}\label{333-0}
  \|u\|_{C^{0,\alpha}(D_{H})}\leq C_{H}.
\end{equation}

We have the following monotonicity results for solutions to \eqref{PDE} in $\mathbb{R}^{N}$.
\begin{thm}\label{thm333-1}
Let $u \in {\mathcal L}_{s}(\mathbb{R}^N)\cap C^{1,1}_{loc}\cap\left(\bigcap_{H>0}C^{0,\alpha}(D_{H})\right)$ (for some $\alpha\in(0,1]$) be a solution of
\begin{eqnarray}\label{333.1}
(-\Delta+m^{2})^s u(x)=f(u(x)), \quad \forall \,\, x \in  \mathbb{R}^N, \\
|u(x)|\leq 1, \quad \forall \,\, x \in  \mathbb{R}^N,
\end{eqnarray}
and
\begin{eqnarray}\label{333.2}
u(x',x_N) \mathop {\longrightarrow }\limits_{{x_N\rightarrow \pm \infty}}\pm 1 \; \mbox{ uniformly w.r.t. }\; x'=(x_1, \cdots, x_{N-1}).
\end{eqnarray}
Assume that $f(\cdot)$ is continuous in $[-1,1]$ and there exists $\delta>0$ such that
\begin{equation}\label{333.3}
	\sup_{\substack{-1\leq t_{1}<t_{2}\leq-1+\delta \\ 1-\delta\leq t_{1}<t_{2}\leq1}}\frac{f(t_{2})-f(t_{1})}{t_{2}-t_{1}}<m^{2s}.
	\end{equation}
Then $u(x)$ is strictly increasing w.r.t. $x_N$. Furthermore, $u(x)$ depends on $x_N$ only.
\end{thm}
\begin{rem}\label{rem3}
One should note that the De Giorgi type nonlinearity $f(u)=u-u^3$ satisfies condition \eqref{333.3}. Theorem \ref{thm333-1} is closely related to the De Giorgi's conjecture (refer to \cite{DG}, see also \cite{AC,DKW,GG1,GG2,Savin}): If $u$ is a solution of
\be\label{DG1}
-\Delta u(x)=u(x)-u^3(x), \quad \forall \,\, x \in \mathbb{R}^N
\ee
such that
$$|u(x)| \leq 1, \,\, \lim_{x_N \ra \pm \infty} u(x',x_N) = \pm 1 \; \mbox{ uniformly w.r.t. }\; x' \in \mathbb{R}^{N-1} \,\, \mbox{ and }\,\, \frac{\partial u}{\partial x_N} >0.$$
Then there exists a vector ${\bf a} \in \mathbb{R}^{N-1}$ and a function $u_1: \mathbb{R} \ra \mathbb{R}$ such that
$$ u(x', x_N)=u_1(\langle{\bf a},x'\rangle+x_N), \quad \forall \,\, x \in \mathbb{R}^N.$$
\end{rem}
\begin{rem}\label{rem4}
Berestycki, Hamel, and Monneau \cite{BHM} proved the same monotonicity result for the equation $-\Delta u=f(u)$ in $\mathbb{R}^N$ under the same conditions as in Theorem \ref{thm333-1} and additional conditions that $f(\cdot)$ is Lipschitz continuous in $[-1,1]$ and $f(\pm 1)=0$. These conditions on $f$ were also assumed by Dipierro, Soave and Valdinoci in \cite{DSV} when they considered fractional equations involving $(-\Delta)^{s}$. In Chen and Liu \cite{CLiu}, Chen and Wu \cite{CW1,CW2}, by applying a new approach, the authors weakened the additional assumptions on $f$ (only assumed $f$ is continuous in $[-1,1]$) and derived the same monotonicity result for the equation $(-\Delta)^{s}u=f(u)$ in $\mathbb{R}^N$.
\end{rem}

\begin{rem}\label{rem10}
Our assumptions \eqref{333.3} in Theorem \ref{thm333-1} and \eqref{333.7} in Lemma \ref{A2} on $f$ are much weaker than that in Theorem 3 in \cite{CW1}, Theorem 4 in \cite{CW2} and Theorem 3 in \cite{CW3} for fractional Laplacians $(-\Delta)^{s}$, which requires $f(t)$ to be non-increasing for $t$ close to $\pm 1$.
\end{rem}

\begin{proof}[Proof of Theorem \ref{thm333-1}]
We start the proof of Theorem \ref{thm333-1} by proving the following Lemma via the maximum principle in unbounded open sets, i.e., Theorem \ref{MP-Ubdd}.
\begin{lem}\label{A2}
Suppose that $u \in {\mathcal{L}_{s}}(\mathbb{R}^N)\cap C_{loc}^{1,1}(\mathbb{R}^N)$ is a solution of
\begin{equation}\label{333.13}
  (-\Delta+m^{2})^{s} u(x)=f(u(x)), \quad \forall \,\, x \in \mathbb{R}^N,
\end{equation}
and
$$|u(x)| \leq 1, \quad \forall \,\, x \in \mathbb{R}^N, $$
\be\label{333.6}
\lim_{x_N \ra \pm \infty} u(x',x_N) =  \pm 1 \mbox{ uniformly w.r.t. } x' \in \mathbb{R}^{N-1}.
\ee
Assume that there exists $\delta>0$ such that
\begin{equation}\label{333.7}
	\sup_{\substack{-1\leq t_{1}<t_{2}\leq-1+\delta \\ 1-\delta\leq t_{1}<t_{2}\leq1}}\frac{f(t_{2})-f(t_{1})}{t_{2}-t_{1}}<m^{2s}.
	\end{equation}
Then $u^\tau(x)\geq u(x)$ in $\mathbb{R}^N$ for sufficiently large $\tau$.
\end{lem}
\begin{proof}
Let
$$
w^\tau(x):=u(x)-u^\tau(x), \quad x\in \mathbb{R}^N.
$$
Our aim is to show that, for sufficiently large $\tau$,
\begin{eqnarray}\label{333.4}
w^\tau(x)\leq 0 \quad \mbox{ in } \mathbb{R}^N.
\end{eqnarray}
Otherwise, suppose \eqref{333.4} does not hold, then
\begin{eqnarray}\label{333.5}
\sup_{\mathbb{R}^N}w^\tau(x)=:M>0,
\end{eqnarray}
we will try to derive a contradiction provided that $\tau$ is large enough.

For arbitrarily chosen $\theta\in(0,1)$, consider the function $w^\tau(x)-\theta M$. We obtain from \eqref{333.6} that there exists a constant $\rho>0$ such that
\begin{equation}\label{333.11}
  u(x',x_N) \geq 1-\delta,  \quad \mbox{ if } x_N \geq\rho
\end{equation}
and
\begin{equation}\label{333.12}
  u(x', x_N) \leq -1+\delta, \quad \mbox{ if } x_N \leq -\rho,
\end{equation}
and there exists a positive constant $L>\rho$ such that
\begin{eqnarray}\label{333.8}
w^\tau(x)-\theta M\leq 0, \quad \forall \,\, x \in \mathbb{R}^{N-1}\times [L,+\infty).
\end{eqnarray}
One can observe that, if $\tau\geq 2\rho$, then for any $x\in\mathbb{R}^{N}$, at least one of the points $x$ and $x+\tau e_{N}$ lie in the domain $\{x\in\mathbb{R}^{N}: |x_N|\geq\rho\}$. Thus we infer from \eqref{333.11} and \eqref{333.12}, either
\begin{eqnarray}\label{333.9}
u^\tau(x',x_N) \geq 1-\delta \quad  (\mbox{ if } x_N \geq -\rho),
\end{eqnarray}
or
\begin{eqnarray}\label{333.10}
u(x',x_N)\leq -1+\delta \quad (\mbox{ if } x_N \leq -\rho).
\end{eqnarray}
Therefore, from the assumption \eqref{333.7} on $f$, we get, at points $x\in D:=\mathbb{R}^{N-1}\times (-\infty, L)$ where $u(x)>u^\tau(x)$,
$$
f(u(x))-f(u^\tau(x))=:c^{\tau}(x)w^{\tau}(x),
$$
where $c^{\tau}(x):=\frac{f(u(x))-f(u^\tau(x))}{u(x)-u^\tau(x)}$ satisfies
\[\sup_{\{x\in D \mid w^{\tau}(x)>0\}}c^{\tau}(x)<m^{2s}.\]
Consequently, \eqref{333.13} implies that
\begin{equation}\label{333.14}
  (-\Delta+m^{2})^s\left(w^\tau-\theta M\right)=f(u(x))-f(u^\tau(x))-m^{2s}\theta M<c^{\tau}(x)(w^{\tau}(x)-\theta M)
\end{equation}
at points $x\in D$ where $w^\tau(x)-\theta M>0$. It follows from \eqref{333.14}, \eqref{333.8} and Theorem \ref{MP-Ubdd} that
$$w^\tau(x)-\theta M\leq 0, \quad \forall \,\, x\in \mathbb{R}^N.$$
This contradicts \eqref{333.5} since $\theta\in(0,1)$, thus we derive \eqref{333.4} for any $\tau\geq2\rho$. This finishes the proof of Lemma \ref{A2}.
\end{proof}

Next, we will continue to carry out the proof of Theorem \ref{thm333-1}. Our aim is to prove that for any $\tau>0$, we have
\begin{equation}\label{AIM}
w^\tau (x)\leq 0 \quad \mbox{ in } \mathbb{R}^N.
\end{equation}

From Lemma \ref{A2}, . This provides a starting point for the sliding procedure. Then we decrease $\tau$ from $2\rho$ to its limit as long as inequality \eqref{AIM} holds.

We divide the sliding procedure into three steps.

\medskip

\emph{Step 1.} As an immediate consequence of Lemma \ref{A2}, we have
\begin{eqnarray}\label{333-start}
w^\tau (x) \leq 0  \quad \mbox{ in } \, \mathbb{R}^N, \quad \forall \,\, \tau\geq 2\rho.
\end{eqnarray}
Thus we have derived \eqref{AIM} for any $\tau\geq2\rho$.

\medskip

\emph{Step 2.} Inequality \eqref{333-start} provides a starting point for the sliding procedure. Next, we will decrease $\tau$ from $2\rho$ to its limit as long as inequality \eqref{AIM} holds. To this end, define
\begin{equation}\label{Def}
  \tau_0:=\inf\,\{\tau\in(0,2\rho] \,\mid\, w^\tau(x)\leq 0, \,\, \forall \,\, x \in \mathbb{R}^N\}.
\end{equation}
It is obvious that
\begin{equation}\label{333.16}
  w^{\tau_0}(x)\leq0, \quad \forall \,\, x\in\mathbb{R}^{N}.
\end{equation}
We will show by contradiction arguments that
\begin{equation}\label{333.15}
  \tau_0=0,
\end{equation}
which indicates immediately that \eqref{AIM} holds for any $\tau>0$. Otherwise, suppose on the contrary that $0<\tau_{0}\leq2\rho$, we will prove that $\tau_0$ can be decreased a little bit while inequality \eqref{AIM} still holds, which contradicts the definition \eqref{Def} of $\tau_0$.

The proof of \eqref{333.15} can be divided into two parts.

\medskip

\emph{Part (i).} We first show that
\begin{eqnarray}\label{333.17}
\sup_{\mathbb{R}^{N-1}\times[-\rho,\rho]}w^{\tau_0}(x)<0.
\end{eqnarray}

If not, then
\begin{equation}\label{333.18}
  \mathop{\sup}\limits_{\mathbb{R}^{N-1}\times[-\rho,\rho]}w^{\tau_0}(x)=0,
\end{equation}
and hence there exists a sequence
$$\{x^k\} \subset \mathbb{R}^{N-1}\times[-\rho,\rho],\quad k=1, 2, \cdots,$$
such that
\begin{equation}\label{333.19}
  w^{\tau_0}(x^k) \rightarrow 0, \quad \mbox{ as } k\rightarrow+\infty.
\end{equation}

Now, we define a function $\psi\in C^{\infty}_{0}(\mathbb{R}^{N})$ by
\begin{eqnarray}\label{psi}
\psi(x)=\left\{\begin{array}{ll}
e^{\frac{|x|^{2}}{|x|^2-1}}, &|x|<1, \\
  0,& |x|\geq 1.
\end{array} \right.
\end{eqnarray}
Let
$$
\psi_k(x)=\psi(x-x^k).
$$
Then, by \eqref{333.19} and \eqref{psi}, there exists a sequence $\{\epsilon_k\}\rightarrow 0$ such that
$$
 w^{\tau_0}(x^k)+\epsilon_k\psi_k(x^k)>0.
$$
Note that $w^{\tau_0}(x)\leq 0$ in $\mathbb{R}^{N}$, and $\psi_k(x)=0$ in $\mathbb{R}^N \backslash B_{1}(x^k)$, one has
$$
w^{\tau_0}(x^k)+\epsilon_k\psi_k(x^k)>w^{\tau_0}(x)+\epsilon_k\psi_k(x),\quad \forall \,\, x \in \mathbb{R}^N \backslash B_1(x^k).
$$
It follows that there exists a point $\bar{x}^k \in B_1(x^k)$ such that
\begin{eqnarray}\label{333.20}
 w^{\tau_0}(\bar{x}^k)+\epsilon_k\psi_k(\bar{x}^k)=\max_{\mathbb{R}^N}\left(w^{\tau_0}(x)+\epsilon_k\psi_k(x)\right)>0.
\end{eqnarray}
In particular, we have
$$
w^{\tau_0}(\bar{x}^k)+\epsilon_k\psi_k(\bar{x}^k)\geq w^{\tau_0}(x^k)+\epsilon_k \psi_k(x^k),
$$
combining this with the fact $\psi_k(\bar{x}^k)\leq\psi_k(x^k)$ yield that
$$
0\geq w^{\tau_0}(\bar{x}^k)\geq w^{\tau_0}(x^k).
$$
As a consequence, \eqref{333.19} implies that
\begin{eqnarray}\label{333.21}
w^{\tau_0}(\bar{x}^k)\rightarrow 0, \quad \text{as} \,\,k\rightarrow+\infty.
\end{eqnarray}

Since $\psi\in C^{\infty}_{0}(\mathbb{R}^{N})$, one has $|(-\Delta+m^{2})^{s}\psi_k(x)|\leq C$ for any $x\in\mathbb{R}^{N}$. It follows from \eqref{333.1}, \eqref{333.21} and the continuity of $f$ that
 \begin{eqnarray}\label{333.22}
(-\Delta+m^{2})^s(w^{\tau_0}+\epsilon_k\psi_k)(\bar{x}^k)=f(u(\bar{x}^k))-f(u^{\tau_0}(\bar{x}^k))+O(1)\varepsilon_k \rightarrow 0,
\end{eqnarray}
as $k\rightarrow+\infty$.

Now we need the following generalized average property for pseudo-relativistic Schr\"{o}dinger operators $(-\Delta+m^{2})^{s}$.
\begin{lem}[A generalized average inequality]\label{GAI}
Assume $u \in {{\mathcal{L}_{s}}(\mathbb{R}^N)} \cap C_{loc}^{1,1}(\mathbb{R}^N)$. For any $r\geq \frac{R_{\infty}}{m}$, if $\bar{x}$ is a maximum point of $u$ in $B_{r}(\bar{x})$. Then, we have
\begin{equation}\label{AI}
\frac{\left[(-\Delta+m^{2})^{s}-m^{2s}\right]u(\bar{x})}{c_{N,s}c_{\infty}m^{\frac{N-1}{2}+s}I(r)}
+\frac{1}{I(r)}\int_{B_r^c(\bar{x})}\frac{u(y)}{|\bar{x}-y|^{\frac{N+1}{2}+s}e^{m|\bar{x}-y|}}dy
\geq u(\bar{x}),
\end{equation}
where the function $I(r)$ for $\frac{R_{\infty}}{m}\leq r<+\infty$ is defined by
$$
I(r):=\int_{B_1^c(0)}\frac{r^{\frac{N-1}{2}-s}}{|y|^{\frac{N+1}{2}+s}e^{mr|y|}}dy.
$$
\end{lem}
\begin{proof}
By definition of $(-\Delta+m^{2})^{s}$, we have, for any $r\geq\frac{R_{\infty}}{m}$,
\begin{eqnarray*}
&&(-\Delta+m^{2})^{s}u(\bar{x})\\
&=&c_{N,s}m^{\frac{N}{2}+s}P.V.\int_{\mathbb{R}^N}\frac{u(\bar{x})-u(y)}{|\bar{x}-y|^{\frac{N}{2}+s}}K_{\frac{N}{2}+s}(m|\bar{x}-y|)dy+m^{2s}u(\bar{x})\\
&\geq& c_{N,s}c_{\infty}m^{\frac{N-1}{2}+s}\int_{B_r^c(\bar{x})}\frac{u(\bar{x})-u(y)}{|\bar{x}-y|^{\frac{N+1}{2}+s}e^{m|\bar{x}-y|}}dy+m^{2s}u(\bar{x})\\
&=& -c_{N,s}c_{\infty}m^{\frac{N-1}{2}+s}\int_{B_r^c(\bar{x})}\frac{u(y)}{|\bar{x}-y|^{\frac{N+1}{2}+s}e^{m|\bar{x}-y|}}dy \\
&& +c_{N,s}c_{\infty}m^{\frac{N-1}{2}+s}\int_{B_r^c(\bar{x})}\frac{u(\bar{x})}{|\bar{x}-y|^{\frac{N+1}{2}+s}e^{m|\bar{x}-y|}}dy+m^{2s}u(\bar{x}) \\
&=& -c_{N,s}c_{\infty}m^{\frac{N-1}{2}+s}\int_{B_r^c(\bar{x})}\frac{u(y)}{|\bar{x}-y|^{\frac{N+1}{2}+s}e^{m|\bar{x}-y|}}dy \\
&& +\left(c_{N,s}c_{\infty}m^{\frac{N-1}{2}+s}I(r)+m^{2s}\right)u(\bar{x}),
\end{eqnarray*}
where the function $I(r)$ for $\frac{R_{\infty}}{m}\leq r<+\infty$ is defined by
$$
I(r):=\int_{B_1^c(0)}\frac{r^{\frac{N-1}{2}-s}}{|y|^{\frac{N+1}{2}+s}e^{mr|y|}}dy.
$$
As a consequence, we have reached \eqref{AI}. This completes the proof of Lemma \ref{GAI}.
\end{proof}

\begin{rem}\label{rem5}
Since there is a modified Bessel function of the second kind $K_{\frac{N}{2}+s}$ in definition of the pseudo-relativistic Schr\"{o}dinger operators $(-\Delta+m^{2})^{s}$, the average inequality \eqref{AI} of $(-\Delta+m^{2})^{s}$ is quite different from those of fraction Laplacians $(-\Delta)^{s}$, see e.g. Lemma 1 in \cite{CW2} and Theorem 6.1 in \cite{CDQ}.
\end{rem}

\begin{rem}\label{rem6}
In the special case when $u$ satisfies the following $s$-subharmonic property at the local maximum point $\bar{x}$:
\begin{equation}\label{subh}
  \left[(-\Delta+m^{2})^{s}-m^{2s}\right]u(\bar{x})\leq0,
\end{equation}
the average inequality \eqref{AI} becomes: for any $r\geq\frac{R_{\infty}}{m}$, if $\bar{x}$ is a maximum point of $u$ in $B_{r}(\bar{x})$, then
\begin{equation}\label{333.37}
u(\bar{x})\leq\int_{B_r^c(\bar{x})}u(y)\,d\mu(y)
\end{equation}
with
$$\int_{B_r^c(\bar{x})}d\mu(y)=1.$$
Here the integral on the right hand side of \eqref{333.37} is actually a weighted average value of $u$ outside the ball $B_r(\bar{x})$. The average inequality \eqref{333.37} also implies that a local maximum point $\bar{x}$ satisfying \eqref{subh} can not be a global maximum point unless $u$ is identically a constant. We believe that the inequalities \eqref{AI} or \eqref{333.37} will become very useful and important tools in analyzing equations involving pseudo-relativistic Schr\"{o}dinger operators or various inhomogeneous fractional order operators.
\end{rem}

With the help of Lemma \ref{GAI}, we can continue the proof of Theorem \ref{thm333-1}.

Since $\bar{x}^{k}$ is a global maximum point of $w^{\tau_0}+\epsilon_k \psi_k$, by applying Lemma \ref{GAI} to $w^{\tau_0}+\epsilon_k \psi_k$ (with the choice $r=\frac{R_{\infty}}{m}$), we derive
\begin{eqnarray*}
&&\quad \frac{\left[(-\Delta+m^{2})^{s}-m^{2s}\right](w^{\tau_0}+\epsilon_k \psi_k)(\bar{x}^{k})}{c_{N,s}c_{\infty}m^{\frac{N-1}{2}+s}}+\int_{B_{\frac{R_{\infty}}{m}}^c(\bar{x}^{k})}
\frac{\left(w^{\tau_0}+\epsilon_k \psi_k\right)(y)}{|\bar{x}^{k}-y|^{\frac{N+1}{2}+s}e^{m|\bar{x}^{k}-y|}}dy \\
&&\geq I\left(\frac{R_{\infty}}{m}\right)\left(w^{\tau_0}+\epsilon_k \psi_k\right)(\bar{x}^{k}),
\end{eqnarray*}
Combining this with \eqref{333.21} and \eqref{333.22}, and the fact that $\epsilon_k\rightarrow 0$, we arrive at
\begin{eqnarray}\label{333.23}
\lim_{k\rightarrow+\infty}\int_{B_{\frac{R_{\infty}}{m}}^c(\bar{x}^k)}\frac{w^{\tau_0}(y)}{|\bar{x}^{k}-y|^{\frac{N+1}{2}+s}e^{m|\bar{x}^{k}-y|}}dy
=\lim_{k\rightarrow+\infty}\int_{B_{\frac{R_{\infty}}{m}}^c(0)}\frac{w^{\tau_0}(x+\bar{x}^k)}{|x|^{\frac{N+1}{2}+s}e^{m|x|}}dx=0.
\end{eqnarray}

Now, let
$$u_k(x):=u(x+\bar{x}^k) \quad \mbox{ and } \quad w^{\tau_0}_k(x):=w^{\tau_0}(x+\bar{x}^k).$$
Since $\bar{x}^{k}\in\mathbb{R}^{N-1}\times[-\rho-1,\rho+1]$ and $u\in\bigcap_{H>0}C^{0,\alpha}(D_{H})$ for some $\alpha\in(0,1]$, by Arzel\`{a}-Ascoli theorem, up to extraction of a subsequence (still denoted by $u_{k}$ and $w^{\tau_0}_k$ respectively), we have
\begin{equation}\label{333.24}
  u_k(x)\rightarrow u_\infty (x) \quad \mbox { uniformly in } \mathbb{R}^N, \quad \mbox{ as } k\rightarrow+\infty,
\end{equation}
and hence by \eqref{333.23},
\begin{equation}\label{333.25}
  w^{\tau_0}_k(x):=u_k(x)-(u_k)^{\tau_{0}}\rightarrow  0 \quad \mbox { uniformly in } B_{\frac{R_{\infty}}{m}}^c(0), \quad \mbox{ as } k\rightarrow+\infty.
\end{equation}
Therefore, we have
\begin{equation}\label{333.26}
  u_\infty(x)-(u_\infty)^{\tau_0}(x)\equiv 0, \quad \forall \,\, x \in B_{\frac{R_{\infty}}{m}}^c(0).
\end{equation}

As a consequence of \eqref{333.26}, we deduce that, for any given $x\in B_{\frac{R_{\infty}}{m}}^c(0)$ and any $k\geq1$,
\begin{eqnarray}\label{333.28}
&&u_{\infty}(x)=u_\infty(x',x_N)=u_\infty(x', x_N+\tau_0)=u_\infty(x', x_N+2\tau_0)\\
&&\qquad\quad =\cdots=u_\infty(x', x_N+k\tau_0)=\cdots.\nonumber
\end{eqnarray}
Since $\bar{x}^{k}\in\mathbb{R}^{N-1}\times[-\rho-1,\rho+1]$, we obtain from \eqref{333.2} and \eqref{333.24} that
\begin{eqnarray}\label{333.27}
u_\infty(x', x_N)  \mathop {\longrightarrow }\limits_{{x_N\rightarrow \pm \infty}}\pm 1 \;\mbox{ uniformly w.r.t. }\; x'=(x_1, \cdots, x_{N-1}).
\end{eqnarray}
Combining \eqref{333.28} with \eqref{333.27}, we get, for any given $x\in B_{\frac{R_{\infty}}{m}}^c(0)$ with $x_{N}$ sufficiently negative,
\begin{equation}\label{333.29}
  -\frac{1}{2}>u_{\infty}(x)=u(x',x_{N})=\lim_{k\rightarrow+\infty}u_\infty(x', x_N+k\tau_0)=1,
\end{equation}
which is absurd. Hence \eqref{333.17} must hold.

\medskip

\emph{Part (ii).} Now, we are to show that, there exists an $\varepsilon\in(0,\tau_{0})$, such that
\begin{eqnarray}\label{333.30}
w^{\tau}(x)\leq0 \quad \text{in} \,\, \mathbb{R}^N, \quad \forall \,\, \tau\in[\tau_0-\varepsilon,\tau_0].
\end{eqnarray}

First, due to $u\in\bigcap_{H>0}C^{0,\alpha}(D_{H})$ for some $\alpha\in(0,1]$, \eqref{333.17} implies immediately that, there exists an $\varepsilon\in(0,\tau_{0})$ small enough such that
\begin{eqnarray}\label{333.31}
\sup_{\mathbb{R}^{N-1}\times[-\rho,\rho]}w^{\tau}(x)<0, \quad \forall \,\, \tau \in [\tau_0-\varepsilon,\tau_0].
\end{eqnarray}
Consequently, we only need to show that
\begin{eqnarray}\label{333.32}
\sup_{\mathbb{R}^{N}\backslash\left(\mathbb{R}^{N-1}\times[-\rho,\rho]\right)}w^{\tau}(x)\leq 0, \quad \forall \,\, \tau \in[\tau_0-\varepsilon,\tau_0].
\end{eqnarray}

We will prove \eqref{333.32} via the maximum principle in unbounded open sets (Theorem \ref{MP-Ubdd}). To this end, define
$$D:=\mathbb{R}^{N}\backslash\left(\mathbb{R}^{N-1}\times[-\rho,\rho]\right).$$
For any $\tau\in[\tau_0-\varepsilon, \tau_0]$, at the points $x\in D$ where $w^{\tau}(x)>0$, we have $u(x)>u^{\tau}(x)$. If $x \in \mathbb{R}^{N-1}\times(\rho,+\infty)$, we have $u(x)>u^{\tau}(x)\geq 1-\delta$. If $x \in \mathbb{R}^{N-1}\times (-\infty, -\rho)$, we have $-1+\delta\geq u(x)>u^{\tau}(x)$. Therefore, we can deduce from equation \eqref{333.1} and assumption \eqref{333.3} that, at points $x\in D$ where $w^{\tau}(x)>0$,
\begin{equation}\label{333.33}
(-\Delta+m^{2})^s w^{\tau}(x)=f(u(x))-f(u^{\tau}(x))=:c^{\tau}(x)w^{\tau}(x),
\end{equation}
where $c^{\tau}(x):=\frac{f(u(x))-f(u^{\tau}(x))}{u(x)-u^{\tau}(x)}$ satisfies
\[\sup_{\{x\in D \mid w^{\tau}(x)>0\}}c^{\tau}(x)<m^{2s}.\]
Combining this with \eqref{333.31}, and applying Theorem \ref{MP-Ubdd}, we arrive at \eqref{333.32} and hence \eqref{333.30} immediately. It is obvious that \eqref{333.30} leads to a contradiction with the definition of $\tau_0$. Thus we obtain \eqref{333.15}, that is, $\tau_{0}=0$.

\medskip

\emph{Step 3.} In this step, we will show that $u$ is strictly increasing with respect to $x_N$ and $u(x)$ depends on $x_N$ only.

From $\tau_{0}=0$ derived in Step 2, we deduce that
\begin{equation}\label{333.34}
w^\tau(x)\leq 0 \quad \mbox{ in } \,\, \mathbb{R}^N, \quad \forall \,\, \tau>0.
\end{equation}
For any $\tau>0$, if there exists some point $x^{\tau}\in\mathbb{R}^N$ such that ${w^\tau}(x^{\tau})=0$, then $x^{\tau}$ is a maximum point of $w^\tau$ in $\mathbb{R}^N$ and
\begin{equation}\label{333.35}
  (-\Delta+m^{2})^s{w^\tau}(x^{\tau})=f(u(x^{\tau}))-f({u^\tau}(x^{\tau}))=0.
\end{equation}
Lemma \ref{SMP-Ubdd} implies immediately that $w^{\tau}=0$ a.e. in $\mathbb{R}^{N}$, which contradicts the assumption \eqref{333.2}. Therefore, we must have
\begin{eqnarray*}
w^\tau(x)< 0 \quad \mbox{ in } \,\, \mathbb{R}^N, \quad \forall \,\, \tau >0.
\end{eqnarray*}
This indicates that $u$ is strictly increasing with respect to $x_N$.

\smallskip

Now we will show that $u(x)$ depends on $x_N$ only.

In fact, it can be seen from the above sliding procedure that the methods should still be valid if we replace $u^\tau(x):=u(x+\tau e_{N})$ by $u(x+\tau\nu)$, where $\nu=(\nu_1,\cdots,\nu_N)$ is an arbitrary vector such that $\langle\nu,e_{N}\rangle=\nu_{N}>0$. Applying similar sliding methods as in Step 1 and 2, we can derive that, for arbitrary such vector $\nu$,
$$u(x+\tau\nu)>u(x) \quad \text{in} \,\, \mathbb{R}^N, \quad \forall \,\, \tau>0.$$
Let $\nu_N\rightarrow0+$, from the continuity of $u$, we deduce that
$$u(x+\tau\nu)\geq u(x)$$
for arbitrary vector $\nu$ with $\nu_N=0$. By replacing $\nu$ by $-\nu$, we arrive at
$$u(x+\tau\nu)=u(x)$$
for arbitrary vector $\nu$ with $\nu_N=0$, this means that $u(x)$ is independent of $x'$, hence $u(x)=u(x_N)$. This concludes our proof of Theorem \ref{thm333-1}.
\end{proof}


\end{document}